\documentclass[12pt]{amsart}
\usepackage[top=30truemm,bottom=30truemm,left=25truemm,right=25truemm]{geometry}
\usepackage{mathrsfs}
\usepackage{amsmath, amsthm, amssymb}
\usepackage{mathtools}
\usepackage{color}
\usepackage{aliascnt}
\usepackage{bm}
\usepackage{amsfonts}
\usepackage{dsfont}
\usepackage{amscd}
\usepackage{extarrows}
\usepackage{enumerate}
\usepackage[all]{xy}
\usepackage[colorlinks]{hyperref}
\usepackage[nameinlink, capitalize, noabbrev]{cleveref}
\definecolor{darkblue}{RGB}{0,0,139}

\hypersetup{
    pdfencoding=auto,
    colorlinks=true,
    linkcolor=darkblue,
    citecolor=darkblue,
    urlcolor=darkblue
}
\newtheorem{theorem}{Theorem}[section]
\newtheorem{definition}{Definition}[section]
\newtheorem{lemma}{Lemma}[section]
\newtheorem{corollary}{Corollary}[section]
\newtheorem{proposition}{Proposition}[section]
\newtheorem{remark}{Remark}[section]

\newtheorem{question}{Question}[section]
\newaliascnt{assumption}{theorem}
\newtheorem{assumption}[assumption]{Assumption}
\aliascntresetthe{assumption}
\crefname{assumption}{Assumption}{Assumptions}
\Crefname{assumption}{Assumption}{Assumptions}

\newcommand{\be}{\begin{equation}}
	\newcommand{\ee}{\end{equation}}
\newcommand{\bea}{\begin{eqnarray}}
	\newcommand{\eea}{\end{eqnarray}}
\newcommand{\ben}{\begin{eqnarray*}}
	\newcommand{\een}{\end{eqnarray*}}
\newcommand{\bt}{\begin{split}}
	\newcommand{\et}{\end{split}}
\newcommand{\bet}{\begin{equation}}

	\numberwithin{equation}{section}
\begin{document}
\title[Complex Hessian flow]
{The Cauchy-Dirichlet Problem for Complex Hessian Flows: From A Priori Estimates to Pluripotential Theory}

\author[H. Sun]{Haoyuan Sun}
\address{Haoyuan Sun: School of Mathematical Sciences\\ Beijing Normal University\\ Beijing 100875\\ P. R. China}
\email{202531130037@mail.bnu.edu.cn}

\begin{abstract}
We study the Cauchy--Dirichlet problem for parabolic complex Hessian equations
on Hermitian manifolds and on bounded strictly $m$-pseudoconvex domains. In
the smooth setting, we prove global existence and uniqueness of classical
solutions under the presence of an admissible parabolic subsolution, by
establishing a priori estimates up to the parabolic boundary. The estimates
combine parabolic boundary techniques for complex Hessian equations with
interior second order estimates and a blow-up argument.

We then develop a general pluripotential framework for degenerate right-hand sides with
$L^p$ densities, \(p>n/m\), and bounded Cauchy--Dirichlet data. Since the
usual automorphism and Walsh-type arguments do not directly apply in a
variable Hermitian background, we use approximation by smooth data, balayage,
parabolic Perron envelopes, and a continuous obstacle approximation based on
Harvey--Lawson--Pli\'s subequation theory. The resulting solution is
continuous for positive time, locally uniformly Lipschitz and semi-concave in
time, and continuous up to the initial slice when the initial datum is
continuous. We also prove a parabolic comparison principle via time
regularization, Riemann sum approximations, and mixed Hessian inequalities.
\end{abstract}


\maketitle

\section{Introduction}

Complex Monge--Amp\`ere equations have played a central role in complex
geometry since Yau's solution of the Calabi conjecture \cite{Yau78}. The
weak theory was initiated by Bedford--Taylor \cite{BT76,BT82}, and later
developed in depth by Ko\l odziej, whose \(L^p\)-theory for the Dirichlet
problem \cite{Ko96,Ko98} and stability results on compact K\"ahler manifolds
\cite{Ko03} became fundamental tools in pluripotential theory. We refer to
\cite{Ko05,GZ17} for a systematic account of degenerate complex Monge--Amp\`ere
equations and their applications to K\"ahler geometry.

The parabolic counterpart is closely related to K\"ahler--Ricci type flows,
which go back to Cao's work on deforming K\"ahler metrics to
K\"ahler--Einstein metrics \cite{Cao85}. Degenerate complex Monge--Amp\`ere
flows have been studied from both viscosity and pluripotential viewpoints:
Eyssidieux--Guedj--Zeriahi developed viscosity techniques for degenerate
complex Monge--Amp\`ere equations and flows \cite{EGZ11,EGZ15,EGZ18}, see also \cite{Do16,Do17a,Do17b,DLTô20}. While
Guedj--Lu--Zeriahi established a powerful pluripotential Cauchy--Dirichlet theory
for complex Monge--Amp\`ere flows by means of parabolic Perron envelopes
\cite{GLZ21a}. This theory has since been further developed in several
directions, including stability and singular K\"ahler--Ricci type flows
\cite{GLZ18,GLZ20,Dang22}. More recently, Kang \cite{Kang25a, Kang25b, Kang26} considered Monge--Amp\`ere flows with right-hand side measures dominated by capacity.

The corresponding theory for complex Hessian equations is subtler. The
elliptic pluripotential theory of \(m\)-subharmonic functions was initiated
locally by B\l ocki \cite{Bło05}, and further developed by
Dinew--Ko\l odziej \cite{DK14}, Lu \cite{Lu13}, Lu--Nguyen \cite{LN15},
and Ko\l odziej--Nguyen \cite{KN26}, among others. On Hermitian manifolds
and domains with boundary, the elliptic Hessian theory has been advanced in
\cite{GN18,GL25,KN26,PSWZ25}. On the parabolic side, a viscosity approach to
parabolic complex Hessian type equations was introduced by Do \cite{Do22},
but a pluripotential Cauchy--Dirichlet theory for bounded Hessian potentials
does not seem to have been systematically developed. The goal of the present
paper is to fill this gap.

We first treat the smooth Cauchy--Dirichlet problem. The proof is based on
the classical boundary strategy of Caffarelli--Nirenberg--Spruck
\cite{CNS85} and S.-Y. Li \cite{Li04}, together with the fundamental complex Hessian
second order estimate of Hou--Ma--Wu \cite{HMW10}. Further developments of
the complex Hessian and fully nonlinear Hermitian estimates include
Guan--Sun \cite{GS15}, Phong--Picard--Zhang \cite{PPZ16},
Sz\'ekelyhidi \cite{Szé18}, Guan--Shi--Sui \cite{GSS15}, Collins--Picard
\cite{CP22}, and Phong--T\^o \cite{PTô21}. In particular, the boundary
estimates in \cite{CP22}, with their gradient-scale second order bound and
blow-up argument, are especially relevant to our setting.
Let $(X,\omega)$ be a compact Hermitian manifold with smooth boundary and let $\chi\in\Gamma_m(\omega)$ be a smooth $(\omega,m)$-positive form on $\overline{X}$. Let $X_T := X \times (0,T)$ be the parabolic cylinder, and denote its parabolic boundary by $\partial_0 X_T := \partial_b X_T \cup \partial_s X_T$, where $\partial_b X_T := \overline{X} \times \{0\}$ is the bottom boundary and $\partial_s X_T := \partial X \times [0,T)$ is the lateral (side) boundary. We consider the following complex Hessian flow:

\begin{equation}\label{introduction:main flow}
\begin{cases}
    \partial_tu=\log\frac{(\chi+dd^cu)^m\wedge\omega^{n-m}}{\omega^n}-\psi\quad \operatorname{in} X_T,\\
    u=\varphi\quad\operatorname{on}\partial_0X_T,
    \end{cases}
\end{equation}
where $0<T\leq\infty$. Our first main result establishes the existence of smooth solutions when the given data and the manifold are sufficiently regular:
\begin{theorem}\label{intro:main_smooth_solution}
Let $(X,\omega)$ be a compact Hermitian manifold with smooth boundary and let $\chi\in\Gamma_m(\omega)$ be a strictly $(\omega,m)$-positive form on $X$. Let $\psi=\psi(x,t)$ be a smooth function on $\overline{X_T}$, and $\varphi$ be a smooth boundary data satisfying the compatibility conditions up to order $k\geq0$ at the corner $\partial X\times\{0\}$ (see \cref{section:smooth data} for concrete definitions) such that $\varphi_0:=\varphi|_{\overline{X}\times\{0\}}\in C^\infty(\overline{X})$ is strictly $(\chi,\omega,m)$-subharmonic. Suppose there exists an admissible parabolic subsolution $\underline{u}\in C^\infty(\overline{X}\times(0,T])\cap C^{2,1}(\overline{X_T})$ for the flow \eqref{introduction:main flow}, that is to say, for each $T>0$, we have
    \begin{equation}\label{eq:subsolution}
        \begin{cases}
\partial_t\underline{u}\leq\log\frac{(\chi+dd^c\underline{u})^m\wedge\omega^{n-m}}{\omega^n}-\psi\quad \operatorname{in} X_T,\\    
\underline{u}\leq\varphi\quad\operatorname{on}\partial_bX_T \quad \operatorname{and}\quad \underline{u}=\varphi \quad\operatorname{on}\partial_sX_T.
        \end{cases}
    \end{equation}
    Here we say that $\underline{u}$ is admissible if for each time $t>0$ the slice $\underline{u}_t$ is strictly $(\chi,\omega,m)$- subharmonic. Then, there is a unique solution $u\in C^\infty(X\times(0,T])\cap C^{2k+\alpha, k+\frac{\alpha}{2}}(\overline{X_T})$ of \eqref{introduction:main flow}, where $\alpha\in(0,1)$ is a positive number.
\end{theorem}
As an immediate corollary, we derive smooth solutions on compact $m$-pseudoconvex Hermitian manifolds $(M,\omega)$ with smooth boundary. Here, we say that $M$ is $m$-pseudoconvex if there is $\rho\in C^\infty(\overline{M})$ such that $\rho$ is strictly $(\omega,m)$-subharmonic on $\overline{M}$ and 
$$
M=\{\rho<0\},\quad\partial M=\{\rho=0\},\quad d\rho|_{\partial M}\neq0.
$$
\begin{corollary}\label{intro:main_smooth_solution 2}
  Let $(M,\omega)$ be a compact $m$-pseudoconvex Hermitian manifold with smooth boundary equipped with a Hermitian metric $\omega$. Let $\psi=\psi(x,t)$ be a smooth function on $\overline{M_T}$, and $\varphi$ be a smooth boundary data satisfying the compatibility conditions up to order $k\geq0$ at the corner $\partial M\times\{0\}$ such that $\varphi_0:=\varphi|_{\overline{M}\times\{0\}}\in C^\infty(\overline{M})$ is strictly $(\omega,m)$-subharmonic. Then, there is a unique solution $u\in C^\infty(M\times(0,T])\cap C^{2k+\alpha, k+\frac{\alpha}{2}}(\overline{M_T})$ of the flow
\begin{equation}\label{introduction:main flow 2}
\begin{cases}
    \partial_tu=\log\frac{(dd^cu)^m\wedge\omega^{n-m}}{\omega^n}-\psi\quad \operatorname{in} M_T,\\
    u=\varphi\quad\operatorname{on}\partial_0M_T,
    \end{cases}
\end{equation}
where $\alpha\in(0,1)$ is a positive number.
\end{corollary}
\cref{intro:main_smooth_solution} also apply when
\(m=n\), and hence give a Cauchy--Dirichlet existence result for the parabolic
complex Monge--Amp\`ere equation in the same Hermitian boundary setting. We do
not rely on a separate gradient estimate. Instead, following the philosophy of
\cite{HMW10,DK17,Szé18,PTô21,CP22}, we establish the boundary second order
estimate at the gradient scale
\[
\sup_{X_T} |\partial\bar{\partial}u|_\omega
\leq C\left(1+\sup_{X_T}|\nabla u|_\omega^2\right).
\]
The parabolic feature of the argument is that the boundary barriers have to be
constructed for the linearized operator on space-time cylinders, with constants
uniform in the time variable and compatible with both the lateral and initial
parts of the parabolic boundary. This gives enough control to combine the
boundary estimate with the interior second order estimate and then run a
blow-up argument.

Compared with the elliptic Dirichlet problem, an additional ingredient is that
one first needs a uniform bound for \(\partial_tu\). This bound controls the
right-hand side \(e^{\partial_tu+\psi}\) and keeps the spatial Hessian operator
uniformly elliptic in the admissible cone. After this step, the tangential,
mixed tangential-normal, and double-normal boundary estimates can be adapted
from the elliptic barrier method, but the maximum principle is applied on the
parabolic cylinder rather than on a fixed time slice. Finally, if the gradient
were unbounded, the gradient-scale \(C^2\) estimate allows a rescaling limit;
the right-hand side disappears in the limit, and the Liouville theorem for
\(m\)-subharmonic functions due to Dinew--Ko\l odziej \cite{DK17} gives the
contradiction. Once the uniform \(C^2\) and \(\partial_tu\) estimates are
available, the Evans--Krylov theory and parabolic Schauder estimates yield the
higher regularity.
\vspace{2ex}

With the smooth solutions at hand, we turn to the second part of our paper: searching for general pluripotential solutions on a  bounded $m$-pseudoconvex domain $\Omega$ equipped with a Hermitian metric $\omega$ in $\mathbb{C}^n$. Here, we say that $\Omega$ is $m$-pseudoconvex if it is an $m$-pseudoconvex manifold as above.

\begin{assumption}[Pluripotential data] \label{ass:pluripotential-data}

Throughout the pluripotential part of the paper, unless otherwise stated, we assume that the data \((h,\psi,g)\) satisfy the following conditions.

\begin{enumerate}[(1)]
    \item \textbf{The Cauchy-Dirichlet boundary data $h$:} We say that $h: \partial_0 \Omega_T \to \mathbb{R}$ is Pluripotential Cauchy-Dirichlet boundary data if it is a bounded and upper semi-continuous function on the parabolic boundary $\partial_0 \Omega_T := (\partial \Omega \times [0,T)) \cup (\overline{\Omega} \times \{0\})$, which satisfies:
    \begin{itemize}
        \item $h$ is continuous on the lateral boundary $\partial \Omega \times [0, T)$.
        \item \textbf{Locally uniformly Lipschitz in $t$:} For any $0 < S < T$, there is a constant $C(S) > 0$ such that
        $$ t|\partial_t h(z, t)| \le C(S), \quad \forall (z, t) \in \partial \Omega \times (0, S]. $$
        \item \textbf{Locally uniformly semi-concave in $t$:} For any $0 < S < T$, there is a constant $C(S) > 0$ such that
        $$ t^2 \partial_t^2 h(z, t) \le C(S), \quad \forall (z, t) \in \partial \Omega \times (0, S]. $$
        \item \textbf{Compatibility:} The initial data $h_0 := h(\cdot, 0)$ is a bounded $(\omega, m)$-subharmonic function on $\overline{\Omega}$, and $\lim_{X \ni z \to \zeta} h_0(z) = h(\zeta, 0)$ for all $\zeta \in \partial \Omega$.
    \end{itemize}
    \vspace{1ex}
    \item \textbf{The twist term $\psi$:} $\psi(t, z)$ is a bounded function which is continuous on $\Omega\times[0,T)$. Moreover, it is uniformly Lipschitz and uniformly semi-convex in $t$. Namely, for each compact interval $J\Subset[0,T)$, there is a constant $C=C(J)$, such that $\psi(t,z)+Ct^2 $ is convex on $J$.
    \vspace{1ex}
    \item \textbf{The density $g$:} $0 \le g \in L^p(\Omega)$ for some $p > \frac{n}{m}$, and the zero set $\{z \in \Omega : g(z)=0\}$ has zero Lebesgue measure.
\end{enumerate}
\end{assumption}

\vspace{2ex}

Before stating our second main result, we must make sense of the degenerate parabolic complex Hessian equation on $\Omega_T := \Omega \times (0,T)$. Following the pluripotential approach pioneered by Guedj, Lu, and Zeriahi \cite{GLZ21a} for the complex Monge--Amp\`ere flow, we interpret the flow in the weak sense of Radon measures:
\begin{equation}\label{intro:weak_flow}
    dt \wedge (dd^c u)^m \wedge \omega^{n-m} = e^{\partial_t u + \psi(t,z)} g(z) dt \wedge dV,
\end{equation}
where $dV$ denotes the standard volume form on $\Omega$. For a parabolic $m$-potential $u \in P_m(\Omega_T) \cap L_{loc}^\infty(\Omega_T)$, the time derivative $\partial_t u$ exists almost everywhere, and the spatial complex Hessian measure $(dd^c u_t)^m \wedge \omega^{n-m}$ is well-defined for each time slice by the elliptic pluripotential theory of $m$-subharmonic functions, which was initiated by B\l ocki \cite{Bło05} in the local setting and Dinew--Ko\l odziej--Lu--Nguyen \cite{Lu13, DK14, LN15} on K\"ahler manifolds, and recently systematically developed by Ko\l odziej--Nguyen on Hermitian manifolds \cite{KN26}. For more information about recent progress of complex Hessian equations on Hermitian manifolds, we refer the reader to \cite{KN16,GL25, Fang25, KN25, PSWZ25, Sun25} and references therein.

The Cauchy-Dirichlet problem for \eqref{intro:weak_flow} consists in finding a parabolic $m$-potential $u \in P_m(\Omega_T) \cap L^\infty(\Omega_T)$ that satisfies \eqref{intro:weak_flow} in the pluripotential sense, while matching the Cauchy-Dirichlet boundary data $h$ on the parabolic boundary $\partial_0 \Omega_T$. Following the standard convention, the boundary value matching is understood as follows: for the lateral boundary, we require
\begin{equation}\label{intro:lateral_match}
    \lim_{\Omega_T \ni (t,z) \to (\tau, \zeta)} u(t,z) = h(\tau, \zeta), \quad \forall (\tau, \zeta) \in [0, T) \times \partial \Omega,
\end{equation}
and for the initial Cauchy data, we require the slices $u_t := u(t, \cdot)$ to converge in $L^1$, specifically:
\begin{equation}\label{intro:initial_match}
    \lim_{t \to 0^+} u_t = h_0 \quad \text{in } L^1(\Omega).
\end{equation}
In this case we say that \textit{$u$ is a pluripotential solution to the Cauchy-Dirichlet problem for the equation \eqref{intro:weak_flow} with boundary values $h$}.
\vspace{2ex}

Our second main result, which is the core of this paper, establishes the existence, uniqueness, and regularity of the pluripotential solution under much weaker assumptions.

\begin{theorem}[Global Pluripotential Solutions]\label{thm:main_pluripotential}
Let $(\Omega, \omega)$ be a bounded, strictly $m$-pseudoconvex domain with smooth boundary equipped with a Hermitian metric. Assume $0<T \leq +\infty$. Assume that the data \((h,\psi,g)\) satisfy \cref{ass:pluripotential-data} on every finite sub-cylinder \(\Omega_S:=\Omega\times(0,S)\), \(S<T\).
Then the parabolic Perron envelope $U := U_{h, g, \psi, \Omega_T}$ of all pluripotential subsolutions is the unique pluripotential solution to the Cauchy-Dirichlet problem for the parabolic complex Hessian equation \eqref{intro:weak_flow}.

Moreover, the solution $U$ is continuous on $(0, T) \times \overline{\Omega}$, and is locally uniformly Lipschitz and locally uniformly semi-concave in time $t \in (0, T)$. In particular, if the initial data $h_0$ is continuous on $\overline{\Omega}$, then $U$ is continuous globally on $[0, T) \times \overline{\Omega}$, and assumes the initial and lateral boundary data continuously.
\end{theorem}
Our strategy is inspired by the parabolic Perron method of
Guedj--Lu--Zeriahi \cite{GLZ21a}, but several Hessian-specific difficulties
have to be overcome. First, the automorphism method used originally by
Bedford--Taylor \cite{BT76} on the unit ball, and later in the parabolic
Monge--Amp\`ere setting, does not preserve the complex Hessian structure on a
general Hermitian background. Thus the $C^{1,1}$ regularity of the Perron
envelope cannot be obtained by the same argument. We instead construct
continuous weak solutions on balls by approximating the data and using the
smooth Cauchy--Dirichlet theory developed in the first part of the paper. The construction near the corner $\{0\}\times\partial\Omega$ and the regularity of the solution is more delicate.
Second, one has to be slightly careful with the continuity of envelopes in the
Hermitian Hessian setting. The usual Walsh translation argument
\cite{Wal69} applies in the flat background (namely, $\omega=dd^c|z|^2$), but after a spatial translation
a variable Hermitian form changes, so $(\omega,m)$-subharmonicity is not
preserved. We instead use the obstacle theorem for locally affinely
jet-equivalent convex subequations of Harvey--Lawson--Pli\'s
\cite{HLP16}, together with the Dirichlet duality framework
\cite{HL11} and the viscosity-distributional equivalence of
Harvey--Lawson \cite{HL13}, to obtain the continuous obstacle approximation
needed later. The continuity of the parabolic Perron envelope is then proved
by balayage, the comparison principle, and approximation by continuous
solutions.

This yields the parabolic Perron envelope as a pluripotential solution on
general strictly \(m\)-pseudoconvex domains. The solution is locally
uniformly Lipschitz and semi-concave in time, continuous for positive time,
and continuous up to the initial slice whenever the initial data are
continuous. 

Finally, to ensure the uniqueness of the pluripotential solutions and to complete the balayage argument, we establish a parabolic comparison principle.
\begin{theorem}[Parabolic Comparison Principle]\label{thm:intro_comparison}
Let $\Phi, \Psi \in P_m(\Omega_T) \cap L^\infty(\Omega_T)$ be locally uniformly semi-concave in $t \in (0, T]$. If $\Phi$ is a pluripotential subsolution and $\Psi$ is a pluripotential supersolution to the complex Hessian flow such that their boundary data satisfy $h_\Phi \le h_\Psi$ on $\partial_0 \Omega_T$, and $h_\Phi$ satisfies the locally uniform Lipschitz and semi-concave conditions, then $\Phi \le \Psi$ in $\Omega_T$.
\end{theorem}
A key technical innovation in the proof of \cref{thm:intro_comparison} lies in the verification of the subsolution property for the time-mollified function $\Phi^\varepsilon$. In the complex Monge-Amp\`ere setting, Guedj--Lu--Zeriahi \cite{GLZ21a} relied on a specific characterization of subsolutions from \cite{GLZ19}, which essentially linearizes the operator \cite[Section 6.3]{GLZ21a}. However, such a characterization is generally unknown for the complex Hessian operator $(dd^c u)^m \wedge \omega^{n-m}$ when $\omega$ is an arbitrary Hermitian metric (see \cref{question} below). When $\omega=dd^c|z|^2$, \cref{question} can be proved exactly as in \cite{GLZ19}. 

To overcome this, we exploit the locally uniform semi-concavity assumption of the subsolutions (which is fulfilled in the setting of \cref{thm:main_pluripotential}). This regularity allows us to approximate the time-convolution integral $\Phi^\varepsilon$ by finite Riemann sums. By invoking the mixed Hessian inequality (cf. \cite{DL15,Sun25}) for discrete convex combinations of $m$-subharmonic functions, we demonstrate that the subsolution property is preserved under mollification. 

\vspace{2ex}

\textbf{Acknowledgements.}
The author would like to thank his advisor, Professor Zhiwei Wang, as well as Professor Ngoc Cuong Nguyen, for suggesting this problem and for many helpful discussions. He also thanks Bowoo Kang for pointing out some typos in the first manuscript.

\section{Preliminaries and Definitions}
\subsection{$m$-positive forms}
Let $(X,\omega)$ be a compact Hermitian manifold of dimension $n$ with smooth boundary. For $1 \le m \le n$, we denote the $m$-th elementary symmetric polynomial of a vector $\lambda = (\lambda_1, \dots, \lambda_n) \in \mathbb{R}^n$ by
$$
\sigma_m(\lambda) = \sum_{1\le j_1 < j_2 < \dots < j_m \le n} \lambda_{j_1}\cdots \lambda_{j_m}.
$$
Following the standard convention, we define the G{\aa}rding cone $\Gamma_m \subset \mathbb{R}^n$ as the connected component of $\{\lambda \in \mathbb{R}^n : \sigma_m(\lambda) > 0\}$ containing the positive orthant, which can be explicitly written as
$$
\Gamma_m = \{\lambda \in \mathbb{R}^n : \sigma_1(\lambda) > 0, \dots, \sigma_m(\lambda) > 0\}.
$$

\begin{definition}
    A real $(1,1)$-form $\chi$ on $X$ is called strictly $(\omega, m)$-positive, denoted as $\chi \in \Gamma_m(\omega)$, if at each point $p \in X$, the eigenvalues of the Hermitian endomorphism with respect to $\omega$ (given locally by $g^{k\bar{l}}\chi_{j\bar{l}}$) lie in the cone $\Gamma_m$. 
\end{definition}

\begin{definition}
    Given a strictly $(\omega, m)$-positive form $\chi \in \Gamma_m(\omega)$, a function $u \in C^2(X)$ is called strictly $(\chi, \omega, m)$-subharmonic if the $(1,1)$-form $\chi + dd^c u$ is strictly $(\omega, m)$-positive. In this case, the complex Hessian operator is well-defined and elliptic. If $\chi = 0$, we simply call $u$ strictly $(\omega, m)$-subharmonic.
\end{definition}
\subsection{Parabolic H\"older Spaces}
To study the regularity of classical solutions to the parabolic complex Hessian flow, we recall standard definitions of parabolic H\"older spaces. Let $X_T := X \times (0, T)$. For any two points $P = (x,t) \in \overline{X_T}$ and $Q = (y,s) \in \overline{X_T}$, we define the parabolic distance as
$$d(P, Q) := \max \{ d_X(x, y), |t - s|^{1/2} \},$$
where $d_X$ is the distance function on $X$ induced by the background Hermitian metric $\omega$.

For a function $u: \overline{X_T} \to \mathbb{R}$ and an exponent $\alpha \in (0, 1)$, we define the parabolic H\"older seminorm as
$$[u]_{\alpha, \alpha/2; X_T} := \sup_{P \neq Q \in \overline{X_T}} \frac{|u(P) - u(Q)|}{d(P, Q)^\alpha}.$$
The space $C^{\alpha, \alpha/2}(\overline{X_T})$ consists of all continuous functions for which this seminorm is finite, equipped with the norm
$$\|u\|_{C^{\alpha, \alpha/2}(\overline{X_T})} := \sup_{X_T} |u| + [u]_{\alpha, \alpha/2; X_T}.$$

More generally, for any non-negative integer $m$, the space $C^{m+\alpha, \frac{m+\alpha}{2}}(\overline{X_T})$ consists of functions whose continuous partial derivatives $\partial_t^j \nabla_x^\beta u$ exist for all multi-indices $(j, \beta)$ satisfying $2j + |\beta| \le m$, and whose highest-order derivatives belong to $C^{\alpha, \alpha/2}(\overline{X_T})$. In this paper, we primarily work with the even-order spaces $C^{2k+\alpha, k+\alpha/2}(\overline{X_T})$. The corresponding norm is defined as
$$\|u\|_{C^{2k+\alpha, k+\alpha/2}(\overline{X_T})} := \sum_{2j + |\beta| \leq 2k} \|\partial_t^j \nabla_x^\beta u\|_{C^0(\overline{X_T})} + \sum_{2j + |\beta| = 2k} [\partial_t^j \nabla_x^\beta u]_{\alpha, \alpha/2; X_T}.$$

\section{The complex Hessian flow for smooth data}\label{section:smooth data}
In this section, let $(X,\omega)$ be a compact Hermitian manifold with smooth boundary. Let $T>0$ be a positive time. Let $X_T:=X\times[0,T)$ be the whole parabolic space and let $\partial_0X_T:=\partial_bX_T\cup \partial_sX_T$ be the parabolic boundary of $X_T$, where $\partial_bX_T:=X\times\{0\}$ is the bottom and $\partial_sX_T:=\partial X\times[0,T]$ is the side. Without loss of generality, we will assume that $0<T<\infty$ and do a priori estimates on $(0,T)$. We shall study the following Cauchy-Dirichlet problem of the complex Hessian flow:
\begin{equation}\label{eq:main flow}
\begin{cases}
    \partial_tu=\log\frac{(\chi+dd^cu)^m\wedge\omega^{n-m}}{\omega^n}-\psi\quad \operatorname{in} X_T,\\
    u=\varphi\quad\operatorname{on}\partial_0X_T,
\end{cases}
\end{equation}
where $\psi=\psi(x,t)$ is a smooth function on $\overline{X_T}$, $\chi\in\Gamma_m(\omega)$ is a strictly $(\omega,m)$-positive form on $\overline{X}$ and $\varphi$ is a smooth boundary data such that $\varphi_0:=\varphi|_{\overline{X}\times\{0\}}\in C^\infty(\overline{X})$ is strictly $(\chi,\omega,m)$-subharmonic and $\varphi|_{\partial X\times[0,T]}\in C^\infty(\partial X\times[0,T])$. 

To ensure sufficient regularity of the solution up to the corner $\partial X \times \{0\}$, we need to impose compatibility conditions between the initial and boundary data. First, we define the initial speed $v_1 \in C^\infty(\overline{X})$ determined by the initial data $\varphi_0$ as follows:
\begin{equation}
    v_1(z) := \log \frac{(\chi + dd^c \varphi_0)^m \wedge \omega^{n-m}}{\omega^n} - \psi(0,z).
\end{equation}
We also denote by $F^{i\bar{j}}(\varphi_0)$ the coefficients of the linearized operator of the complex Hessian equation evaluated at $\varphi_0$, given in local coordinates by $F^{i\bar{j}}(\varphi_0) = \frac{\partial}{\partial u_{i\bar{j}}} \log \frac{(\chi+dd^c u)^m\wedge\omega^{n-m}}{\omega^n} \Big|_{u=\varphi_0}$. 

The boundary data $\varphi$ is said to satisfy at least the $2$nd order compatibility conditions at the corner $\partial X\times\{0\}$ if we have the $0$-th order condition $\varphi(0,\cdot)=\varphi_0$ on $\partial X$, the $1$-st order condition
\begin{equation}\label{eq:first order compatible}
    \partial_t \varphi(0,z) = v_1(z),\quad \forall z \in \partial X,
\end{equation}
and the $2$-nd order condition obtained by differentiating the equation with respect to $t$:
\begin{equation}\label{eq:second order compatible}
    \partial_{tt} \varphi(0,z) = F^{i\bar{j}}(\varphi_0) \cdot (v_1)_{i\bar{j}}(z) - \partial_t \psi(0, z), \quad \forall z \in \partial X.
\end{equation}
For $k\geq2$, the definition of $k$th order compatibility is similar by differentiating the equation $k$ times.

By standard short-time existence theory for fully non-linear parabolic equations with corner compatibility conditions (cf. Lieberman \cite[Chapter 14]{Lie96}), we have the following theorem:

\begin{theorem}[Short-time Existence and Corner Regularity]\label{thm:short_time_existence}
Under the above assumptions, suppose the boundary data $\varphi$ satisfies the compatibility conditions up to order $k \ge 0$ at the corner $\partial X \times \{0\}$. Then there exists a maximal existence time $T_{max} \in (0, T]$ such that the Cauchy-Dirichlet problem \eqref{eq:main flow} admits a unique classical solution $u \in C^\infty(X \times (0, T_{max}))$. 

Moreover, the solution $u$ belongs to the parabolic H\"older space $C^{2k+\alpha, k+\alpha/2}(\overline{X_S})$ up to the boundary and the corner for any $S < T_{max}$ and some $0 < \alpha < 1$. In particular, the second-order compatibility ($k=2$) guarantees that $u \in C^{4+\alpha, 2+\alpha/2}(\overline{X_S})$, which provides sufficient regularity to justify all classical maximum principle arguments for up to second-order spatial and first-order time derivatives.
\end{theorem}
\begin{proof}
    The existence of a unique short-time solution $u \in C^{2k+\alpha, k+\alpha/2}(\overline{X_S})$ when $k=0,1$ for some maximal time $T_{max} > 0$ and any $S < T_{max}$ follows from the standard theory of fully non-linear parabolic equations with first-order corner compatibility conditions (cf. \cite[Theorem 14.18]{Lie96}). Our primary task is to establish the higher-order corner regularity, specifically $u \in C^{4+\alpha, 2+\alpha/2}(\overline{X_S})$ under the second-order compatibility condition ($k=2$), via a bootstrapping argument.

    Let $v := \partial_t u$. Differentiating the complex Hessian flow equation \eqref{eq:main flow} with respect to $t$, we observe that $v$ satisfies the linearized parabolic equation:
    \begin{equation}\label{eq:linearized_v}
        \partial_t v = F^{i\bar{j}}(A[u]) v_{i\bar{j}} - \partial_t \psi \quad \text{in } X_S,
    \end{equation}
    with boundary condition $v = \partial_t \varphi$ on $\partial_s X_S$ and initial condition $v(\cdot, 0) = v_1$ on $\overline{X}$.

    Since we already know from the base step that $u \in C^{2+\alpha, 1+\alpha/2}(\overline{X_S})$, the coefficients $F^{i\bar{j}}(A[u])$ of the linearized operator naturally belong to the H\"older space $C^{\alpha, \alpha/2}(\overline{X_S})$. The source term $\partial_t \psi$ is smooth, hence also in $C^{\alpha, \alpha/2}(\overline{X_S})$. Thus, equation \eqref{eq:linearized_v} is a linear, uniformly parabolic equation with $C^{\alpha, \alpha/2}$ coefficients.

    To obtain $v \in C^{2+\alpha, 1+\alpha/2}(\overline{X_S})$ from the linear Schauder theory (cf. \cite[Theorem 4.28]{Lie96}), the Cauchy-Dirichlet data for $v$ must satisfy the first-order compatibility condition at the corner $\partial X \times \{0\}$. That is to say:
    \begin{equation}
        \partial_t (\partial_t \varphi)(0,z) = F^{i\bar{j}}(A[\varphi_0])(v_1)_{i\bar{j}}(z) - \partial_t \psi(0,z), \quad \forall z \in \partial X.
    \end{equation}
    This is precisely the second-order compatibility condition \eqref{eq:second order compatible} assumed in the theorem. It follows that $v = \partial_t u \in C^{2+\alpha, 1+\alpha/2}(\overline{X_S})$.

    Finally, we return to the original nonlinear equation, which can be written as
    \[
    F(A[u]) = \partial_t u + \psi.
    \]
    Since we have already obtained $\partial_t u \in C^{2+\alpha, 1+\alpha/2}(\overline{X_S})$ and $\psi$ is smooth, the right-hand side belongs to $C^{2+\alpha, 1+\alpha/2}(\overline{X_S})$. 

    Thus, by the parabolic Evans--Krylov theory for concave uniformly parabolic equations in domains \cite{Evans82,Kry82,Kry83,Lie96}, and equivalently by the complex-geometric perturbative estimates of Tosatti--Wang--Weinkove--Yang and Chu \cite{TWWY15,Chu16}, we first obtain $u \in C^{2+\alpha,1+\alpha/2}(\overline{X_S})$ with quantitative control. Then, differentiating the equation in space and time and applying linear Schauder estimates to the resulting equations for derivatives of $u$, we can bootstrap the regularity. In particular, under the second-order compatibility condition, this yields
    \[
    u \in C^{4+\alpha, 2+\alpha/2}(\overline{X_S}).
    \]
    Iterating this argument gives higher-order regularity for arbitrary $k$.
\end{proof}
Our aim is to study the long-time existence of flow \eqref{eq:main flow}. 
The elliptic analog of \cref{intro:main_smooth_solution} was established in \cite{CP22}. Since the short-time existence \cref{thm:short_time_existence} is known, the main task is to establish a priori $C^2$ estimates. To facilitate the a priori estimates, we clarify the notations regarding the complex Hessian operator. Let $\omega=\sqrt{-1}g_{j\overline{k}}dz^j\wedge d\overline{z}^k$ and $A[u]$ be the Hermitian endomorphism of $T^{1,0}X$ defined by $h^k_j:=A[u]^k_j = g^{k\bar{m}}(\chi_{j\bar{m}} + u_{j\bar{m}})$, and let $\lambda[u] = (\lambda_1, \dots, \lambda_n)$ denote its eigenvalues. The operator is defined as $F(A[u]) = f(\lambda[u]) := \log\sigma_m(\lambda[u])$, where $\sigma_m$ is the $m$-th elementary symmetric polynomial.

The linearized operator of $F$ at $u$ is denoted by the symmetric tensor $F^{i\bar{j}} := \frac{\partial F}{\partial u_{i\bar{j}}}=\frac{\partial F}{\partial h^p_q}\frac{\partial h^{p}_q}{\partial u_{i\bar{j}}}=\frac{\partial F}{\partial h^p_i}g^{p\bar{j}}$. In a local coordinate system where $\omega_{i\bar{j}} = \delta_{ij}$ and the matrix $(\chi_{i\bar{j}} + u_{i\bar{j}})$ is diagonalized with entries $\lambda_i$, the tensor $F^{i\bar{j}}$ is diagonal with $F^{i\bar{i}} = f_i := \frac{\partial f}{\partial \lambda_i}$. 
We also define the crucial quantity 
\begin{equation}
\mathcal{F} := \operatorname{tr}_\omega(F^{i\bar{j}}) = \sum_{i=1}^n f_i,
\end{equation}
which represents the trace of the linearized operator. Furthermore, we denote by $\sigma_k(\lambda|i)$ the $k$-th elementary symmetric polynomial of the $(n-1)$-tuple $(\lambda_1, \dots, \hat{\lambda}_i, \dots, \lambda_n)$ where $\lambda_i$ is omitted. We recall that $f_i = \frac{1}{\sigma_m(\lambda)}  \sigma_{m-1}(\lambda|i)$, and the following standard properties for $f \in \Gamma_m$ hold: $f_i > 0$ and $f$ is concave.

Now, we can adjust our notations to fit in the framework of \cite{PTô21} by rewriting the flow \eqref{eq:main flow} as follows:
\begin{equation}\label{eq:main flow 2}
\begin{cases}    
\partial_tu=F(A[u])-\psi\quad \operatorname{in} X_T,\\    
u=\varphi\quad\operatorname{on}\partial_0X_T.
\end{cases}
\end{equation}
We first recall the following definition of parabolic $\mathcal{C}$-subsolution due to \cite{PTô21}:
\begin{definition}\label{parabolic C-subsolution}
    An admissible function $\underline{u}\in C^\infty(X_T)\cap C^{2,1}(\overline{X_T})$ is called a parabolic $\mathcal{C}$-subsolution of \eqref{eq:main flow}, if there exist constants $\delta,K>0$ such that for any $(z,t)\in X_T$, the condition
    \begin{equation}
        f(\lambda[\underline{u}(z,t)]+\mu)-\partial_t\underline{u}+\tau=\psi(t,z),\quad\mu+\delta I\in\Gamma_n,\quad\tau>-\delta
    \end{equation}
    implies that $|\mu|+|\tau|<K$. Here, $I:=(1,1,...,1)$.
\end{definition}
Although the twist term $\psi$ is allowed to depend on $t$ here, the arguments in \cite[Lemma 8]{PTô21} still remain valid:
\begin{lemma}\label{subsolution criterion}
    Suppose that $\|u\|_{C^{2,1}(X\times[0,T])}<\infty$ and $\psi(t,z)$ is bounded on $X_T$. Then $\underline{u}$ is a parabolic $\mathcal{C}$-subsolution if and only if there exists a constant $\tilde{\delta}>0$ independent of $(z,t)$ such that
    \begin{equation}
        f_\infty(\lambda[\underline{u}(z,t)])>\tilde{\delta}+\psi(z,t).
    \end{equation}
    Here, we define $f_\infty$ as
    $$
f_\infty(\lambda):=\underset{1\leq i\leq n}{\min}\underset{\mu\rightarrow\infty}{\lim}f(\lambda+\mu e_i).
    $$
\end{lemma}
In our setting, it is clear that $f_\infty(\lambda)\equiv\infty$ for any $\lambda\in\Gamma_m$, so \cref{subsolution criterion} immediately yields that a parabolic subsolution in the sense of \eqref{eq:subsolution} is a parabolic $\mathcal{C}$-subsolution in the sense of \cref{parabolic C-subsolution}.

\subsection{The $\partial_tu$ estimate}

\begin{proposition}\label{partial t estimate}
    There exists a constant $C$ depending only on $T, \|\partial_t \psi\|_{L^\infty(X_T)}, \|\partial_t \varphi\|_{L^\infty(\partial_s X_T)}$, and $\|F(A[\varphi_0])\|_{L^\infty(X)}$ (which in turn depends on $\|\varphi_0\|_{C^2(\overline{X})}$ and its strict $(\chi,\omega,m)$-subharmonicity), such that:
    \begin{equation}
        |\partial_t u| \le C \quad \text{in } X_T.
    \end{equation}
\end{proposition}
\begin{proof}
    Differentiating the complex Hessian flow \eqref{eq:main flow 2} with respect to $t$, we observe that $v := \partial_t u$ satisfies the following linearized parabolic equation:
    $$
    \partial_t v = \partial_t F(A[u]) - \partial_t\psi = F^{i\bar{j}} \partial_t(u_{i\bar{j}}) - \partial_t\psi = F^{i\bar{j}} v_{i\bar{j}} - \partial_t\psi.
    $$
    Define the linear uniformly parabolic operator $L := \partial_t - F^{i\bar{j}} \partial_i \partial_{\bar{j}}$. Let $B := \sup_{X_T} |\partial_t \psi|$. It follows that:
    \begin{equation*}
        L(v - Bt) = -\partial_t \psi - B \le 0 \quad \text{and} \quad L(v + Bt) = -\partial_t \psi + B \ge 0.
    \end{equation*}
    By the parabolic weak maximum principle, the maximum of subsolutions and the minimum of supersolutions must be attained on the parabolic boundary $\partial_0 X_T$. Therefore, we have:
    \begin{equation*}
        \sup_{X_T} (v - Bt) \le \sup_{\partial_0 X_T} (v - Bt) \quad \text{and} \quad \inf_{X_T} (v + Bt) \ge \inf_{\partial_0 X_T} (v + Bt).
    \end{equation*}
    Since $0 \leq t \leq T$, this immediately yields the global bounds in $X_T$:
    \begin{equation}\label{eq:v_bound_intermediate}
        \inf_{\partial_0 X_T} v - BT \le v(z, t) \le \sup_{\partial_0 X_T} v + BT \quad \text{in } X_T.
    \end{equation}
  because $\partial_tu=\partial_t\varphi$ on $\partial_sX_T$ and $\partial_tu=F(A[u(0,z)])-\psi(0,z)=F(A[\varphi_0(z)])-\psi(0,z)$ on $\partial_bX_T$ are both well dominated. This gives the uniform bound of $v$.
\end{proof}
\begin{remark}
    As noted in \cite[Remark 1.3]{GSS15}, to use the maximum principle, we must at least ensure that $\partial_tu$ is continuous on the parabolic boundary, which means that we have to assume the first-order compatibility of the Cauchy-Dirichlet boundary data to get $u\in C^{2,1}(\overline{X_T})$. However, by the short-time existence, we automatically have a smooth solution $u$ in a short time $[0,\varepsilon)$. We may then start the Hessian flow from $\frac{\varepsilon}{2}$.
\end{remark}

\subsection{The $C^0$ estimate}

\begin{proposition}\label{prop:C0 estimate}
   Let $u$ be the smooth solution to \eqref{eq:main flow}. There exists a constant $C$ depending only on 
    $\|\underline{u}\|_{L^\infty(X_T)}$, $\|\varphi\|_{L^\infty(\partial_0 X_T)}$, $\|\operatorname{tr}_\omega \chi\|_{L^\infty(X)}$, 
    and the geometry of $(X, \omega)$, such that
    \begin{equation}\label{C^0 estimate}
        \sup_{X_T}|u|\leq C.
    \end{equation}
\end{proposition}

\begin{proof}
    As in \cite{GSS15}, let $b \in C^\infty(\overline{X_T})$ be the solution of the following linear elliptic Dirichlet problem, solved slice-wise for each $t \in [0, T]$:
    \begin{equation}\label{construction of b}
        \Delta_\omega b + \operatorname{tr}_\omega\chi = 0 \quad \operatorname{in} X, \quad b = \varphi \quad \operatorname{on} \partial X.
    \end{equation}
    The existence and smoothness of $b$ are guaranteed by standard linear elliptic theory, since $\varphi$ is smooth in both space and time. We claim that $\underline{u} \le u \le b$ on $\overline{X_T}$.

    For the lower bound, since $\underline{u}$ is an admissible parabolic subsolution, we have $\partial_t \underline{u} - F(A[\underline{u}]) \le -\psi = \partial_t u - F(A[u])$ in $X_T$. On the parabolic boundary $\partial_0 X_T$, the definition requires $\underline{u} \le u$. The parabolic comparison principle immediately yields $\underline{u} \le u$ in $\overline{X_T}$.

    For the upper bound, we can check it slice-wise. Fix a time $t\in[0,T)$, it follows from the definition of $b$ that the eigenvalues of $g^{j\overline{k}}(\chi_{j\overline{k}}+b_{j\overline{k}})$ lie outside $\{\lambda,\sigma_m(\lambda)\geq e^{\partial_tu+\psi}\}$ due to \cref{partial t estimate}. The comparison principle \cite[Lemma 3.1]{CP22} (see also \cite{CNS85}) then yields the desired estimate
\[
\sup_{X_T}|u|\leq C
\]
and we have concluded the proof.
\end{proof}

\subsection{The boundary gradient estimates}

Fix a point $p\in\partial X\times\{t_0\}$ and choose a coordinate chart of $X$ centered at $p$. Let $s_\alpha$ be a tangent vector and $x_n$ be the unit inner normal vector at $p$. We have
$$
\partial_{s_\alpha}(u-\underline{u})(p)=0,\quad\partial_{x_n}(u-\underline{u})(p)\geq0\quad\operatorname{and}\quad \partial_{x_n}(b-u)(p)\geq0,
$$
which immediately implies the desired estimate
\begin{equation}\label{eq:boundary gradient estimate}
    \sup_{\partial_sX_T}\|\nabla_\omega u\|\leq C,
\end{equation}
where $C$ depends on $\|\nabla \underline{u}\|_{L^\infty(\partial_s X_T)},\|\nabla \varphi\|_{L^\infty(\partial_0 X_T)},\|\operatorname{tr}_\omega \chi\|_{L^\infty(X)}$, and the geometry of $\partial X$.

\subsection{Interior estimates}

In this subsection, we establish the interior $C^2$ estimates. This is a parabolic analogue of \cite[Proposition 3.3]{CP22} and adapts the techniques developed for closed manifolds by Phong-T\^o \cite{PTô21} to the Cauchy-Dirichlet setting.

\begin{proposition}\label{prop: interior estimate}
    Let $u$ be the smooth solution to \eqref{eq:main flow}. There exists a constant $C$ depending on the geometry of $(X, \omega)$, $\|\chi\|_{C^2}$, $\|\psi\|_{C^{2,1}(X_T)}$, $\|\underline{u}\|_{C^{2,1}(X_T)}$, $\|u\|_{L^\infty(X_T)}$, and $\|\partial_t u\|_{L^\infty(X_T)}$, such that
    \begin{equation}\label{eq: interior estimates}
        \|\sqrt{-1}\partial\bar{\partial}u\|_{L^\infty(X_T)} \leq C\left(1 + \sup_{\partial_0X_T}|\sqrt{-1}\partial\bar{\partial}u| + \sup_{X_T}|\nabla_\omega u|^2\right).
    \end{equation}
\end{proposition}

To prove this proposition, the following key algebraic property of parabolic $C$-subsolutions (cf. \cite{Guan14, Szé18, PTô21}) is crucial:

\begin{lemma}\label{lem:subsolution_algebraic}
    Let $\underline{u}$ be a parabolic $\mathcal{C}$-subsolution of equation \eqref{eq:main flow}. There exists a uniform constant $\kappa > 0$ such that at any point $(z,t) \in X_T$, if the eigenvalues $\lambda = \lambda[u]$ satisfy $|\lambda[u] - \lambda[\underline{u}]| > K$ (where $K$ is the constant from the $\mathcal{C}$-subsolution definition), then either:
    \begin{equation}
        F^{p\bar{q}}(A[u])(\underline{u}_{p\bar{q}} - u_{p\bar{q}}) + (\partial_t u - \partial_t \underline{u}) \geq \kappa \mathcal{F},
    \end{equation}
    or we have for any $1 \leq i \leq n$:
    \begin{equation}
        F^{i\bar{i}}(A[u]) \geq \kappa \mathcal{F},
    \end{equation}
    where $\mathcal{F} := \sum_{i=1}^n \frac{\partial f}{\partial \lambda_i}(\lambda[u])$.
\end{lemma}

\begin{proof}[Proof of \cref{prop: interior estimate}]
    Set $v := u - \underline{u}$. Since we are dealing with the Cauchy-Dirichlet problem, $v$ is uniformly bounded globally by the $C^0$ estimates (\cref{prop:C0 estimate}). 
    
    Consider the following test function on $X_T$:
    $$
    G := \log\lambda_1 + \phi(|\nabla u|^2) + \varphi(v),
    $$
    where $\lambda_1$ is the largest eigenvalue of the endomorphism $A[u]$ (suitably perturbed to be distinct, as in \cite[Proposition 13]{Szé18}), and the functions $\phi, \varphi$ are defined as:
    $$
    \phi(s) := -\frac{1}{2}\log\left(1 - \frac{s}{2K_0}\right), \quad K_0 := \sup_{X_T}(1 + |\nabla u|^2), \quad \varphi(v) := D_1 e^{-D_2 v},
    $$
    for some large constants $D_1, D_2 > 0$ to be chosen.

    Suppose that $G$ attains its maximum at a point $(x_0, t_0)$. If $(x_0, t_0)$ lies on the parabolic boundary $\partial_0 X_T$, we immediately obtain:
    $$
    \sup_{X_T} G \leq \sup_{\partial_0 X_T} G,
    $$
    which yields the desired estimate \eqref{eq: interior estimates} by bounding $\lambda_1$ in terms of its boundary values, the gradient bound $K_0$, and the global $C^0$ bound of $v$.

    If $(x_0, t_0)$ lies in the interior of $X \times (0, T]$, we evaluate the linearized parabolic operator $\mathcal{L} := -\partial_t + F^{k\bar{k}}\nabla_k\nabla_{\bar{k}}$ at this maximum point. At $(x_0, t_0)$, we have $\nabla G = 0$ and $\mathcal{L} G \leq 0$. 

    The calculation of $\mathcal{L} G$ is purely local and follows word-for-word the computation in \cite[Lemma 2]{PTô21}. The spatial derivatives of the equation generate terms involving the curvature and torsion of $\omega$, the derivatives of $\chi$, and the spatial derivatives of $\psi$. Because $v$ is globally bounded independently of time, we can select $D_2 \gg 1$ and $D_1 \gg 1$ such that $\varphi'(v) < 0$ and $\varphi''(v) \gg |\varphi'(v)|$. 

    Dividing into two cases based on the eigenvalues (either $\theta \lambda_1 \leq -\lambda_n$ or $\theta \lambda_1 > -\lambda_n$ for a small $\theta > 0$) and invoking \cref{lem:subsolution_algebraic} in the latter case to control the positive sum of $F^{i\bar{i}}$, we derive:
    $$
    0 \geq \mathcal{L}G \geq c_1 \frac{\lambda_1^2}{K_0} - c_2,
    $$
    where $c_1, c_2$ are uniform positive constants depending on the quantities specified in the proposition statement. This forces:
    $$
    \lambda_1(x_0, t_0) \leq C K_0^{\frac{1}{2}} \leq C K_0.
    $$
This implies the global bound \eqref{eq: interior estimates} across $X_T$.
\end{proof}

Since $u=\varphi_0$ on $\partial_bX_T$, the remaining task is now to establish the $\partial\bar{\partial}$-estimates of $u$ on $\partial_sX_T$, as will be carried out in the following three subsections.

\subsection{Boundary double tangential estimates}\label{section:double tangential}
Fix a positive time $t_0>0$ and let $p\in\partial X\times\{t_0\}$ be a point on the side, we will choose coordinates as in \cite{CP22}. Namely, choose coordinates $z=(z_1,...,z_n)$ such that $p$ corresponds to the origin and $\omega(0)=Id$. Fix also a defining function $\rho$ of a boundary chart $\Omega$ centered at $p$ satisfying
\begin{equation}
   ( \partial X\times\{t_0\})\cap\Omega=\{\rho=0\},\,\Omega\subset\{\rho\leq0\},\quad d\rho\neq0\operatorname{on}\partial X\times\{t_0\}.
\end{equation}
Let $T^{1,0}\partial X$ be the holomorphic tangent subbundle of $T^{1,0}X$, so we clearly have the following formula
$$
T_p^{1,0}\partial X=T_p\partial X\cap JT_p\partial X=\{V\in T^{1,0}_pX:V(\rho)=0\}.
$$
Up to an orthogonal rotation we can assume furthermore that
$$
T_p^{1,0}\partial X=\operatorname{Span}\{\partial_{z_1},...,\partial_{z_{n-1}}\},
$$
and that $x_n$ is in the direction of the unit inner normal vector at $0$. Replacing $\rho$ by $\frac{\rho}{-\rho_{x_n}(0)}$ we can derive the following Taylor expansion of $\rho$ at $0$:
\begin{equation}\label{Taylor of rho}
    \rho=-x_n+O(|z|^2).
\end{equation}
Write $z_i=x_i+\sqrt{-1}y_i$ and set
$$
s_\alpha=y_\alpha,\,1\leq\alpha\leq n,\quad s_{n+\beta}=x_\beta,\,1\leq\beta\leq n-1.
$$
The implicit function theorem yields a function $\zeta(s)$ such that $\rho(s,\zeta(s))=0$. We can deduce easily that the boundary $\partial X\cap\Omega$ is defined by the equation $x_n=\zeta(s)$ from the definition of $\rho$ and hence $(u_{t_0}-\underline{u}_{t_0})(s,\zeta(s))=0$ because $u=\underline{u}$ on $\partial_sX_T$. Here, we use the notation $u_t$ to represent the slice $u(t,\cdot)\in\operatorname{SH}_m(X,\chi,\omega)$. Now, differentiating both equations immediately gives $\rho_{\alpha}(0)=\zeta_{\alpha}(0)$, $\rho_{\alpha\beta}(0)=\zeta_{\alpha\beta}(0)$,
\begin{equation}
    (u_{t_0}-\underline{u}_{t_0})_\alpha=- (u_{t_0}-\underline{u}_{t_0})_{x_n}\zeta_{\alpha}
\end{equation}
and
\begin{equation}\label{second order derivatives}
     (u_{t_0}-\underline{u}_{t_0})_{\alpha\beta}(0)=- (u_{t_0}-\underline{u}_{t_0})_{x_n}(0)\rho_{\alpha\beta}(0).
\end{equation}
This yields that
\begin{equation}\label{eq:double tangential}
|(u_{t_0}-\underline{u}_{t_0})_{\alpha\beta}|\leq C.
\end{equation}
by the boundary gradient estimates \eqref{eq:boundary gradient estimate} with the uniform constant $C$ depend on $\|\rho\|_{C^2(\partial\Omega)}$, $\|\nabla u\|_{L^\infty(\partial_s X_T)}$ and $\|\nabla \underline{u}\|_{L^\infty(\partial_s X_T)}$. We summarize it as follows:

\begin{proposition}\label{prop:double tangential}
Let $u \in C^{2,1}(\overline{X_T})$ be an admissible solution to the flow \eqref{introduction:main flow} and $\underline{u}$ be an admissible subsolution. For any fixed $(p, t_0) \in \partial X \times (0, T]$, let $\rho$ be the defining function of the boundary near $p$ and choose coordinates as above. Then there exists a uniform constant $C > 0$ such that for any tangential directions $s_\alpha, s_\beta \in \{y_1, \dots, y_n, x_1, \dots, x_{n-1}\}$, we have:
\begin{equation}\label{eq:tangential bound}
|(u - \underline{u})_{s_\alpha s_\beta}(0, t_0)| \leq C.
\end{equation}
The constant $C$ depends only on the $C^2$ norm of $\rho$, the background data, and the uniform boundary gradient estimates.
\end{proposition}
\subsection{Boundary tangential-normal estimates}\label{section:tan normal}
The main goal of this section is to establish the following proposition:
\begin{proposition}\label{tan-normal estimates}
    In the setting of \cref{intro:main_smooth_solution}, there is a uniform
    constant \(C>0\), depending only on the geometry of \((X,\omega)\) and
    \(\partial X\), \(\|\chi\|_{C^2}\), \(\|\psi\|_{C^{2,1}(X_T)}\),
    \(\|\underline{u}\|_{C^{3,1}(X_T)}\), and
    \(\|\varphi\|_{C^4(\partial_0X_T)}\), such that
    \[
        |u_{i\overline n}|(0)\leq C K^{1/2}.
    \]
    Here, we use the boundary coordinates chosen in
    \cref{section:double tangential}, and
    \(K:=\sup_{X_T}(1+|\nabla_\omega u|^2)\).
\end{proposition}

Let $d(z):=d(z,\partial X)$ be the distance function to the boundary $\partial X$, we adopt the following barrier due to \cite{CNS85, GSS15, CP22}: for a large number $N>>1$ and $c_0>0$ to be determined, set
$$
v(t,z):=(u-\underline{u})(t,z)+c_0d(z)-Nd(z)^2.
$$
As before, we fix a positive time $t_0>0$ and consider a point $p$ lies in the boundary slice $\partial X\times\{t_0\}$. Set $\Omega_\delta:=\Omega\cap\{|z|<\delta\}$ for $\delta>0$ to be determined and $\Omega_{\delta,p}:=\Omega_\delta\times[0,t_0)$. In the sequel, we will always work on the slice $\Omega_\delta\times\{t_0\}$.

\begin{lemma}\label{lem:estimate of Lv}
    There exists $\delta,c_0,N,\tau>0$ such that the function $v:\Omega_{\delta,p}\rightarrow\mathbb{R}$ satisfies 
\begin{enumerate}
    \item $\mathcal{L}v:=-\partial_tv+F^{p\overline{q}}v_{p\overline{q}}\leq-\tau(1+\mathcal{F})$.
    \item $v\geq0$ on $\Omega_{\delta,p}$.
    \item $v\geq\frac{c_0}{2}d$ on $\Omega_\delta\times\{0\}$.
\end{enumerate}
\end{lemma}
\begin{proof}
The argument is a generalization of \cite[Lemma 4.2]{CP22} with some simplifications in the parabolic case. We first show the last two statements. The comparison principle yields $u\geq\underline{u}$. By a Taylor expansion, we deduce that $d(z)\leq A|z|$ for some constant $A>0$ in $\Omega_\delta$ if we take $\delta$ small enough. This implies that $v\geq d(c_0-Nd)\geq d(c_0-NA|z|)\geq d(c_0-NA\delta)\geq\frac{c_0}{2}d\geq0$ if we take further 
\begin{equation}\label{choice of delta}
        \delta<\frac{c_0}{2NA}.
\end{equation}
    Now we deal with the first statement. We work at a point $p\in\Omega_{\delta}\times\{t_0\}$ with coordinates such that $g_{j\bar{k}}=\delta_{jk}$ and $h_{j\bar{k}}=\lambda_j\delta_{j\bar{k}}$. We can calculate
    \begin{equation}
        F^{p\bar{q}}v_{p\bar{q}}=  F^{p\bar{q}}(u-\underline{u})_{p\bar{q}}+c_0  F^{p\bar{q}}d_{p\bar{q}}-2NF^{p\bar{q}}d_pd_{\bar{q}}-2NdF^{p\bar{q}}d_{p\bar{q}}.
    \end{equation}
    For the first term, note that $F^{1\bar{1}}\leq...\leq F^{n\bar{n}}$ (see e.g. equation (67) in \cite{Szé18}), we have by the Schur-Horn's Lemma
    $$
 F^{p\bar{q}}(u-\underline{u})_{p\bar{q}}=F^{p\bar{p}}(\lambda_p-(\chi_{p\bar{p}}+\underline{u}_{p\bar{p}}))\leq F^{p\bar{p}}\lambda_p-F^{p\bar{p}}\underline{\lambda}_p=-\tau\mathcal{F}+F^{p\bar{p}}(\lambda_p-(\underline{\lambda}_p-\tau)) ,
    $$
    for each $\tau>0$. Here, we use the notation $\underline{\lambda}_1\geq...\geq\underline{\lambda}_n$ to denote the eigenvalues of $\chi+\sqrt{-1}\partial\bar{\partial}\underline{u}$ with respect to $\omega$ at a point. Now we choose $\tau>0$ such that $\underline{\lambda}-\tau\textbf{1}\in\Gamma_m(\omega)$. Indeed, by the definition of subsolutions and the uniform bounds on $\|\partial_t\underline{u}\|_\infty,\|\psi\|_\infty$, we have
    \begin{equation}\label{subsolution l}
\chi_{\underline{u}}^m\wedge\omega^{n-m}\geq e^{\partial_t\underline{u}+\psi}\omega^n\geq C^{-1}\omega^n.
    \end{equation}
    G{\aa}rding's inequality yields furthermore that 
    $$
\chi_{\underline{u}}^k\wedge\omega^{n-k}\geq C^{-\frac{k}{m}}\omega^n
    $$
    for each $1\leq k\leq m$. As we have $(\chi_{\underline{u}}-\tau\textbf{1})^k\wedge\omega^{n-k}=\chi_{\underline{u}}^k\wedge\omega^{n-k}+O(\tau)$, the constant $\tau>0$ can be chosen uniformly. By the concavity of $f(\lambda)=\log\sigma_m(\lambda)$ on $\Gamma_m(\omega)$, we can write
    $$
F^{p\bar{p}}(\lambda)[\lambda_p-(\underline{\lambda}_p-\tau)]\leq F(\lambda)-F(\underline{\lambda}-\tau\textbf{1})=\partial_tu+\psi-F(\underline{\lambda}-\tau\textbf{1}).
    $$
    This implies that
    \begin{equation}\label{eq:estimate1}
 F^{p\bar{q}}(u-\underline{u})_{p\bar{q}}\leq-\tau\mathcal{F}+C.
    \end{equation}
    Next, since $d$ is smooth in $\Omega_\delta$ for $\delta>0$ small enough and $|\nabla d|=\frac{1}{2}$, we can estimate the remaining terms as follows:
    \begin{equation}\label{eq:estimate2}
        c_0  F^{p\bar{q}}d_{p\bar{q}}-2NdF^{p\bar{q}}d_{p\bar{q}}\leq c_0C\mathcal{F}+NdC\mathcal{F},
    \end{equation}
    and
    \begin{equation}\label{eq:estimate 3}
        -2NF^{p\bar{q}}d_pd_{\bar{q}}\leq-2NF^{1\bar{1}}|\partial d|^2=-\frac{N}{2}F^{1\bar{1}}.
    \end{equation}
    Combining \eqref{eq:estimate1}, \eqref{eq:estimate2}, \eqref{eq:estimate 3} we arrive at
    $$
F^{p\bar{q}}v_{p\bar{q}}\leq-\tau\mathcal{F}-\frac{N}{2}F^{1\bar{1}}+C(c_0+Nd)\mathcal{F}+C.
$$
 Since by our $\partial_tu$- estimate we have $|\partial_tv|=|\partial_t(u-\underline{u})|\leq C$, we can infer furthermore that
$$
\mathcal{L}v=-\partial_tv+F^{p\overline{q}}v_{p\overline{q}}\leq-\tau\mathcal{F}-\frac{N}{2}F^{1\bar{1}}+C(c_0+Nd)\mathcal{F}+C.
$$
Now, since $d(z) \leq A|z|$ for some constant $A>0$ in $\Omega_\delta$, we can choose $c_0$ and $\delta$ sufficiently small such that
$$
C(c_0 + NA\delta) \leq \frac{\tau}{4}.
$$
This allows us to absorb the positive $\mathcal{F}$ error term, yielding:
$$
\mathcal{L}v \leq -\frac{3\tau}{4}\mathcal{F} - \frac{N}{2}F^{1\bar{1}} + C.
$$
To control the remaining terms, we write:
$$
-\frac{\tau}{2}\mathcal{F} - \frac{N}{2}F^{1\bar{1}} = - \sum_{i=1}^n F^{i\bar{i}} q_i = - \frac{1}{\sigma_m(\lambda)} \sum_{i=1}^n \sigma_{m-1}(\lambda|i) q_i,
$$
where $q = \frac{1}{2}(\tau+N, \tau, \dots, \tau)$. By G\aa rding's inequality \cite{Går59}, we have:
$$
\sum_{i=1}^n \sigma_{m-1}(\lambda|i) q_i \geq m \sigma_m(\lambda)^{\frac{m-1}{m}} \sigma_m(q)^{\frac{1}{m}}.
$$
Dividing both sides by $\sigma_m(\lambda)$, we get:
$$
-\frac{\tau}{2}\mathcal{F} - \frac{N}{2}F^{1\bar{1}} \leq -m \sigma_m(\lambda)^{-\frac{1}{m}} \sigma_m(q)^{\frac{1}{m}}.
$$
Recall that $\sigma_m(\lambda) = e^{\partial_t u + \psi}$. Thanks to our uniform $\partial_t u$ estimate (\cref{partial t estimate}) and the boundedness of $\psi$, we have an upper bound $\sigma_m(\lambda) \leq C_1$ for some uniform constant $C_1 > 0$. Thus, $\sigma_m(\lambda)^{-\frac{1}{m}} \geq C_1^{-1/m} > 0$. 

On the other hand, a direct calculation of $\sigma_m(q)$ gives:
$$
\sigma_m(q) = \frac{1}{2^m} \left[ \tau^{m-1}(\tau+N)\binom{n-1}{m-1} + \tau^m\binom{n-1}{m} \right] \geq \frac{N \tau^{m-1}}{2^m}.
$$
Therefore, we obtain the estimate:
$$
-\frac{\tau}{2}\mathcal{F} - \frac{N}{2}F^{1\bar{1}} \leq - \frac{m}{2} C_1^{-1/m} \tau^{\frac{m-1}{m}} N^{\frac{1}{m}}.
$$
By choosing $N$ sufficiently large (depending on $m$, $\tau$, $C$, and $C_1$), we can ensure that:
$$
-\frac{\tau}{2}\mathcal{F} - \frac{N}{2}F^{1\bar{1}} \leq -C - \frac{\tau}{4}.
$$
Substituting this back into our estimate for $\mathcal{L}v$, we arrive at:
$$
\mathcal{L}v \leq -\frac{\tau}{4}\mathcal{F} - \frac{\tau}{4} = -\frac{\tau}{4}(1+\mathcal{F}).
$$
Renaming $\frac{\tau}{4}$ as our new $\tau$, we conclude that $\mathcal{L}v \leq -\tau(1+\mathcal{F})$ on $\Omega_{\delta,p}$. Finally, we remark that all the constants $\delta, c_0, N, \tau$ can be chosen uniformly and independently of the time $t$, owing to the short time existence and the fact that our coordinates depend continuously on $t$.
\end{proof}

For any index $\alpha\in\{1,2,...,2n-1\}$, define
$$
T_\alpha:=\partial_{s_\alpha}-\frac{\rho_{s_\alpha}}{\rho_{x_n}}\partial_{x_n},
$$
which are real vector fields on $\Omega_{\delta,p}$ satisfying $T_\alpha(\rho)=0$. As in \cite{CP22}, we use the notation $\mathcal{E}$ to denote terms which can be controlled by
$$
|\mathcal{E}|\leq C(1+K^{\frac{1}{2}})\mathcal{F}+C\sum_i\sigma_{m-1}(\lambda|i)\lambda_i+C,
$$
where $C$ is a uniform constant depending on $X,\omega,\underline{u},\chi$. We emphasize that all the quantities depend on the time $t$ here and $K$ is defined by $K:=\sup_{X_T}\left(1+|\nabla_\omega u|^2\right)$.

\begin{lemma}\label{lem:estimate T_alpha}
    For $\delta>0$ small enough, we have the following estimate:
    \begin{equation}\label{estimate of LT_alpha}
        |\mathcal{L}(T_\alpha(u-\underline{u}))|\leq\frac{1}{K^{\frac{1}{2}}}F^{p\bar{q}}(u-\underline{u})_{py_n}(u-\underline{u})_{\bar{q}y_n}+\mathcal{E}.
    \end{equation}
\end{lemma}
\begin{proof}
    By a basic calculation, we can write
    \begin{align*}
        \partial_t(T_\alpha u)=\partial_{s_\alpha}\partial_tu-\frac{\rho_{s_\alpha}}{\rho_{x_n}}\partial_{x_n}\partial_tu=F^{p\bar{q}}(\nabla_{s_\alpha}\chi_{p\bar{q}}+\nabla_{s_\alpha}u_{p\bar{q}})-\frac{\rho_{s_\alpha}}{\rho_{x_n}}F^{p\bar{q}}(\nabla_{x_n}\chi_{p\bar{q}}+\nabla_{x_n}u_{p\bar{q}})+\frac{\rho_{s_\alpha}}{\rho_{x_n}}\psi_{x_n}.
    \end{align*}
    While
    \begin{align*}
        F^{p\bar{q}}\partial_p\partial_{\bar{q}}T_\alpha(u)=F^{p\bar{q}}u_{p\bar{q}s_\alpha}-\frac{\rho_{s_\alpha}}{\rho_{x_n}}F^{p\bar{q}}u_{p\bar{q}x_n}-2\operatorname{Re}\left(F^{p\bar{q}}\left(\frac{\rho_{s_\alpha}}{\rho_{x_n}}\right)_pu_{\bar{q}x_n}\right)-F^{p\bar{q}}\left(\frac{\rho_{s_\alpha}}{\rho_{x_n}}\right)_{p\bar{q}}u_{x_n}.
    \end{align*}
    Here we have made the following calculation:
    \begin{align*}
    \nabla_{s_\alpha}u_{p\bar{q}}:&=(\nabla_{s_\alpha}u)(\partial_p,\partial_{\bar{q}})=u_{p\bar{q}s_\alpha}-\frac{\partial g^{p\bar{j}}}{\partial s_\alpha}u_{j\bar{q}}=u_{p\bar{q}s_\alpha}-\frac{1}{2}\left(\frac{\partial g^{p\bar{j}}}{\partial z_i}+\frac{\partial g^{p\bar{j}}}{\partial \bar{z}_i}\right)u_{j\bar{q}}\\
    &=u_{p\bar{q}s_\alpha}-\frac{1}{2}\left(\Gamma_{ip}^j+\Gamma_{\bar{i}\bar{p}}^{\bar{j}}\right)u_{j\bar{q}},
    \end{align*}
    if $s_\alpha=x_i=\frac{1}{2}(z_i+\bar{z}_i)$. Similarly we have the corresponding formula for $\nabla_{x_n}u_{p\bar{q}}$. Adding both equations we can cancel the third-order derivatives of $u$ and write
    \begin{align*}
\mathcal{L}({T_\alpha(u)})= \mathcal{E}+\frac{1}{2}F^{p\bar{q}}\left(\Gamma_{ip}^j+\Gamma_{\bar{i}\bar{p}}^{\bar{j}}\right)u_{j\bar{q}}+F^{p\bar{q}}\frac{1}{2}\left(\Gamma_{ip}^j+\Gamma_{\bar{i}\bar{p}}^{\bar{j}}\right)u_{j\bar{q}}-2\operatorname{Re}\left(F^{p\bar{q}}\left(\frac{\rho_{s_\alpha}}{\rho_{x_n}}\right)_pu_{\bar{q}x_n}\right).
\end{align*}
As in \cite{CP22}, we may further estimate
\begin{align*}
    |F^{p\bar{q}}\Gamma_{ip}^ju_{j\bar{q}}|=\left|\frac{1}{\sigma_m}\sum_k\sigma_{m-1}(\lambda|k)\lambda_kS^k_j\Gamma^{j}_{ip}\bar{S^k_p}\right|\leq C\sum_k\sigma_{m-1}(\lambda|k)|\lambda_k|,
\end{align*}
where $S=S^k_p$ is a unitary matrix which simultaneously diagonalizes $u_{j\bar{q}}$ and $F^{p\bar{q}}$ and we have used that fact that $\sigma_m=e^{\partial_tu+\psi}$ is uniformly bounded.

Since $\underline{u}\in C^\infty(X_T)$, we infer that
\begin{equation}
     |\mathcal{L}(T_\alpha(u-\underline{u}))|=-2\operatorname{Re}\left(F^{p\bar{q}}\left(\frac{\rho_{s_\alpha}}{\rho_{x_n}}\right)_p(u-\underline{u})_{\bar{q}x_n}\right)+\mathcal{E}.
\end{equation}
Note that $\frac{\partial}{\partial z_n}=\frac{1}{2}(\frac{\partial}{\partial x_n}-\sqrt{-1}\frac{\partial}{\partial y_n})$, we can write
$$
F^{p\bar{q}}\left(\frac{\rho_{s_\alpha}}{\rho_{x_n}}\right)_p(u-\underline{u})_{\bar{q}x_n}=2F^{p\bar{q}}\left(\frac{\rho_{s_\alpha}}{\rho_{x_n}}\right)_p(u-\underline{u})_{\bar{q}n}+\sqrt{-1}F^{p\bar{q}}\left(\frac{\rho_{s_\alpha}}{\rho_{x_n}}\right)_p(u-\underline{u})_{\bar{q}y_n}.
$$
The argument above gives that the first term can be dominated by $C\sum_k\sigma_{m-1}(\lambda|k)|\lambda_k|$, while the Cauchy-Schwarz inequality yields that
$$
\left|F^{p\bar{q}}(\frac{\rho_{s_\alpha}}{\rho_{x_n}})_p(u-\underline{u})_{\bar{q}y_n})\right|\leq \frac{1}{K^{\frac{1}{2}}}F^{p\bar{q}}(u-\underline{u})_{py_n}(u-\underline{u})_{\bar{q}y_n}+\mathcal{E}.
$$
The proof is finally concluded.
\end{proof}
\subsection{Estimates of quadratic gradient terms}

In this subsection, we establish the estimate for the standard quadratic gradient terms, which will serve as a crucial component in the final barrier construction.

\begin{lemma}\label{lem:quadratic_gradient}
    Let $u$ be the smooth solution to the flow \eqref{eq:main flow} and $\underline{u}$ be the admissible subsolution. For any $i \in \{1, \dots, n\}$, we have the following estimate on the parabolic cylinder $\Omega_{\delta,p}$:
    \begin{equation}\label{eq:estimate of quadratic}
        \mathcal{L}\left(\frac{1}{K^{\frac{1}{2}}}(\partial_{y_i}(u-\underline{u}))^2\right) \geq \frac{2}{K^{\frac{1}{2}}}F^{p\bar{q}}(u-\underline{u})_{y_ip}(u-\underline{u})_{y_i\bar{q}} - |\mathcal{E}|,
    \end{equation}
    where the error term $\mathcal{E}$ is bounded by $|\mathcal{E}| \leq C(1+K^{\frac{1}{2}})\mathcal{F} + C\sum_j\sigma_{m-1}(\lambda|j)|\lambda_j| + C$. The constant $C$ depends only on the geometry of $(X,\omega)$, $\|\chi\|_{C^2}$, $\|\underline{u}\|_{C^{3,1}(X_T)}$, and $\|\psi\|_{C^1(X_T)}$.
\end{lemma}

\begin{proof}
    We calculate the action of the linearized parabolic operator $\mathcal{L} = -\partial_t + F^{p\bar{q}}\partial_p\partial_{\bar{q}}$ on the term $\frac{1}{K^{1/2}}(\partial_{y_i}(u-\underline{u}))^2$. First, we compute the time derivative:
    \begin{align*}
        \frac{1}{K^{\frac{1}{2}}}\partial_t(\partial_{y_i}(u-\underline{u}))^2 &= \frac{2}{K^{\frac{1}{2}}}(u-\underline{u})_{y_i}\partial_{y_i}(\partial_tu-\partial_t\underline{u}) \\
        &= \frac{2}{K^{\frac{1}{2}}}(u-\underline{u})_{y_i}\left(\partial_{y_i}F(A[u]) - \psi_{y_i} - \partial_{y_i}\partial_t\underline{u}\right) \\
        &= \frac{2}{K^{\frac{1}{2}}}(u-\underline{u})_{y_i}\left(F^{p\bar{q}}(\nabla_{y_i}\chi_{p\bar{q}} + \nabla_{y_i}u_{p\bar{q}}) - \psi_{y_i} - \partial_{y_i}\partial_t\underline{u}\right) + \mathcal{E}.
    \end{align*}
    Here, passing from partial derivatives to covariant derivatives introduces commutator terms involving the Christoffel symbols of $\omega$. As demonstrated in the proof of \cref{lem:estimate T_alpha}, these terms are linear in $u_{j\bar{k}}$ and can be bounded by $C\sum_k\sigma_{m-1}(\lambda|k)|\lambda_k|$, thus they are absorbed into the error term $\mathcal{E}$.

    Next, we compute the spatial second-order derivatives:
    \begin{align*}
        \frac{1}{K^{\frac{1}{2}}}F^{p\bar{q}}\partial_p\partial_{\bar{q}}(\partial_{y_i}(u-\underline{u}))^2 &= \frac{2}{K^{\frac{1}{2}}}F^{p\bar{q}}(u-\underline{u})_{y_ip}(u-\underline{u})_{y_i\bar{q}} + \frac{2}{K^{\frac{1}{2}}}F^{p\bar{q}}(u-\underline{u})_{y_i}(u-\underline{u})_{y_ip\bar{q}}.
    \end{align*}
    Adding both terms together, the third-order derivatives of $u$ cancel exactly, leaving:
    \begin{align*}
        \mathcal{L}\left(\frac{1}{K^{\frac{1}{2}}}(\partial_{y_i}(u-\underline{u}))^2\right) &\geq \frac{2}{K^{\frac{1}{2}}}F^{p\bar{q}}(u-\underline{u})_{y_ip}(u-\underline{u})_{y_i\bar{q}} \\
        &\quad - \frac{2}{K^{\frac{1}{2}}}|(u-\underline{u})_{y_i}|\cdot\left| F^{p\bar{q}}\nabla_{y_i}\chi_{p\bar{q}} - \psi_{y_i} - \partial_{y_i}\partial_t\underline{u} \right| - |\mathcal{E}|.
    \end{align*}
    By the definition of $K$, we have $|(u-\underline{u})_{y_i}| \leq CK^{1/2}$. Therefore, the remaining first-order term is uniformly bounded by a constant multiple of $\mathcal{F}$ and can be absorbed into $-|\mathcal{E}|$. This yields \eqref{eq:estimate of quadratic}.
\end{proof}
\subsection{Estimates of quadratic gradient terms with frame}

We apply a similar methodology to the quadratic gradient terms associated with the tangential frame $e_a$. As in \cite[Section 4.5]{CP22}, we choose a local orthonormal frame $e_a:=e^i_a\partial_i$ of $T^{1,0}X$ such that $\{e_a\}_{a=1}^{n-1}$ are tangent to the level sets of $\rho$ (i.e., $e_a(\rho)=0$), and satisfying $e_a(0)=\frac{\partial}{\partial z_a}$ via the Gram-Schmidt process.

\begin{lemma}\label{lem:quadratic_frame}
    Under the same assumptions as \cref{lem:quadratic_gradient}, there exists an index $1 \leq r \leq n$ (which may depend on the time slice $t$), such that:
    \begin{equation}\label{estimate of quardtic frame}
        \mathcal{L}\left(\frac{1}{K^{\frac{1}{2}}}\sum_{a=1}^{n-1}|\nabla_{e_a}(u-\underline{u})|^2\right) \geq \frac{1}{2n\sigma_m(\lambda)K^{\frac{1}{2}}}\sum_{i\neq r}\sigma_{m-1}(\lambda|i)\lambda_i^2 - \frac{1}{K^{\frac{1}{2}}}\sum_{i=1}^n F^{p\bar{q}}(u-\underline{u})_{py_i}(u-\underline{u})_{y_i\bar{q}} - |\mathcal{E}|.
    \end{equation}
\end{lemma}

\begin{proof}
    We begin by computing the time derivative:
    \begin{align*}
        \partial_t|\nabla_a(u-\underline{u})|^2 &= \partial_t\left(e_a^i\bar{e_a^j}(u-\underline{u})_i(u-\underline{u})_{\bar{j}}\right)\\
        &= e_a^i\bar{e_a^j}\left[F^{p\bar{q}}(\nabla_i u_{p\bar{q}} + \nabla_i\chi_{p\bar{q}}) - \psi_i - \underline{u}_{it}\right](u-\underline{u})_{\bar{j}} \\
        &\quad + e_a^i\bar{e_a^j}\left[F^{p\bar{q}}(\nabla_{\bar{j}} u_{p\bar{q}} + \nabla_{\bar{j}}\chi_{p\bar{q}}) - \psi_{\bar{j}} - \underline{u}_{\bar{j}t}\right](u-\underline{u})_i + \mathcal{E}.
    \end{align*}
    Next, differentiating spatially gives:
    \begin{align*}
        F^{p\bar{q}} \partial_p \partial_{\bar{q}} |\nabla_a (u - \underline{u})|^2 &= F^{p\bar{q}} \partial_p \partial_{\bar{q}} (e_i^a\bar{e_j^a})(u - \underline{u})_i (u - \underline{u})_{\bar{j}} + F^{p\bar{q}} \partial_{\bar{q}} (e_i^a \bar{e_j^a})(u - \underline{u})_{pi} (u - \underline{u})_{\bar{j}} \\
        &\quad + F^{p\bar{q}} \partial_{\bar{q}} (e_i^a\bar{e_j^a})(u - \underline{u})_i (u - \underline{u})_{\bar{j}p} + F^{p\bar{q}} \partial_p (e_i^a\bar{e_j^a})(u - \underline{u})_{\bar{q}i} (u - \underline{u})_{\bar{j}} \\
        &\quad + F^{p\bar{q}} (e_i^a\bar{e_j^a})(u - \underline{u})_{p\bar{q}i} (u - \underline{u})_{\bar{j}} + F^{p\bar{q}} (e_i^a \bar{e_j^a})(u - \underline{u})_{\bar{q}i} (u - \underline{u})_{\bar{j}p} \\
        &\quad + F^{p\bar{q}} \partial_p (e_i^a \bar{e_j^a})(u - \underline{u})_i (u - \underline{u})_{\bar{q}\bar{j}} + F^{p\bar{q}} e_i^a \bar{e_j^a} (u - \underline{u})_{pi} (u - \underline{u})_{\bar{q}\bar{j}} \\
        &\quad + F^{p\bar{q}} e_i^a \bar{e_j^a} (u - \underline{u})_i (u - \underline{u})_{p\bar{q}\bar{j}}.
    \end{align*}
    Summing the terms inside the operator $\mathcal{L}$, the third-order derivatives of $u$ cancel directly:
    \begin{align*}
        \mathcal{L}\left(\frac{1}{K^{\frac{1}{2}}}|\nabla_a (u - \underline{u})|^2\right) &= \frac{1}{K^{\frac{1}{2}}}F^{p\bar{q}} (e_i^a \bar{e_j^a})(u - \underline{u})_{\bar{q}i} (u - \underline{u})_{\bar{j}p} + \frac{1}{K^{\frac{1}{2}}}F^{p\bar{q}} e_i^a \bar{e_j^a} (u - \underline{u})_{pi} (u - \underline{u})_{\bar{q}\bar{j}} \\
        &\quad + \frac{2}{K^{\frac{1}{2}}}\operatorname{Re}F^{p\bar{q}} \partial_{\bar{q}} (e_i^a \bar{e_j^a})(u - \underline{u})_{pi} (u - \underline{u})_{\bar{j}} + \mathcal{E}.
    \end{align*}
    This identity is precisely the parabolic counterpart to Equation (4.27) in \cite{CP22}. To extract the positivity, note that $F^{p\bar{q}} = \frac{\sigma_m^{p\bar{q}}}{\sigma_m(\lambda)}$. Because the time derivative bound (\cref{partial t estimate}) strictly bounds $\sigma_m(\lambda) = e^{\partial_t u + \psi}$ away from zero and infinity, we can carry over the algebraic manipulations from \cite[Section 4.5]{CP22}. 
\end{proof}

\subsection{Final estimates}
For constants $A,B>>1$, we define the final barrier as follows:
\begin{equation}
    \Psi:=AK^{\frac{1}{2}}v+BK^{\frac{1}{2}}|z|^2-\frac{1}{K^{\frac{1}{2}}}\sum_{i=1}^n((u-\underline{u})_{y_i})^2-\frac{1}{K^{\frac{1}{2}}}\sum_{a=1}^{n-1}|\nabla_a(u-\underline{u})|^2.
\end{equation}
Putting \cref{lem:estimate of Lv}, \cref{lem:estimate T_alpha}, \cref{lem:quadratic_gradient},\cref{lem:quadratic_frame} together, we arrive at the following estimate:
\begin{align*}
    \mathcal{L}(\Psi+T_\alpha(u-\underline{u}))\leq&-\tau(1+\mathcal{F})AK^{\frac{1}{2}}+BK^{\frac{1}{2}}\mathcal{F}-\frac{2}{K^{\frac{1}{2}}}\sum_{i=1}^nF^{p\bar{q}}(u-\underline{u})_{y_ip}(u-\underline{u})_{y_i\bar{q}}\\
    &-\frac{1}{2n\sigma_m(\lambda)K^{\frac{1}{2}}}\sum_{i\neq r}\sigma_{m-1}(\lambda|i)\lambda_i^2+\frac{1}{K^{\frac{1}{2}}}\sum_{i=1}^nF^{p\bar{q}}(u-\underline{u})_{py_i}(u-\underline{u})_{y_i\bar{q}}\\
    &+\frac{1}{K^{\frac{1}{2}}}F^{p\bar{q}}(u-\underline{u})_{py_n}(u-\underline{u})_{\bar{q}y_n}+\mathcal{E}\\
    =&-\tau(1+\mathcal{F})AK^{\frac{1}{2}}+BK^{\frac{1}{2}}\mathcal{F}-\frac{1}{K^{\frac{1}{2}}}\sum_{i=1}^{n-1}F^{p\bar{q}}(u-\underline{u})_{y_ip}(u-\underline{u})_{y_i\bar{q}}\\    
&-\frac{1}{2n\sigma_m(\lambda)K^{\frac{1}{2}}}\sum_{i\neq r}\sigma_{m-1}(\lambda|i)\lambda_i^2+C(1+K^{\frac{1}{2}})\mathcal{F}+C\sum_i\sigma_{m-1}(\lambda|i)|\lambda_i|+C.
\end{align*}
Now, we choose constants $A>CB\tau^{-1}+A_0\tau^{-1}$ for $A_0$ large enough and write
\begin{align*}
     \mathcal{L}(\Psi+T_\alpha(u-\underline{u}))\leq-\frac{2}{A_0}K^{\frac{1}{2}}(1+\mathcal{F})-\frac{1}{2n\sigma_m(\lambda)K^{\frac{1}{2}}}\sum_{i\neq r}\sigma_{m-1}(\lambda|i)\lambda_i^2+C\sum_i\sigma_{m-1}(\lambda|i)|\lambda_i|.
\end{align*}
An application of \cite[Lemma 3.4]{CP22} (see also \cite{GS10, Guan14}) immediately gives that
\begin{equation}
     \mathcal{L}(\Psi+T_\alpha(u-\underline{u}))<0\quad\operatorname{on}\quad \Omega_{\delta,p}.
\end{equation}
Now we look at the value of $\Psi+T_\alpha(u-\underline{u})$ on the parabolic boundary of $\Omega_{\delta,p}$. There are three parts to be considered: $\partial X\cap\Omega_\delta\times[0,t_0)$, $(\partial\Omega_\delta-\partial X)\times[0,t_0)$ and $\Omega_\delta\times\{0\}$. 

On the side $(\partial X\cap\Omega_\delta)\times[0,t_0)$, we have $u\equiv\underline{u}$ and thus $T_\alpha(u-\underline{u})=\nabla_a(u-\underline{u})=0$ for $1\leq a\leq n-1$. For the term $(\partial_{y_i}(u-\underline{u}))^2$, note that $(u-\underline{u})(s,\zeta(s))=0$ yields that
\begin{equation}\label{eq:partial_yi}
(\partial_{y_i}(u-\underline{u}))^2=(\partial_{x_n}(u-\underline{u}))^2(\zeta_{y_i})^2.
\end{equation}
Since $\zeta_{y_i}(0)=0$, we have $|\zeta_{y_i}(s)|\leq C|s|$ and hence
$$
(\partial_{y_i}(u-\underline{u}))^2\leq C|z|^2 \quad\operatorname{on}\,(\partial X\cap\Omega_\delta)\times[0,t_0)
$$
by the boundary gradient estimates. Overall, these estimates yield that
$$
\Psi+T_\alpha(u-\underline{u})\geq AK^{\frac{1}{2}}v+BK^{\frac{1}{2}}|z|^2-C|z|^2\geq0
$$
provided $B$ is large enough because $v\geq0$ by \cref{lem:estimate of Lv} and $K\geq1$.

On the side  $(\partial\Omega_\delta-\partial X)\times[0,t_0)$, since by definition $\Omega_\delta:=\Omega\cap B_\delta(0)$, we have that $|z|=\delta$ and hence
$$
\Psi+T_\alpha(u-\underline{u})\geq AK^{\frac{1}{2}}v+BK^{\frac{1}{2}}\delta^2-CK^{\frac{1}{2}}\geq0
$$
provided we take $B$ large enough.

On the bottom $\Omega_\delta\times\{0\}$, similar to \eqref{eq:partial_yi}, we have the estimate $\frac{1}{K^{\frac{1}{2}}}(\partial_{y_i}(u-\underline{u}))^2\leq CK^{\frac{1}{2}}|z|^2$. On the other hand, notice that $\nabla_a(u-\underline{u})=T_\alpha(u-\underline{u})=0$ on $\partial X$, we have the estimate 
$$
|\nabla_a(u-\underline{u})|\leq Cd\quad\operatorname{and}\quad |T_\alpha(u-\underline{u})|\leq Cd.
$$
Putting everything together and invoking that $v\geq\frac{c_0}{2}d$ on $\Omega_\delta$ by \cref{lem:estimate of Lv}, we can write
$$
\Psi+T_\alpha(u-\underline{u})\geq AK^{\frac{1}{2}}v+BK^{\frac{1}{2}}|z|^2-CK^{\frac{1}{2}}|z|^2-Cd\geq0,
$$
when $A,B>>1$ are taken large enough.

Now, the parabolic maximum principle yields that $\Psi+T_\alpha(u-\underline{u})\geq0$ on the whole $\Omega_{\delta,p}$. We also have $(\Psi+T_\alpha(u-\underline{u}))(0)=0$. This implies that $\partial_{x_n}(\Psi+T_\alpha(u-\underline{u}))(0)\geq0$ and hence
$$
0\leq AK^{\frac{1}{2}}v_{x_n}(0)+(u-\underline{u})_{x_ns_\alpha}(0)-\left(\frac{\rho_{s_\alpha}}{\rho_{x_n}}\right)_{x_n}(0)(u-\underline{u})_{x_n(0)}.
$$
Thanks to the boundary gradient estimates, we get that
$$
u_{s_\alpha x_n}(0)\geq-CK^{\frac{1}{2}}.
$$
Similarly, by considering $\Psi-T_\alpha(u-\underline{u})$ we get the upper bound, whence
\begin{equation}
    |u_{s_\alpha x_n}|(0)\leq CK^{\frac{1}{2}}.
\end{equation}
We have then finished the proof of the boundary tangential normal estimate \cref{tan-normal estimates}.

\subsection{Boundary double-normal estimates}
It remains to estimate the double normal derivatives $|u_{n\bar{n}}|$ of $u$. The arguments in the elliptic case established by Collins-Picard \cite{CP22} (which in turn builds on the classic work of Caffarelli-Nirenberg-Spruck \cite{CNS85}) are readily applicable to the parabolic case here, thanks to our a priori estimates.

Let $p\in\partial X\times\{t_0\}$ be a fixed boundary point with coordinates $(z_1,...,z_n)$ where $p$ corresponds to the origin and $g_{j\bar{k}}(0)=\delta_{j\bar{k}}$. Up to a unitary transformation we may assume that $\partial_{x_n}$ is the unit inner normal vector at $p$. Rotating coordinates in the tangential directions, we may assume that the matrix $h_{j\bar{k}}=\chi_{j\bar{k}}+u_{j\bar{k}}$ is of the form
\begin{equation}\label{matrix of h}
\begin{bmatrix}
\lambda'_1 & 0 & \cdots & 0 & {h}_{1\bar{n}} \\
0 & \lambda'_2 & \cdots & 0 & {h}_{2\bar{n}} \\
\vdots & \vdots & \ddots & \vdots & \vdots \\
0 & 0 & \cdots & \lambda'_{n-1} & {h}_{n-1,\bar{n}} \\
{h}_{n\bar{1}} & {h}_{n\bar{2}} & \cdots & {h}_{n,\overline{n-1}} & {h}_{n\bar{n}}
\end{bmatrix}.
\end{equation}
Here, $\lambda^\prime=(\lambda_1^\prime,...,\lambda_{n-1}^\prime)$ denotes the eigenvalues of the endomorphism $A$ restricted to the subbundle $T^{1,0}\partial X$. The trace of this matrix is positive; hence it follows from the double tangential estimates that
\begin{equation}\label{lower bound of double normal}
    h_{n\bar{n}}\geq-(\lambda_1^\prime+...+\lambda_{n-1}^\prime)\geq-C.
\end{equation}

It remains to bound $h_{n\bar{n}}$ from above. As observed in \cite[Theorem 5.1]{CP22}, it suffices to establish a uniform positive lower bound for $\sigma_{m-1}(\lambda^\prime)$. We claim that there is a uniform constant $c_0>0$, depending only on the background data and independent of $t_0$, such that
\begin{equation}\label{core lemma eq}
    \sigma_{m-1}(\lambda^\prime)\geq c_0>0.
\end{equation}

Unlike the mixed tangential-normal estimates, the double-normal estimate does not require differentiating the equation. Fixing the time slice $t = t_0$, we can view the flow equation purely as an elliptic equation:
\begin{equation*}
    \sigma_m(\lambda[u(\cdot, t_0)]) = \tilde{\psi}(\cdot, t_0) := e^{\partial_t u(\cdot, t_0) + \psi(\cdot, t_0)}.
\end{equation*}
By our crucial $\partial_t u$-estimate (\cref{partial t estimate}), the modified right-hand side $\tilde{\psi}$ is strictly bounded away from zero and infinity, uniformly in $t_0$. Furthermore, we have already established the global comparison $u \ge \underline{u}$ (\cref{prop:C0 estimate}) and the uniform boundary gradient estimates (\eqref{eq:boundary gradient estimate}). The admissible subsolution $\underline{u}$ satisfies $\sigma_m(\lambda[\underline{u}]) \ge e^{\partial_t \underline{u} + \psi} \ge c > 0$, ensuring its eigenvalues are uniformly separated from the boundary of the cone $\Gamma_m$.

Since these conditions perfectly match the elliptic setup, the purely $C^0$ barrier argument of Caffarelli-Nirenberg-Spruck \cite{CNS85}, as adapted to the complex Hessian setting by Collins-Picard (see \cite[Theorem 5.1 and Lemma 5.3]{CP22}), applies slice-wise without modification. This directly yields the uniform bound \eqref{core lemma eq}.

With \eqref{core lemma eq} in hand, we can now bound $h_{n\bar{n}}$ from above. At the point $p$, expanding $\sigma_m(h)$ along the last column yields:
\begin{align*}
    e^{\partial_t u + \psi} = \sigma_m(\lambda[h]) = h_{n\bar{n}}\sigma_{m-1}(\lambda^\prime) + \sigma_m(\lambda^\prime) - \sum_{i=1}^{n-1} |h_{i\bar{n}}|^2 \sigma_{m-2}(\lambda^\prime|i).
\end{align*}
Using \eqref{core lemma eq}, we obtain:
\begin{align*}
    h_{n\bar{n}} \leq c_0^{-1} \left( e^{\partial_t u + \psi} - \sigma_m(\lambda^\prime) + \sum_{i=1}^{n-1} |h_{i\bar{n}}|^2 \sigma_{m-2}(\lambda^\prime|i) \right).
\end{align*}
By the tangential estimates and the mixed tangential-normal estimates (\cref{tan-normal estimates}), we know that $|\lambda^\prime| \leq C$ and $|h_{i\bar{n}}| \leq C K^{\frac{1}{2}}$. Combining these with the uniform bounds on $\partial_t u$ and $\psi$, we conclude that
\begin{equation*}
    h_{n\bar{n}} \leq C K.
\end{equation*}
This completes the proof of the boundary double-normal estimates.

Combining the interior estimates \cref{prop: interior estimate} and all the boundary estimates, we now arrive at:
\begin{proposition}\label{prop: global second order}
    In the setting of \cref{intro:main_smooth_solution}, let $u$ be the smooth solution to the complex Hessian flow \eqref{eq:main flow}. There exists a uniform constant $C$ depending only on the geometry of $(X, \omega)$, $\|\chi\|_{C^2(X)}$, $\|\underline{u}\|_{C^{3,1}(X_T)}$, $\|\varphi\|_{C^4(\partial_0 X_T)}$, $\|\psi\|_{C^{2,1}(X_T)}$, and $\inf_{X_T} \psi$, such that
    \begin{equation}\label{eq: global second order}
        \sup_{X_T}\|\sqrt{-1}\partial\bar{\partial}u\|_{\omega} \leq C K,
    \end{equation}
    where $K := \sup_{X_T}(1 + |\nabla_\omega u|^2)$.
\end{proposition}

\subsection{Blow-up argument}
Having established all the required boundary second order estimates, we now use a Blow-up argument to derive the gradient estimate:
\begin{proposition}\label{blowup argument}
    Let $(X,\omega)$ be a compact Hermitian manifold with smooth boundary and let $\chi\in\Gamma_m(\omega)$ be a strictly $(\omega,m)$-positive form on $X$. Let $\psi=\psi(x,t)$ be a smooth function on $\overline{X_T}$ and $\varphi$ be a smooth boundary data which is second-order compatible at the corner such that $\varphi_0:=\varphi|_{\overline{X}\times\{0\}}$ is strictly $(\chi,\omega,m)$- subharmonic. Suppose there exists an admissible parabolic subsolution $\underline{u}\in C^\infty(\overline{X}\times(0,T])\cap C^{2,1}(\overline{X_T})$ for the flow \eqref{eq:main flow}, that is to say, for each $T>0$, we have
    \begin{equation}\label{eq:subsolution smooth}
        \begin{cases}
              \partial_t\underline{u}\leq\log\frac{(\chi+dd^c\underline{u})^m\wedge\omega^{n-m}}{\omega^n}-\psi\quad \operatorname{in} X_T,\\    
\underline{u}\leq\varphi\quad\operatorname{on}\partial_bX_T \quad \operatorname{and}\quad \underline{u}=\varphi \quad\operatorname{on}\partial_sX_T.
        \end{cases}
    \end{equation}
   Then, there is a constant $C$ depending on $X,\omega,\chi,\underline{u},\psi,\varphi$ and $T>0$ such that
   $$
\|\nabla u\|_{L^{\infty}(X_T)}\leq C,
   $$
   for each positive time $T$. Consequently, we also have that
   $$
\sup_{X_T}|\sqrt{-1}\partial\bar{\partial}u|\leq C.
   $$
\end{proposition}
\begin{proof}
    We argue by contradiction. It is clear that we can assume without loss of generality that $T<\infty$. Suppose that there are points $p_i\in X$ and positive times $t_i$ with $p_i\to p_\infty\in\overline{X}$ and $t_i\to t_\infty\in(0,T]$ such that $|\nabla u|(p_i,t_i)\to\infty$. Up to a new choice and rescaling we may suppose that
    $$
    |\nabla u|(p_i,t_i)=\sup_{X_{t_i}}|\nabla u|:=M_i\to\infty.
    $$
    Following the philosophy of \cite[Proposition 6.1]{CP22}, there are two cases to deal with:

    \textbf{Case 1.} We first assume that $p_\infty$ lies in the interior of $X$, where the situation is similar to that in \cite{PTô21}. We can then select a coordinate chart $B$ centered at $p_\infty$ (that is, we identify $p_\infty$ and the origin $0$) such that $\omega(0)=Id$. Since $p_i\to p_\infty$, we may assume without loss of generality that all the $p_i$ lie in $B$. Now consider the complex space $\mathbb{C}^n$, for each $R>0$, define a function $u_i:B_R(0)\to\mathbb{R}$ by
    $$
    u_i(w):=u(M_i^{-1}w+p_i,t_i).
    $$
    It follows that $|\nabla u_i|(w)=M_i^{-1}|\nabla u|(M_i^{-1}w+p_i,t_i)\leq1$ and $|\nabla u_i|(0)=1$. By the $C^0$-estimates, we have $\|u_i\|_\infty\leq C$. The interior estimates \cref{prop: interior estimate} combined with all the boundary second-order estimates then give that
    \begin{align*}
        \|\sqrt{-1}\partial\bar{\partial}u_i\|_\infty\leq CM_i^{-2}\sup_{X_{t_i}}|\sqrt{-1}\partial\bar{\partial}u|\leq CM_i^{-2}(1+\sup_{X_{t_i}}|\nabla u|^2)=CM_i^{-2}(1+M_i^2)\leq C.
    \end{align*}
    The Laplacian estimate combined with the Sobolev embedding theorem yields a H\"older norm bound of $u_i$:
    $$
    \|u_i\|_{C^{1,\alpha}\left(B_{\frac{R}{2}}(0)\right)}\leq C,
    $$
    for some $0<\alpha<1$. Therefore, one can extract a subsequence of $u_i$ that converges in $C^{1,\frac{\alpha}{2}}$ norm to a function $u_\infty$ defined on $B_{\frac{R}{2}}(0)$. Letting $R\to\infty$ and using a diagonal argument we get a function $u_\infty:\mathbb{C}^n\to\mathbb{R}$ with $|\nabla u_\infty|(0)=1$.

    At the time $t_i$, we have the following equation on $X$:
    \begin{equation}\label{eq:Hessian slice}
        (\chi(z)+\sqrt{-1}\partial\bar{\partial}u(z,t_i))^m\wedge\omega(z)^{n-m}=e^{\partial_tu(z,t_i)+\psi(z,t_i)}\omega(z)^n.
    \end{equation}
    Define a biholomorphic map $\Phi_i:\mathbb{C}^n\to\mathbb{C}^n$ by 
    $$
    \Phi_i(w):=M_i^{-1}w+p_i.
    $$
    Pulling back equation \eqref{eq:Hessian slice} under $\Phi_i$, the form $\chi(z)=\sqrt{-1}\chi_{j\bar{k}}dz^j\wedge d\bar{z}^k$ is changed to $M_i^{-2}\sqrt{-1}\tilde{\chi}_{j\bar{k}}^i(w)dw^j\wedge d\bar{w}^k$, where $\tilde{\chi}_{j\bar{k}}^i(w)=\chi_{j\bar{k}}(M_i^{-1}w+p_i)$. We obtain that
    \begin{align*}
        &\left(\frac{\sqrt{-1}}{M_i^2}\tilde{\chi}_{j\bar{k}}^i(w)dw^j\wedge d\bar{w}^k+\sqrt{-1}\partial\bar{\partial}u_i(w)\right)^m\wedge\left(\frac{\sqrt{-1}}{M_i^2}\tilde{\omega}_{j\bar{k}}^i(w)dw^j\wedge d\bar{w}^k\right)^{n-m}\\
        =&e^{\partial_tu(M_i^{-1}w+p_i,t_i)+\psi(M_i^{-1}w+p_i,t_i)}\left(\frac{\sqrt{-1}}{M_i^2}\tilde{\omega}_{j\bar{k}}^i(w)dw^j\wedge d\bar{w}^k\right)^{n}.
    \end{align*}
    Multiplying both sides of the equation by $M_i^{2(n-m)}$, the left-hand side stabilizes, while the right-hand side acquires a factor of $M_i^{-2m}$. Since $m \geq 1$ and the term $e^{\partial_t u + \psi}$ is uniformly bounded globally thanks to our a priori $\partial_t u$-estimates, the right-hand side strictly converges to $0$ as $M_i \to \infty$. 
    
    It is clear by definition that $M_i^2 \left(\frac{\sqrt{-1}}{M_i^2}\tilde{\omega}_{j\bar{k}}^i(w)dw^j\wedge d\bar{w}^k\right) \to \beta = \sqrt{-1}\partial\bar{\partial}|w|^2$ uniformly on compact sets of $\mathbb{C}^n$. Therefore, letting $i\to\infty$ we get by standard Bedford-Taylor convergence that the following holds on the whole $\mathbb{C}^n$:
    $$
    (\sqrt{-1}\partial\bar{\partial}u_\infty)^m\wedge(\sqrt{-1}\partial\bar{\partial}|w|^2)^{n-m}=0.
    $$
    The Liouville-type theorem for $m$-subharmonic functions established in \cite{DK17} then gives that $u_\infty$ is constant, contradicting the fact that $|\nabla u_\infty|(0)=1$.

\textbf{Case 2.} We now deal with the case where $p_\infty\in\partial X$. Let $\Omega=B_{2s}\cap\{\rho\leq0\}$ be a boundary coordinate chart centered at $p_\infty$ such that $\omega(p_\infty)=Id$. Recall that by our choice of coordinates, the inner normal direction is given by the real part $x_n$. We may assume without loss of generality that all the $p_i$ lie in $B_{s}\cap\{\rho\leq0\}$. Choose a point $y_i\in\partial X\cap\Omega$ minimizing the Euclidean distance between $p_i$ and $\partial X$ with $r_i:=|y_i-p_i|$.
    
\textbf{Case 2a.}  Assume that $\liminf_{i}M_ir_i=\infty$. Set
    $$
    u_i(z):=u(M_i^{-1}z+p_i, t_i).
    $$
As in the first case, we have $|\nabla u_i(0)|=1$, $|\nabla u_i|\leq 1$ and 
    \begin{equation}\label{eq:case2a}
    \|u_i\|_{L^\infty(\Omega_i)}+\|\Delta u_i\|_{L^{\infty}(\Omega_i)}\leq C,
    \end{equation}
    where $\Omega_i:=\{z:M_i^{-1}z+p_i\in B_{2s}\}\cap\{\rho\leq0\}$. Set $\hat{\rho}_i(z):=\rho(M_i^{-1}z+p_i)$. It is easy to see that we have an inclusion of sets:
    $$
    B_{\frac{M_i r_i}{2}}(0) \subset B_{s M_i}(0)\cap\{\hat{\rho}_i\leq0\}\subset\Omega_i.
    $$
Since $\liminf_{i}M_ir_i=\infty$, the sequence of domains $B_{sM_i}\cap\{\hat{\rho}_i\leq0\}$ exhausts the whole space $\mathbb{C}^n$. Invoking \eqref{eq:case2a}, standard elliptic theory and the Sobolev embedding theorem gives a constant $0<\alpha<1$ such that $\|u_i\|_{C^{1,\alpha}}\leq C$ uniformly on compact subsets and we can extract a $C^{1,\frac{\alpha}{2}}$ limit $u_\infty$ of $u_i$ on every compact subset of $\mathbb{C}^n$. The contradiction follows exactly as in Case 1 from the Liouville-type theorem.

\textbf{Case 2b.} Suppose now that $\liminf_{i \to \infty} M_i r_i = l \in [0, \infty)$. Up to extracting a subsequence, we may assume $M_i r_i \to l$. Choose a ball of radius $r^*$ centered at a point $y^*$ such that $B_{r^*}(y^*)\subset B_{s}\cap\{\rho\leq0\}$ and that $B_{r^*}(y^*)$ is tangent to $\partial X$ at $p_\infty$. Shrinking $r^*$ if necessary we may also assume that there are points $y_i^*$ such that $B_{r^*}(y_i^*)$ is tangent to $\partial X$ at $y_i$. It is then clear that $y_i\to p_\infty$ and that $y_i^*\to y^*$. For each $i$, we can find a unitary transformation (an isometry of $\mathbb{C}^n$) $\mathcal{R}_i : \mathbb{C}^n \to \mathbb{C}^n$ such that $\mathcal{R}_i(0) = y_i$, and the inward unit normal vector to $\partial X$ at $y_i$ is pulled back precisely to the positive real direction of the $w_n$-axis, i.e., $\frac{\partial}{\partial \operatorname{Re}(w_n)}$. Since $p_i$ and $y_i^*$ lie exactly on the inward normal line originating from $y_i$ at distance $r_i$ and $r^*$, their preimage under this transformation are strictly:
    $$
    \mathcal{R}_i^{-1}(p_i) = (0, \dots, 0, r_i),\quad    \mathcal{R}_i^{-1}(y_i^*) = (0, \dots, 0, r^*).
    $$
    Now, we define the rescaled sequence $u_i$ centered at $y_i$ with the spatial scaling factor $M_i$:
    $$
    u_i(w) := u\left(\mathcal{R}_i(M_i^{-1}w), t_i\right).
    $$
 The condition $\mathcal{R}_i(M_i^{-1}w) \in B_{r^*}(y_i^*)$ is equivalent to $M_i^{-1}w \in B_{r^*}((0, \dots, 0, r^*))$, which in turn means $w \in M_i B_{r^*}((0, \dots, 0, r^*)) = B_{M_i r^*}((0, \dots, 0, M_i r^*))$. As $M_i \to \infty$, the radius $M_i r^* \to \infty$ and the centers move to infinity along the positive real $w_n$-axis. Consequently, this sequence of scaled balls exhaust the right half-space $H := \{ w \in \mathbb{C}^n \mid \operatorname{Re}(w_n) > 0 \}$. Therefore, looking at the restriction of $u$ to $B_{r^*}(y_i^*)$, the domain of our rescaled functions $u_i$ exhausts $H$.
    
Observe that the original maximal gradient point $p_i$ corresponds to the point $w_i^*$ in our new coordinates, where:
    $$
    w_i^* = M_i \mathcal{R}_i^{-1}(p_i) = (0, \dots, 0, M_i r_i).
    $$
    Notice that $w_i^* \to (0, \dots, 0, l)$ as $i \to \infty$. By definition, we have $|\nabla u_i|(w_i^*) = 1$ and $|\nabla u_i| \leq 1$. 

    By standard elliptic theory and the Sobolev embedding theorem, we have a uniform $C^{1, \alpha}$ bound on compact subsets of ${H}$. Thus, we can extract a subsequence of $u_i$ converging uniformly in $C^{1, \frac{\alpha}{2}}$ to a limit function $u_\infty$ on $H\cup\{0\}$. Due to the convergence of the points $w_i^*$, the limit function must satisfy:
    \begin{equation}\label{eq:case2b}
    |\nabla u_\infty|(0, \dots, 0, l) = 1.
    \end{equation}

The final contradiction is derived from the global boundary barriers. Recall that we have the global bound $\underline{u} \leq u \leq b$ on $X_T$. We define the corresponding rescaled functions for the subsolution and the upper barrier exactly as we did for $u$:
    $$
    \underline{u}_i(w) := \underline{u}(\mathcal{R}_i(M_i^{-1}w), t_i), \quad \text{and} \quad b_i(w) := b(\mathcal{R}_i(M_i^{-1}w), t_i).
    $$
    This immediately implies that $\underline{u}_i(w) \leq u_i(w) \leq b_i(w)$ on their domains of definition. 
    
    Fix an arbitrary compact subset $K \subset H\cup\{0\}$. As $M_i \to \infty$, for any $w \in K$, the points $\mathcal{R}_i(M_i^{-1}w)$ uniformly converge to $p_\infty \in \partial X$. Recall that both $\underline{u}$ and $b$ are fixed smooth functions that smoothly coincide with the boundary data $\varphi$ on the lateral boundary $\partial_s X_T$. Because $y_i \to p_\infty$ and $t_i \to t_\infty$, both sequences of functions $\underline{u}_i(w)$ and $b_i(w)$ converge uniformly on $K$ to the exact same constant value, namely $\varphi(p_\infty, t_\infty)$.
    
    Since $\underline{u}_i(w) \leq u_i(w) \leq b_i(w)$, the squeeze theorem forces the limit function $u_\infty(w)$ to be identically equal to the constant $\varphi(p_\infty, t_\infty)$ on $H\cup\{0\}$. This implies $\nabla u_\infty \equiv 0$ everywhere, which contradicts \eqref{eq:case2b}. The proof is thus complete.
\end{proof}
We can now complete the proof of \cref{intro:main_smooth_solution}:
\begin{proof}[End of proof of \cref{intro:main_smooth_solution}]
   Suppose \(T_{\max}\) is the maximal existence time of the flow
\eqref{eq:main flow}. Combining \cref{partial t estimate},
\cref{prop:C0 estimate}, and \cref{blowup argument}, we have established
the following uniform a priori estimates for the smooth solution \(u\) on
each finite time interval:
\begin{equation}
    \|u\|_{L^\infty(X_T)}
    +\|\partial_t u\|_{L^\infty(X_T)}
    +\|\nabla_\omega u\|_{L^\infty(X_T)}
    +\|\sqrt{-1}\partial\bar{\partial}u\|_{L^\infty(X_T)}
    \le C,
\end{equation}
where \(C\) depends only on the background data and on \(T\).

Since
\[
    \sigma_m(\lambda[\chi_u])
    =
    e^{\partial_tu+\psi}
\]
is bounded from above and below, and since \(\chi_u=\chi+dd^cu\) is
uniformly bounded, the eigenvalues \(\lambda[\chi_u]\) remain in a compact
subset of the admissible cone \(\Gamma_m\). Hence the linearized operator is
uniformly elliptic, and the equation is uniformly parabolic.

We now apply the standard Evans--Krylov--Schauder regularity theory. This is
a local issue. In holomorphic coordinate charts, and in boundary charts after
flattening \(\partial X\), the complex Hessian flow can be viewed as a concave
fully nonlinear parabolic equation in the real Hessian variables. More
precisely, since the eigenvalues of \(\chi_u\) stay in a compact subset of
\(\Gamma_m\), one may extend the operator off a small neighborhood of the
compact set traced by the solution to obtain a concave uniformly parabolic
operator in the real Hessian. This is the standard reduction used in the
complex Evans--Krylov theory of Wang \cite{Wang12}, Tosatti--Wang--Weinkove--
Yang \cite{TWWY15}, and Chu \cite{Chu16}; near the boundary we use the same
localization and flattening philosophy as in Collins--Picard \cite{CP22}.

The boundary \(C^{2+\alpha,1+\alpha/2}\) estimate then follows from the
parabolic Evans--Krylov theory for concave uniformly parabolic Dirichlet
problems, together with standard boundary Schauder theory; see Krylov
\cite{Kry83} and Lieberman \cite{Lie96}. Consequently, for some
\(\alpha\in(0,1)\),
\[
    \|u\|_{C^{2+\alpha,1+\alpha/2}(\overline{X_T})}\le C.
\]
Here the estimate is first obtained away from the initial corner; the
regularity up to \(t=0\) follows from the imposed compatibility conditions
and the standard short-time theory.

Once this estimate is obtained, the coefficients of the linearized operator
belong to \(C^{\alpha,\alpha/2}\) up to the parabolic boundary. Standard
parabolic Schauder estimates then yield, for every integer \(k\ge0\),
\[
    u\in C^\infty(X\times(0,T])
    \cap C^{2k+\alpha,k+\alpha/2}(\overline{X_T}).
\]
The uniform estimates allow the solution to be continued past
\(T_{\max}\) unless \(T_{\max}=T\). Hence the maximal existence time is the
prescribed time, completing the proof.
\end{proof}
In order to finish the proof of \cref{intro:main_smooth_solution 2}, we utilize the strict $m$-pseudoconvexity of $\Omega$ to construct a concrete parabolic subsolution on $\Omega_T$:
\begin{proof}[Proof of \cref{intro:main_smooth_solution 2}]
Since $M$ is a smooth $m$-pseudoconvex Hermitian manifold, we may choose a smooth defining
function $\rho\in C^\infty(\overline{M})$ such that
\[
\rho<0 \quad \text{in } M,\qquad \rho=0 \quad \text{on }\partial M,
\]
and $\rho$ is strictly $(\omega,m)$-subharmonic on $\overline{ M}$.
Set
\[
\chi:=dd^c\rho .
\]
Then $\chi$ is a smooth strictly $(\omega,m)$-positive $(1,1)$-form on $\overline{ M}$.

We now reduce \eqref{introduction:main flow 2} to the form treated in
\cref{intro:main_smooth_solution}. Let
\[
v:=u-\rho.
\]
Since $\rho$ is independent of $t$, the equation \eqref{introduction:main flow 2} becomes
\begin{equation}\label{eq:transformed-flow}
\begin{cases}
\partial_t v
=\log\displaystyle\frac{(\chi+dd^cv)^m\wedge\omega^{n-m}}{\omega^n}-\psi
\quad \text{in } M_T,\\[1.2ex]
v=\widehat{\varphi}\quad \text{on }\partial_0 M_T,
\end{cases}
\end{equation}
where
\[
\widehat{\varphi}:=\varphi-\rho .
\]
Since $\rho|_{\partial M}=0$, the lateral boundary datum of $v$ is still $\varphi$.
Moreover,
\[
\widehat{\varphi}_0:=\widehat{\varphi}|_{\overline{ M}\times\{0\}}
=\varphi_0-\rho,
\]
and
\[
\chi+dd^c\widehat{\varphi}_0
=dd^c\rho+dd^c(\varphi_0-\rho)=dd^c\varphi_0.
\]
Hence $\widehat{\varphi}_0$ is strictly $(\chi,\omega,m)$-subharmonic because
$\varphi_0$ is strictly $(\omega,m)$-subharmonic by assumption.
Also, since $\rho$ is time-independent and vanishes on $\partial M$,
the compatibility conditions of order $k$ for $\varphi$ are inherited by
$\widehat{\varphi}$.

It therefore remains only to construct an admissible parabolic subsolution for
\eqref{eq:transformed-flow}.

\medskip
\noindent
\textbf{Step 1: Extension of the parabolic boundary data.}
We claim that there exists a function
\[
\widetilde{\varphi}\in C^{2,1}(\overline{ M_T})\cap C^\infty(\overline{ M}\times(0,T])
\]
such that
\[
\widetilde{\varphi}=\varphi \quad \text{on } \partial_s M_T,
\qquad
\widetilde{\varphi}(\cdot,0)=\varphi_0 \quad \text{on } \overline{ M},
\]
and such that both $\partial_t\widetilde{\varphi}$ and $dd^c\widetilde{\varphi}$ are bounded on
$\overline{ M_T}$.

Indeed, since $\partial M$ is smooth and compact, we may choose finitely many boundary
charts $\{U_\alpha\}_{\alpha=1}^N$ covering $\partial M$ such that for each $\alpha$ there is
a smooth diffeomorphism
\[
\kappa_\alpha:U_\alpha\longrightarrow B_2\subset\mathbb{R}^{2n}
\]
satisfying
\[
\kappa_\alpha(U_\alpha\cap M)=B_2^+:=B_2\cap\{x_{2n}>0\},
\qquad
\kappa_\alpha(U_\alpha\cap\partial M)=B_2':=B_2\cap\{x_{2n}=0\}.
\]
Let $V\Subset \bigcup_{\alpha=1}^N U_\alpha$ be an open neighborhood of $\partial M$, and
choose a partition of unity
\[
\theta_0,\theta_1,\dots,\theta_N\in C^\infty(\overline{ M}\cup V),
\qquad
\sum_{\alpha=0}^N\theta_\alpha\equiv 1 \ \text{on }\overline{ M},
\]
such that $\operatorname{Supp} \theta_0\Subset  M$ and $\operatorname{Supp}\theta_\alpha\Subset U_\alpha$ for $\alpha=1,\dots,N$.

Fix a cutoff function $\eta\in C^\infty([0,\infty))$ such that
\[
\eta(s)\equiv 1 \quad \text{for } 0\le s\le \tfrac14,
\qquad
\eta(s)\equiv 0 \quad \text{for } s\ge \tfrac12.
\]
For each $\alpha=1,\dots,N$, write
\[
\kappa_\alpha(x)=(y',y_{2n}),\qquad y'\in\mathbb{R}^{2n-1},\ y_{2n}\ge 0,
\]
and define on $(U_\alpha\cap\overline{ M})\times[0,T]$
\[
H_\alpha(x,t)
:=
\eta(y_{2n})
\Big(
\varphi(\kappa_\alpha^{-1}(y',0),t)
-
\varphi(\kappa_\alpha^{-1}(y',0),0)
\Big).
\]
Since $\varphi$ is smooth on $\partial M\times[0,T)$, each $H_\alpha$ belongs to
\[
C^{2,1}\big((U_\alpha\cap\overline{ M})\times[0,T]\big)
\cap
C^\infty\big((U_\alpha\cap\overline{ M})\times(0,T]\big).
\]
We then set
\[
\widetilde{\varphi}(x,t)
:=
\varphi_0(x)+\sum_{\alpha=1}^N \theta_\alpha(x)\,H_\alpha(x,t),
\qquad (x,t)\in\overline{ M_T}.
\]

By construction, $\widetilde{\varphi}\in C^{2,1}(\overline{ M_T})\cap
C^\infty(\overline{ M}\times(0,T])$. Moreover, if $x\in\partial M$, then
$\theta_0(x)=0$, $y_{2n}=0$, and hence $\eta(y_{2n})=1$. Therefore,
\[
H_\alpha(x,t)=\varphi(x,t)-\varphi(x,0)
\]
for every $\alpha$ such that $x\in U_\alpha$, and thus
\[
\widetilde{\varphi}(x,t)
=
\varphi_0(x)+\sum_{\alpha=1}^N\theta_\alpha(x)\bigl(\varphi(x,t)-\varphi(x,0)\bigr)
=
\varphi(x,t),
\]
because $\varphi_0|_{\partial M}=\varphi(\cdot,0)$ and $\sum_{\alpha=1}^N\theta_\alpha(x)=1$ on $\partial M$. On the other hand, at $t=0$ we have $H_\alpha(x,0)=0$ for every $\alpha$, so
\[
\widetilde{\varphi}(x,0)=\varphi_0(x)
\qquad\text{for all }x\in\overline{ M}.
\]
Finally, since $\widetilde{\varphi}\in C^{2,1}(\overline{ M_T})$ and $\overline{ M_T}$ is compact, both $\partial_t\widetilde{\varphi}$ and
$dd^c\widetilde{\varphi}$ are bounded on $\overline{ M_T}$.

\medskip
\noindent
\textbf{Step 2: Construction of a subsolution.}
For a constant $A>1$ to be chosen, define
\[
\underline{v}:=\widetilde{\varphi}+(A-1)\rho .
\]
Then
\[
\partial_t\underline{v}=\partial_t\widetilde{\varphi},
\qquad
\chi+dd^c\underline{v}
=dd^c\rho+dd^c\widetilde{\varphi}+(A-1)dd^c\rho
=dd^c\widetilde{\varphi}+A\,dd^c\rho .
\]
Since $\rho$ is strictly $(\omega,m)$-subharmonic on $\overline{ M}$ and
$dd^c\widetilde{\varphi}$ is bounded, by compactness there exist constants
$A_0>1$ and $c_0>0$ such that for all $A\ge A_0$,
\[
dd^c\widetilde{\varphi}+A\,dd^c\rho \in \Gamma_m(\omega)
\quad\text{on }\overline{ M_T},
\]
and
\[
\bigl(dd^c\widetilde{\varphi}+A\,dd^c\rho\bigr)^m\wedge\omega^{n-m}
\ge c_0 A^m\,\omega^n
\quad\text{on }\overline{ M_T}.
\]
Therefore, for $A\ge A_0$,
\[
\log\frac{(\chi+dd^c\underline{v})^m\wedge\omega^{n-m}}{\omega^n}
\ge \log(c_0A^m).
\]
Choosing $A$ even larger if necessary so that
\[
\log(c_0A^m)\ge
\sup_{\overline{ M_T}}\bigl(\partial_t\widetilde{\varphi}+\psi\bigr),
\]
we obtain
\[
\partial_t\underline{v}
\le
\log\frac{(\chi+dd^c\underline{v})^m\wedge\omega^{n-m}}{\omega^n}-\psi
\quad\text{in } M_T.
\]
On the lateral boundary $\partial_s M_T$, since $\rho=0$ on $\partial M$ and
$\widetilde{\varphi}=\varphi$, we have
\[
\underline{v}=\varphi=\widehat{\varphi}
\quad\text{on }\partial_s M_T.
\]
On the bottom boundary $\partial_b M_T=\overline{ M}\times\{0\}$, we have
\[
\underline{v}(x,0)=\varphi_0(x)+(A-1)\rho(x)\le \varphi_0(x)-\rho(x)
=\widehat{\varphi}(x,0),
\]
because $\rho\le 0$ on $\overline{ M}$ and $A>0$.
Hence $\underline{v}$ is an admissible parabolic subsolution of
\eqref{eq:transformed-flow} in the sense of \cref{intro:main_smooth_solution}.

\medskip
\noindent
\textbf{Step 3: Application of \cref{intro:main_smooth_solution}.}
All the assumptions of \cref{intro:main_smooth_solution} are now satisfied for the
transformed problem \eqref{eq:transformed-flow}. Therefore there exists a unique solution
\[
v\in C^\infty( M\times(0,T])\cap
C^{2k+\alpha,k+\frac{\alpha}{2}}(\overline{ M_T}).
\]
Setting
\[
u:=v+\rho,
\]
we conclude that $u$ is the unique solution in
\[
C^\infty( M\times(0,T])\cap
C^{2k+\alpha,k+\frac{\alpha}{2}}(\overline{ M_T})
\]
of \eqref{introduction:main flow 2}. This proves the corollary.
\end{proof}

\section{Parabolic m-potentials}
In this section, we shift our focus to the pluripotential framework and start to develop the parabolic pluripotential theory for $(\omega,m)$-subharmonic functions. Let $\Omega \subset \mathbb{C}^n$ be a bounded $m$-pseudoconvex domain equipped with a Hermitian metric $\omega$. Throughout this paper, we use the standard convention $d^c := \frac{\sqrt{-1}}{2}(\bar{\partial} - \partial)$, so that $dd^c = \sqrt{-1}\partial\bar{\partial}$. Firstly, let us recall the definition of $(\omega,m)$-subharmonic functions:

\begin{definition}\label{def: def of m-sh}
    A function $u\in L_{loc}^1(\Omega,\omega^n)$ is called $(\omega,m)$-subharmonic if it satisfies (i) and one of (ii), (iii):
    \begin{enumerate}
   \item $dd^cu$ is an $(\omega,m)$-positive current, i.e., for arbitrary $(\omega,m)$-positive (1,1)-forms $\alpha_1,...\alpha_{m-1}$ on $\Omega$, the following inequality holds in the weak sense of currents:
    $$
    dd^{c}u\wedge\alpha_1\wedge...\wedge\alpha_{m-1}\wedge\omega^{n-m}\geq 0.
    $$
   \item    $u$ is strongly upper semi-continuous, $$i.e. \forall x\in\Omega, u(x)=\underset{\Omega\ni y\rightarrow x}{\operatorname{ess}\limsup}\,u(y):=\underset{r\searrow0}{\lim}\operatorname{ess}\underset{B_r(x)}{\sup}u .$$
   \item Assume $v\in L_{loc}^1(X,\Omega)$, upper semi-continuous and satisfies (i). If u and v coincide almost everywhere, then $u\leq v$.
\end{enumerate} 
\end{definition}
We next recall the definition of Hessian measures with respect to a background Hermitian metric due to Ko\l odziej-Nguyen \cite{KN26}:

\begin{definition}\label{def: Hessian measures in the local setting}
    Let $u_1,...,u_m$ be bounded $(\omega,m)$-subharmonic functions on a bounded domain $\Omega$ in $\mathbb{C}^n$ and $\omega$ be a Hermitian metric in $\Omega$. By \cite[Theorem 3.3]{KN26}, we can define inductively 
    $$
dd^cu_{p+1}\wedge...\wedge dd^cu_1:=dd^c(u_{p
+1}dd^cu_p\wedge...\wedge dd^cu_1)
    $$
    as closed real currents of order $0$. Then
    $$
H_p(u_1,...,u_p):=dd^cu_{p}\wedge...\wedge dd^cu_1\wedge\omega^{n-p},\quad 1\leq p\leq m
    $$
    is a well-defined  positive current (positive Radon measure when $p=m$) on $\Omega$ (when there is no confusing with the reference hermitian metric $\omega$, we will omit the subscript in $H_{p,\omega}$ for convenience of notations). When $u_1=...=u_m=u$, we write
    $$
H_p(u):=(dd^cu)^p\wedge\omega^{n-p},\quad1\leq p\leq m.
    $$
    \end{definition}

\begin{definition}\cite[Definition 1.1]{GLZ21a}
    A function $u:\Omega_T\to[-\infty,+\infty)$ is said to be locally uniformly Lipschitz in $(0,T)$ if for any relatively compact subinterval $J\subset\subset(0,T)$ there exists a constant $\kappa:=\kappa_J(u)>0$ such that for any $s,t\in J$ and $z\in\Omega$, we have
    $$
|u(t,z)-u(s,z)|\leq\kappa|t-s|.
    $$
\end{definition}
\begin{definition}
    A function $u:\Omega_T\to[-\infty,+\infty)$ is said to be an $(\omega,m)$-parabolic potential ($m$-parabolic potential for short) if it is locally uniformly Lipschitz in $t\in(0,T)$ and such that each slice $u_t(\cdot):=u(\cdot,t)$ is $(\omega,m)$-subharmonic. We let $P_{\omega,m}(\Omega_T)$ ($P_m(\Omega_T)$ if there is no confusion with the background Hermitian form) to be the set of $(\omega,m)$-parabolic potentials.
\end{definition}

Here are some immediate properties of $m$-parabolic potentials inherited from the corresponding analog of $(\omega,m)$-subharmonic functions:

\begin{proposition}\label{prop:envelope basic properties}
    \begin{enumerate}
\item If $u,v\in P_m(\Omega_T)$, then, $u+v\in P_m(\Omega_T)$ and $\max(u,v)\in P_m(\Omega_T)$.
\item Let $\mathcal{U}\subset P_m(\Omega_T)$ be a family of parabolic $m$-potentials which is locally uniformly bounded above. Set $U:=\sup\{u:u\in\mathcal{U}\}$ and assume that it is locally uniformly Lipschitz in $t\in(0,T)$. Then we have $U^*\in P_m(\Omega_T)$, $U^*(t,\cdot)=(U_t)^*$ in $\Omega$ for each $t\in(0,T)$ and the (negligible) set 
        $$
        \{(t,z)\in\Omega_T:U(t,z)<U^*(t,z)\}
$$
has zero Lebesgue measure in $\Omega_T\subset\mathbb{R}^{2n+1}$.
\item If $u\in P_m(\Omega_T)$, then $u$ is upper semi-continuous in $\Omega_T=\Omega\times(0,T)$, hence locally bounded from above in $\Omega_T$.
    \end{enumerate}
\end{proposition}
\begin{proof}
    The first statement is immediate from the analog of $(\omega,m)$-subharmonic functions. The second statement can be proved in the same way from \cite[Lemma 1.7]{GLZ21a} (here we have to use \cite[Theorem 7.8]{KN26}). The last statement is a direct consequence of \cite[Lemma 1.5]{GLZ21a}.
\end{proof}

To move further, we have to recall the submean value inequality of $\omega$-subharmonic functions, established in \cite{GN18}. For each ball $B_r(0)\subset\mathbb{C}^n$, it is well known that there exists a Poisson kernel $P_r(x,y)$ defined on $B_r(0)\times S_r(0)$ for $\Delta_\omega$ such that it satisfies
\begin{enumerate}
    \item $\Delta_{\omega(x)}P_r(x,y)=0$ on $B_r(0)$.
    \item $\int_{S_r(0)}P_r(x,y)d\sigma_r(y)=1$, where $d\sigma_r$ is the surface Lebesgue measure on $S_r(0)$.
    \item For any continuous data $\varphi\in C^0(S_r(0))$, the function $h(x):=\int_{S_r(0)}\varphi(y)P_r(x,y)d\sigma(y)$ is the unique continuous solution to the following Dirichlet problem:
    $$
\begin{cases}
    \Delta_{\omega}h=0 &\operatorname{on} B_r(0),\\
    h=\varphi&\operatorname{on} S_r(0).
\end{cases}
$$
\end{enumerate}
Here, $S_r(0):=\partial B_r(0)$ denotes the boundary of the ball. Then we have the following submean value inequality:
\begin{lemma}\cite[Lemma 9.6]{GN18}\label{SMI}
    Let $u$ be an $\omega$-subharmonic function on a domain $\Omega\subset\mathbb{C}^n$. For each $\delta>0$, set $\Omega_\delta:=\{z\in\Omega,dist(z,\partial\Omega)>\delta\}$. Then, for all $a\in\Omega_\delta$ and $r\in[0,\delta]$, we have
    $$
u(a)\leq M(u,a,r):=\int_{\partial B_r(a)}u(z)P_{a,r}(a,z)d\sigma_{r}(z),
    $$
    where $P_{a,r}(x,y)$ denotes the Poisson kernel of $\Delta_\omega$ on the ball $B_r(a)$. Furthermore, $M(u,a,r)$ decreases to $u(a)$ as $r\searrow0$. 
\end{lemma}

Now, we turn to show some basic facts of $m$-parabolic potentials, following \cite{GLZ21a}. Firstly, 
note that if $v\in P_m(\Omega_T)$ is bounded above in $\Omega_T$, then it can be extended to an upper semi-continuous function on $[0,T)\times\Omega$:
\begin{lemma}\label{lem:usc extension to 0}
    Let $u\in P_m(\Omega_T)$ be a $m$-parabolic potential, which is bounded from above on $\Omega_T=(0,T)\times\Omega$. For each $z\in\Omega$, set
    $$
u_0(z):=(\limsup_{t\to0^+}u_t)^*(z),
    $$
    where $*$ denotes the upper semi-continuous regularization. Then, $u_0$ is $(\omega,m)$-subharmonic in $\Omega$ and the extension $\tilde{u}:[0,T)\times\Omega\to[-\infty,+\infty)$ by defining $\tilde{u}(0,z):=u_0(z)$ is upper semi-continuous in $[0,T)\times\Omega$.
\end{lemma}
\begin{proof}
    Using \cref{SMI}, the proof is an easy adaptation of \cite[Lemma 1.6]{GLZ21a}. By \cite[Corollary 9.16]{GN18}, we have that $u_0$ is $\omega$-subharmonic on $\Omega$ and hence $(\omega,m)$-subharmonic by checking the current inequality. By \cref{prop:envelope basic properties}, it remains to show that $\tilde{u}$ is upper semi-continuous at some point $(0,z_0)$, which is an easy consequence of the sub-mean-value inequality \cref{SMI}. Since $u$ is bounded above, we may assume without loss of generality that $u\leq0$. Fix small positive numbers $\delta,r$ such that $\delta<r<T$ and $B(z_0,2r)\subset\Omega$. For each $t\in[0,\delta)$ and $z\in B(z_0,\delta)$, we apply \cref{SMI} to the slice $u_t$ on the ball $B(z,r+\delta)$ we get
    \begin{equation}
        u(t,z)\leq\frac{1}{r+\delta}\int_0^{r+\delta}ds\int_{\partial B(z,s)}u(t,\zeta)P_{z,s}(z,\zeta)d\sigma(\zeta).
    \end{equation}
Fatou's lemma then gives that
    $$
\limsup_{(t,z)\to(0,z_0)}u(t,z)\leq\frac{1}{r+\delta}\int_0^{r+\delta}ds\int_{\partial B(z,s)}u_0(\zeta)P_{z_0,s}(z_0,\zeta)d\sigma(\zeta).
    $$
    Letting $\delta\to0$ and then $r\to0$ we conclude by \cref{SMI} again that
    $$
\limsup_{(t,z)\to(0,z_0)}u(t,z)\leq u_0(z_0).
    $$
    The proof is thus finished.
\end{proof}

\subsection{Approximate submean value inequalities}
In this subsection we establish an interesting submean value inequality for $m$-parabolic potentials:

\begin{lemma}\label{approximate SMI}
    Let $u\in P_m(\Omega_T)$ be a $m$-parabolic potential. Fix
    $(t_0,z_0)\in\Omega_T$ and positive numbers $\varepsilon_0,r_0$ such that
    $[t_0-\varepsilon_0,t_0+\varepsilon_0]\times\bar{B}(z_0,r_0)
    \Subset\Omega_T$. Then, for each $0<\varepsilon\leq\varepsilon_0$
    and $0<r\leq r_0$, we have
    \[
    u(t_0,z_0)\leq
    \frac{1}{2}\int_{-1}^1 ds\,\frac{1}{r}\int_0^r w^{2n-1}dw
    \int_{\partial B_1(0)}
    u(t_0+\varepsilon s,z_0+w\xi)
    P_{z_0,w}(z_0,z_0+w\xi)d\sigma_1(\xi)
    +\kappa_0\varepsilon ,
    \]
    where $\kappa_0$ is the uniform Lipschitz constant of $u$ on
    $[t_0-\varepsilon_0,t_0+\varepsilon_0]\times\bar{B}(z_0,r_0)$.
\end{lemma}

\begin{proof}
    The proof is an easy adaptation of \cite[Lemma 1.8]{GLZ21a}. Applying
    \cref{SMI} to the $\omega$-subharmonic function $u(t_0,\cdot)$ (recall that $\Omega$ has real dimension $2n$), we get
    \begin{align*}
    u(t_0,z_0)
    &\leq
    \frac{1}{r}\int_0^r dw
    \int_{\partial B_w(z_0)}
    u(t_0,\zeta)P_{z_0,w}(z_0,\zeta)d\sigma_w(\zeta)\\
    &=
    \frac{1}{r}\int_0^r w^{2n-1}dw
    \int_{\partial B_1(0)}
    u(t_0,z_0+w\xi)
    P_{z_0,w}(z_0,z_0+w\xi)d\sigma_1(\xi).
    \end{align*}
    By the definition of $\kappa_0$, for every $-1\leq s\leq1$,
    \[
    |u(t_0,z_0+w\xi)-u(t_0+\varepsilon s,z_0+w\xi)|
    \leq \kappa_0\varepsilon |s|.
    \]
    Since
    \[
    \int_{\partial B_1(0)}
    w^{2n-1}P_{z_0,w}(z_0,z_0+w\xi)d\sigma_1(\xi)=1,
    \]
    we obtain, for every fixed $s\in[-1,1]$,
    \[
    u(t_0,z_0)\leq
    \frac{1}{r}\int_0^r w^{2n-1}dw
    \int_{\partial B_1(0)}
    u(t_0+\varepsilon s,z_0+w\xi)
    P_{z_0,w}(z_0,z_0+w\xi)d\sigma_1(\xi)
    +\kappa_0\varepsilon |s|.
    \]
    Averaging this inequality over $s\in[-1,1]$ gives the desired estimate, \(\frac{1}{2}\int_{-1}^1 |s|\,ds=\frac{1}{2}\leq1\).
\end{proof}

Here is a simple corollary:
\begin{corollary}\label{SMI and integrability}
    Let $u\in P_m(\Omega_T)$ be a $m$-parabolic potential, then for each $(t_0,z_0)\in\Omega_T$, we have
    $$
u(t_0,z_0)=\lim_{\varepsilon,r\to0}\frac{1}{2}\int_{-1}^1ds\frac{1}{r}\int_0^rw^{2n-1}dw\int_{\partial B_1(0)}u(t_0+\varepsilon s,z_0+w\xi)P_{z_0,w}(z_0,z_0+w\xi)d\sigma_1(\xi).
    $$
    In particular, if $u,v\in P_m(\Omega_T)$ and $u\leq v$ a.e. in $\Omega_T$, then $u\leq v$ everywhere.
\end{corollary}
\begin{proof}
One direction is immediate from \cref{approximate SMI}. The other direction is an easy consequence of \cref{SMI} and the uniform Lipschitz assumption of $u$.
\end{proof}

\subsection{Integrability and weak compactness}
$m$-parabolic potentials enjoy similar weak compactness as $m$-subharmonic potentials, following \cite[Proposition 1.17]{GLZ21a}:
\begin{proposition}\label{weak cptness}
    Let $\{u_j\}_j\in P_m(\Omega_T)$ be a sequence of $m$-parabolic potentials which is locally uniformly bounded from above in $\Omega_T$ and is locally uniformly Lipschitz. Then, either $u_j$ converges to $-\infty$ locally uniformly or that there exists a subsequence converging to some function $u\in P_m(\Omega_T)$ in the $L_{loc}^1(\Omega_T)$-topology.
\end{proposition}
\begin{proof}
Assume that the sequence $\{u_j\}$ does not converge locally uniformly to $-\infty$ in $\Omega_T$. We will first demonstrate that $\{u_j\}$ is bounded in $L_{loc}^1(\Omega_T)$. 
    
    Fix a relatively compact subinterval $J \subset (0, T)$ and a compact subset $K \subset \Omega$. By the locally uniformly Lipschitz property of the sequence $u_j$, for any fixed $t \in J$, the sequence of slices $\{u_j(t, \cdot)\}$ does not converge locally uniformly to $-\infty$ on $K$. Applying the compactness properties for $\omega$-subharmonic functions (see \cite[Lemma 9.12]{GN18}), we deduce that $\{u_j(t, \cdot)\}$ is bounded in $L_{loc}^1(\Omega)$. Because $\{u_j\}$ is locally uniformly Lipschitz in the time variable, it readily follows that the sequence is uniformly bounded in $L^1(J \times K)$, establishing that $\{u_j\}$ is bounded in $L_{loc}^1(\Omega_T)$.

    Next, let us consider the dense countable subset $\mathbb{Q} \cap (0, T)$. For each rational $r \in \mathbb{Q} \cap (0, T)$, the slices $\{u_j(r, \cdot)\}$ are locally uniformly bounded above, allowing us to extract a subsequence that converges in $L_{loc}^1(\Omega)$ to an $(\omega,m)$-subharmonic function $u(r, \cdot)$. By utilizing a standard Cantor diagonal process, we can extract a single subsequence, which we still denote as $\{u_j\}$, such that for every rational $r \in \mathbb{Q} \cap (0, T)$, $u_j(r, \cdot)$ converges in $L_{loc}^1(\Omega)$ to $u(r, \cdot)$.

    Since the functions $\{u_j\}$ share a local uniform Lipschitz bound in $t$, the limit function $(r, z) \mapsto u(r, z)$ inherits this identical locally uniform Lipschitz continuity with respect to $r$. Consequently, we can uniquely extend $u$ to the entire domain $(0, T) \times \Omega$ by defining $u(t, z) := \lim_{\mathbb{Q} \ni r \to t} u(r, z)$.

    The uniform Lipschitz property in time ensures that for all $t \in (0, T)$, the slice sequence $\{u_j(t, \cdot)\}$ converges in $L_{loc}^1(\Omega)$ to $u(t, \cdot)$, and that the extended function $u$ remains locally uniformly Lipschitz in $t \in (0, T)$. Because limits of $(\omega,m)$-subharmonic functions in $L_{loc}^1$ remain $(\omega,m)$-subharmonic, each slice $u(t, \cdot)$ belongs to this class, implying $u \in P_m(\Omega_T)$. Finally, an application of Fubini's theorem combined with the dominated convergence theorem guarantees that the subsequence $\{u_j\}$ converges to $u$ in $L_{loc}^1(\Omega_T)$.
\end{proof}
Now we turn to the integrability of $m$-parabolic potentials:

\begin{proposition}
    We have $P_m(\Omega_T)\subset L_{loc}^p(\Omega_T)$ for any $1\leq p<\frac{n}{n-m}$. Moreover, if $\{u_j\}_j\subset P_m(\Omega_T)$ converges in $L^1_{loc}(\Omega_T)$ to $u$, then it converges in $L^p_{loc}(\Omega_T)$ to $u$ for each $1<p<\frac{n}{n-m}$.
\end{proposition}
\begin{proof}
Let us fix a compact interval $J\subset(0,T)$ and a compact subset $K\subset\Omega$. We are going to show that $u\in L^p(K\times J)$ for $1\leq p<\frac{n}{n-m}$. For each fixed $t_0\in J$, $u(t_0,\cdot)$ is $(\omega,m)$-subharmonic and hence lies in $L^p(K\times J)$ by \cite{Fang25} (note that Fang proved similar integrability on compact Hermitian manifolds, the local case follows easily by an extension argument). We can then apply Fubini's theorem to write
    \begin{align*}
        \left(\int_{J\times K}|u(t,z)|^pdVdt\right)^{\frac{1}{p}}&\leq\left(\int_{J\times K}|u(t,z)-u(t_0,z)|^pdVdt\right)^{\frac{1}{p}}+\left(\int_{J\times K}|u(t_0,z)|^pdVdt\right)^{\frac{1}{p}}\\
        &\leq\kappa_J\left(\int_J|t-t_0|^pdV\right)^{\frac{1}{p}}+\left(\int_{J\times K}|u(t_0,z)|^pdVdt\right)^{\frac{1}{p}},
    \end{align*}
    where $\kappa_J$ is the Lipschitz constant of $u$ on $J$. This proves the first statement. The proof of the second statement is identical to that in \cite[Proposition 1.17]{GLZ21a}.
\end{proof}
The proof of the following slice estimate follows from \cite[Lemma 1.11]{GLZ21a} word by word:
\begin{lemma}\label{lem:L1_slice_estimate}
    Let $u, v \in P_m(\Omega_T)$ be two $m$-parabolic potentials. Fix time parameters $0 < \tau_1 < \tau_2 < S < T$. Then for any time slice $t \in [\tau_1, \tau_2]$, the $L^1(\Omega)$ norm of the difference satisfies the following bound:
    $$
    \|u_t - v_t\|_{L^1(\Omega)} \le 2 K \max \left\{ \|u - v\|_{L^1(\Omega_{\tau_2})}^{\frac{1}{2}}, \|u - v\|_{L^1(\Omega_{\tau_2})} \right\},
    $$
    where $u_t := u(t,\cdot)$, $v_t := v(t,\cdot)$, and the constant $K$ is defined as $K := \max \left\{ \sqrt{\kappa \operatorname{Vol}(\Omega)}, \frac{1}{S - \tau_2} \right\}$. Here, $\kappa > 0$ represents the uniform Lipschitz constant of the integral map $t \mapsto \int_{\Omega} |u(t,z) - v(t,z)| dV(z)$ defined on the interval $[\tau_1, S]$.
\end{lemma}

Finally, we show that locally bounded $m$-parabolic potentials naturally lie in the Sobolev space, as in \cite[Lemma 1.19]{GLZ21a}:
\begin{lemma}
\label{lem:m-parabolic-sobolev}
We have $P_m(\Omega_T)\cap L_{loc}^\infty(\Omega_T)\subset W_{loc}^{1,1}(\Omega_T)$.
\end{lemma}

\begin{proof}
Fix a compact subset $J \times K \Subset ]0,T[ \times \Omega$. 
Since $u$ is locally bounded, there exists a constant $M > 0$ such that 
$|u(t,z)| \le M$ for all $(t,z) \in J \times K$. 
Moreover, by the locally uniform Lipschitz condition in time, there is a constant 
$L_J > 0$ such that $|\partial_t u(t,z)| \le L_J$ for almost every $(t,z) \in J \times \Omega$.

For each fixed $t \in J$, the slice $u_t$ is a locally bounded $(\omega,m)$-subharmonic function 
on $\Omega$ with $\|u_t\|_{L^\infty(K)} \le M$. 
By \cite[Lemma 3.10]{KN26}, $u_t$ belongs to the Sobolev space $W^{1,2}_{\mathrm{loc}}(\Omega)$; 
in fact, the proof of \cite[Lemma 3.10]{KN26} provides a uniform estimate 
\[
\int_K |\nabla_z u(t,z)|^2 \, dV(z) \le C,
\]
by using a smooth decreasing approximation and their Chern-Levine-Nirenberg inequalities \cite[Proposition 3.1]{KN26} along with the Mazur lemma in functional analysis, where the constant $C$ depends only on $M$, $K$, $\Omega$ and the Hermitian metric $\omega$, but not on the particular $t \in J$. 
Consequently, $\nabla_z u \in L^2(J \times K) \subset L^1(J \times K)$ by the Cauchy--Schwarz inequality.

Combining the spatial gradient bound with the boundedness of $\partial_t u$, we obtain
\[
\int_J \int_K \bigl( |\partial_t u(t,z)| + |\nabla_z u(t,z)| \bigr) \, dV(z) \, dt 
\le |J| \cdot |K| \cdot L_J + |J|^{1/2} \, |K|^{1/2} \, C^{1/2} < \infty.
\]
Thus the full gradient $\nabla_{t,z} u$ belongs to $L^1_{\mathrm{loc}}(\Omega_T)$, 
which means exactly that $u \in W^{1,1}_{\mathrm{loc}}(\Omega_T)$.
\end{proof}
This regularity result will be crucial in showing that pluripotential subsolutions are closed under finite maxima.

\subsection{Parabolic Hessian operator}
For each $u\in P_m(\Omega_T)\cap L_{loc}^\infty(\Omega_T)$, the parabolic Hessian operator of $u$ on $\Omega_T$ is formally defined to be $$
(dd^cu)^m\wedge\omega^{n-m}\wedge dt.
$$
More precisely, we give the following definition:
\begin{definition}
    Let $u\in P_m(\Omega_T)\cap L_{loc}^\infty(\Omega_T)$ be a locally bounded $m$-parabolic potential. For each test function $\varphi\in C_c^0(\Omega_T)$, we define
    $$
\langle(dd^cu)^m\wedge\omega^{n-m}\wedge dt,\varphi\rangle:=\int_0^Tdt\int_\Omega\varphi(t,\cdot)(dd^cu_t)^m\wedge\omega^{n-m},
    $$
    where $u_t:=u(t,\cdot)\in\operatorname{SH}_m(\Omega,\omega)$ is the $m$-subharmonic slice for $u$. This is a well-defined positive Radon measure in $\Omega_T$ by Riesz's representation theorem.
\end{definition}

Using the Chern-Levine-Nirenberg inequality for Hessian operators (cf. \cite[Proposition 3.7]{KN26}), the proof of the following parabolic Chern-Levine-Nirenberg inequality follows from that in \cite[Lemma 2.1]{GLZ21a} verbatim:

\begin{proposition}\label{CLN}
    Let $u\in P_m(\Omega_T)\cap L_{loc}^\infty(\Omega_T)$ be a locally bounded $m$-parabolic potential and fix a continuous test function $\chi\in C_c^0(\Omega_T)$. Let $1\leq p\leq m$ be an integer. Then the function
    $$
    \Gamma_{\chi}:t\mapsto\int_{\Omega}\chi(t,\cdot)(dd^cu_t)^p\wedge\omega^{n-p}
    $$
    is continuous in $(0, T)$. Moreover, if $\operatorname{Supp}\chi \subset E_1 \Subset E_2 \subset\Omega_T$, then
    \begin{equation}
        \sup_{0 < t < T}\left|\int_\Omega\chi(t,\cdot)(dd^cu_t)^p\wedge\omega^{n-p}\right|\leq C\max_{\Omega_T}|\chi|\cdot \left(1+\max_{E_2}|u|\right)^p,
    \end{equation}
    where $C>0$ is a constant depending only on $E_1, E_2, \Omega_T$, and $\omega$.
\end{proposition}

Similar to \cite[Proposition 2.3]{GLZ21a}, we can use the monotone convergence in \cite{KN26} to obtain the monotone convergence of parabolic Hessian measures:

\begin{proposition}\label{prop:measure_convergence}
    Let $u \in P_m(\Omega_T) \cap L_{loc}^\infty(\Omega_T)$ and let $(u^j)$ be a monotone sequence of functions in $P_m(\Omega_T) \cap L_{loc}^\infty(\Omega_T)$ converging to $u$ almost everywhere in $\Omega_T$. Then for any integer $1 \le p \le m$, we have
    $$
    dt \wedge (dd^c u^j)^p \wedge \omega^{n-p} \to dt \wedge (dd^c u)^p \wedge \omega^{n-p}
    $$
    in the weak sense of Radon measures on $\Omega_T$.
\end{proposition}

\begin{proof}
    Let $\chi \in C_c^0(\Omega_T)$ be a continuous test function. By definition, evaluating the measure against the test function yields:
    $$
    \int_{\Omega_T} \chi dt \wedge (dd^c u^j)^p \wedge \omega^{n-p} = \int_0^T dt \left( \int_\Omega \chi(t, \cdot) (dd^c u^j_t)^p \wedge \omega^{n-p} \right) := \int_0^T F_j(t) dt.
    $$
    Since the sequence $(u^j)$ is monotone and converges to $u$ almost everywhere, for almost every $t \in (0,T)$, the slice sequence $(u^j_t)$ is a monotone sequence of $m$-subharmonic functions converging almost everywhere to $u_t$. It follows from the fundamental convergence theorem for complex Hessian measures (cf. \cite[Theorem 3.4]{KN26}) that $F_j(t)$ converges pointwise to $F(t) := \int_\Omega \chi(t, \cdot) (dd^c u_t)^p \wedge \omega^{n-p}$ for almost every $t \in (0,T)$.

    Furthermore, since the sequence $(u^j)$ converges monotonically to $u$, it is uniformly bounded on any compact subset $E_2 \Subset \Omega_T$ containing the support of $\chi$. The parabolic Chern-Levine-Nirenberg inequality established in \cref{CLN} guarantees that the sequence of functions $F_j(t)$ is uniformly bounded by a constant independent of $j$ and $t$. 
    
    We can therefore apply the Lebesgue Dominated Convergence Theorem to conclude that
    $$
    \lim_{j \to \infty} \int_0^T F_j(t) dt = \int_0^T F(t) dt,
    $$
    which completes the proof.
\end{proof}

\subsection{Key convergence theorem}

We first briefly recall some fundamental properties regarding the time derivatives of $m$-parabolic potentials. These properties are essential for properly making sense of the parabolic complex Hessian equation on time slices. We refer the reader to \cite[Section 1.3]{GLZ21a} for detailed proofs in the Monge-Amp\`ere setting, which apply to our case verbatim.

For a general $m$-parabolic potential $u \in P_m(\Omega_T)$, the locally uniform Lipschitz condition in time ensures that the classical time derivative $\partial_t u(t,z)$ exists for almost all $(t,z) \in \Omega_T$ with respect to the Lebesgue measure. However, when studying the Perron envelope and applying the comparison principle, we frequently encounter functions with additional semi-concavity properties.

\begin{definition}\label{def:semi-concave}
    A function $u: \Omega_T \to \mathbb{R}$ is said to be \emph{locally uniformly semi-concave} in $t \in (0, T)$ if for any compact subinterval $J \Subset (0, T)$, there exists a constant $\kappa = \kappa(J, u) > 0$ such that for all $z \in \Omega$, the function $t \mapsto u(t,z) - \kappa t^2$ is concave on $J$.
\end{definition}

For a bounded $m$-parabolic potential $u$ that is locally uniformly semi-concave in time, the one-sided time derivatives
$$
\partial_t^+ u(t,z) := \lim_{s \to 0^+} \frac{u(t+s, z) - u(t,z)}{s}, \quad \text{and} \quad \partial_t^- u(t,z) := \lim_{s \to 0^-} \frac{u(t+s, z) - u(t,z)}{s}
$$
exist for all $t \in (0,T)$ and $z \in \Omega$. Furthermore, since concavity naturally interacts with semi-continuity, these one-sided derivatives enjoy the following crucial topological properties:

\begin{lemma}[{\cite[Lemma 1.15]{GLZ21a}}]\label{lem:time_derivatives_semicontinuity}
    Let $u: \Omega_T \to \mathbb{R}$ be a continuous function which is locally uniformly semi-concave in $t \in (0, T)$. Then the left derivative $(t,z) \mapsto \partial_t^- u(t,z)$ is upper semi-continuous, while the right derivative $(t,z) \mapsto \partial_t^+ u(t,z)$ is lower semi-continuous on $\Omega_T$. 
    
    In particular, there exists a Borel set $E \subset \Omega_T$ of Lebesgue measure zero such that $\partial_t^+ u$ and $\partial_t^- u$ coincide and are continuous at each point in $\Omega_T \setminus E$.
\end{lemma}

This structural regularity is of great importance: the lower semi-continuity of $\partial_t^+ u$ allows us to apply Fatou's lemma when passing to weak limits of Monge-Amp\`ere or Hessian measures, thereby enabling us to define pluripotential subsolutions on each individual time slice, bypassing the negligible set where the classical derivative might not exist.

The following convergence is of key importance in the sequel, whose proof can be carried out word by word as in \cite[Proposition 2.9]{GLZ21a}.

\begin{proposition}\label{key convergence}
    Let $f_j$ be a sequence of positive functions converging to
$f$ in $L^1(\Omega_T)$. Let $\phi_j$ be a sequence of functions in $P_m(\Omega_T)$ which is locally uniformly (w.r.t $j$) semi-concave in $(0,T)$ and such that $\phi_j\to\phi\in P_m(\Omega_T)$ almost everywhere. Then, we have 
$$
\lim_{j\to\infty}\dot{\phi}_j(t,x)=\dot{\phi}(t,x)
$$
for almost all $(t,x)\in \Omega_T$ and 
$$
\theta(\dot{\phi}_j)f_j\to\theta(\dot{\phi})f
$$
in the sense of distributions for all $\theta\in C^0(\mathbb{R},\mathbb{R})$.
\end{proposition}
\section{Continuous pluripotential solutions on a Euclidean ball}
In this section, we prove \cref{thm:main_pluripotential} on a unit ball by constructing smooth Cauchy-Dirichlet boundary data approximations on the parabolic boundary $\partial_0\mathbb{B}_T$.

\subsection{Time a priori estimates}

The following a priori estimates of the locally uniformly Lipschitz and semi-concavity properties are of key importance and are inspired by \cite[Theorem 4.2, Theorem 4.7]{GLZ21a}:

\begin{lemma}\label{lem:a priori lipschitz}
    Let $\Omega \subset \mathbb{C}^n$ be a bounded, strictly $m$-pseudoconvex domain with smooth boundary, and let $T > 0$. Suppose $u \in C^\infty(\overline{\Omega_T})$ is a smooth admissible solution to the parabolic complex Hessian equation:
    $$
    dt \wedge (dd^c u)^m \wedge \omega^{n-m} = e^{\partial_t u + \psi(t,z)} g(z) dt \wedge dV,
    $$
    where $g > 0$ is a continuous density on $\overline{\Omega}$, and $\psi(t,z)$ is a smooth function on $\Omega_T$. Let $h := u|_{\partial_0 \Omega_T}$ be the smooth Cauchy-Dirichlet boundary data. Assume that there exist uniform constants $M, \kappa_h, \kappa_\psi > 0$ such that: $\sup_{\Omega_T} |u| \le M$; $t|\partial_t h(t,z)| \le \kappa_h$ for all $(t,z) \in \partial_0 \Omega_T$; $|\psi(t_1, z) - \psi(t_2, z)| \le \kappa_\psi |t_1 - t_2|$ for all $(t_1, z), (t_2, z) \in \Omega_T$. Then there exists a uniform constant 
    $$
    \kappa_u:=(2M + 2\kappa_h + 2m + \kappa_\psi T)(T+1)+M 
    $$
    such that:
    $$
    t|\partial_t u(t,z)| \le \kappa_u, \quad \forall (t,z) \in \Omega_T.
    $$
\end{lemma}
\begin{proof}
   As in \cite{GLZ21a}, the uniform Lipschitz assumption $t|\partial_t h(t,z)| \le \kappa_h$ implies that for any $s > 0$ such that $st < T$, we have the estimate:
   $$
   |h(t,z) - h(st,z)| \le \kappa_h \frac{|s-1|}{\min(s,1)}.
   $$
Fix an arbitrary $0 < S < T$ and let $s \ge 1/2$ be close enough to $1$ such that $sS < T$. We define a scaled and shifted comparison function on the sub-cylinder $\Omega_S$:$$v^s(t,z) := s^{-1}u(st,z) - C|s-1|(t+1),$$where $C > 0$ is a large constant to be determined. We claim that $v^s$ is a subsolution to the parabolic Hessian equation on $\Omega_S$. Since $u$ is an exact smooth solution on $\Omega_T$, evaluating the Hessian operator on $v^s$ yields:
$$
(dd^c v^s(t,\cdot))^m \wedge \omega^{n-m} = s^{-m}(dd^c u(st,\cdot))^m \wedge \omega^{n-m}= s^{-m} e^{\partial_\tau u(st,z) + \psi(st,z)} g(z) dV.
$$
To ensure that $v^s$ is a pluripotential subsolution, we require:
$$
s^{-m} e^{\partial_\tau u(st,z) + \psi(st,z)} g(z) dV \ge e^{\partial_t v^s(t,z) + \psi(t,z)} g(z) dV.
$$
Observing that $\partial_t v^s(t,z) = \partial_\tau u(st,z) - C|s-1|$, we can take the logarithm of both sides to see that the subsolution condition is equivalent to:
$$
-m \log s + \psi(st,z) \ge -C|s-1| + \psi(t,z).
$$
By the uniform Lipschitz property of $\psi$, we have $|\psi(t,z) - \psi(st,z)| \le \kappa_\psi t |s-1| \le \kappa_\psi T |s-1|$. Furthermore, for $s \in [1/2, 3/2]$, elementary calculus gives $-m \log s \ge -2m|s-1|$. Consequently, the inequality holds as long as we choose $C$ such that:$$C \ge 2m + \kappa_\psi T.$$Next, we verify the boundary conditions on the parabolic boundary $\partial_0 \Omega_S$. For any $t \in [0,S]$ and $z \in \partial \Omega$, we have:
\begin{align*}
v^s(t,z) &= s^{-1}h(st,z) - C|s-1|(t+1)\le h(st,z) + |s^{-1}-1| M - C|s-1|\\
&\le h(t,z) + 2\kappa_h |s-1| + 2M |s-1| - C|s-1| \le h(t,z),
\end{align*}
where we used the inequality $|s^{-1}-1| \le 2|s-1|$ for $s \ge 1/2$ and we need $C\geq 2\kappa_h+2M$. Similarly, on the bottom boundary $\{0\} \times \Omega$, we trivially have $v^s(0,z) \le h(0,z)$ provided $C \ge 2M$. Thus, setting $C := 2M + 2\kappa_h + 2m + \kappa_\psi T$ guarantees that $v^s$ is a smooth subsolution on $\Omega_S$ with $v^s \le u$ on $\partial_0 \Omega_S$. Since $u$ is a smooth solution, the smooth parabolic comparison principle (see also \cref{lem:comparison_C1} below) implies $v^s \le u$ everywhere in $\Omega_S$. This yields the point-wise inequality:
$$
s^{-1}u(st,z) - C|s-1|(t+1) \le u(t,z), \quad \forall (t,z) \in \Omega_S.
$$
Letting $s \to 1$ from above ($s > 1$) and below ($s < 1$), we deduce the differential inequality:
$$
|\partial_t u(t,z) \cdot t - u(t,z)| \le C(t+1).
$$
Rearranging this, we obtain:
$$
t|\partial_t u(t,z)| \le M + C(T+1) =: \kappa_u.
$$
Since $M, \kappa_h$, and $\kappa_\psi$ are uniform bounds independent of $S$, the constant $\kappa_u$ is uniform. Letting $S \to T$ concludes the proof.
\end{proof}
\begin{lemma}\label{lem:a priori semi-concave}
    Let $\Omega \subset \mathbb{C}^n$ be a bounded, strictly $m$-pseudoconvex domain with smooth boundary, and let $T > 0$. Suppose $u \in C^\infty(\overline{\Omega_T})$ is a smooth admissible solution to the parabolic complex Hessian equation:
    $$
    dt \wedge (dd^c u)^m \wedge \omega^{n-m} = e^{\partial_t u + \psi(t,z)} g(z) dt \wedge dV.
    $$
    Assume that $u$ satisfies the a priori bounds $\sup_{\Omega_T} |u| \le M$ and $t|\partial_t u| \le \kappa$. Furthermore, assume that the smooth boundary data $h := u|_{\partial_0 \Omega_T}$ and the function $\psi$ satisfy the following semi-concavity conditions in time: there exist uniform constants $C_h, C_\psi > 0$ such that for all $(t,z) \in \Omega_T$,
    $$
    t^2 \partial_t^2 h(t,z) \le C_h, \quad \text{and} \quad |\partial_t^2 \psi(t,z)| \le C_\psi.
    $$Then there exists a uniform constant
$$
K := \max \Big\{ C_\psi + \kappa_\psi, \, C_h + \kappa_h + M \Big\} (T+1) + \kappa + M,
$$
such that:
    $$
    t^2 \partial_t^2 u(t,z) \le K, \quad \forall (t,z) \in \Omega_T.
    $$
\end{lemma}
\begin{proof}
    Fix an arbitrary $0 < S < T$ and a scaling parameter $s > 1/2$ sufficiently close to $1$ such that $sS < T$. We construct a barrier by taking the arithmetic mean of symmetrically scaled solutions. For $(t,z) \in \Omega_S$, we define:$$v^s(t,z) := \frac{s^{-1}u(st,z) + s u(s^{-1}t,z)}{2} - C(t+1)(s-1)^2,$$where $C > 0$ is a large uniform constant to be chosen later. We claim that $v^s$ acts as a subsolution to the parabolic complex Hessian equation on $\Omega_S$.First, we estimate the complex Hessian measure of $v^s$. Setting $w_1 = s^{-1}u(st,z)$ and $w_2 = s u(s^{-1}t,z)$, we evaluate the operators:
    \begin{align*}
    (dd^c w_1)^m \wedge \omega^{n-m} &= s^{-m}(dd^c u(st,\cdot))^m \wedge \omega^{n-m}= s^{-m} e^{\partial_\tau u(st,z) + \psi(st,z)} g(z) dV
    \end{align*}
    and
    $$
(dd^c w_2)^m \wedge \omega^{n-m} = s^m(dd^c u(s^{-1}t,\cdot))^m \wedge \omega^{n-m} 
    = s^m e^{\partial_\tau u(s^{-1}t,z) + \psi(s^{-1}t,z)} g(z) dV.
    $$
   The mixed Hessian inequality (cf, \cite{Sun25}) then yields that
   $$
   (dd^c v^s)^m \wedge \omega^{n-m} \ge \exp\left( \frac{\partial_\tau u(st,z) + \partial_\tau u(s^{-1}t,z)}{2} + \frac{\psi(st,z) + \psi(s^{-1}t,z)}{2} \right) g(z) dV.
   $$
Since $\partial_t v^s(t,z) = \frac{1}{2}\big[\partial_\tau u(st,z) + \partial_\tau u(s^{-1}t,z)\big] - C(s-1)^2$, comparing the exponents reduces the subsolution condition to:
    $$
    \frac{\psi(st,z) + \psi(s^{-1}t,z)}{2} \ge \psi(t,z) - C(s-1)^2.
    $$
The semi-convex bound $|\partial_t^2 \psi| \le C_\psi$ on the compact domain $\overline{\Omega_T}$ ensures that $\big|\frac{\psi(st) + \psi(s^{-1}t)}{2} - \psi(t)\big| \le 2C_\psi T^2(s-1)^2$.

Next, we verify the boundary condition $v^s \le h$ on the parabolic boundary $\partial_0 \Omega_S$. On the side boundary $\partial_s \Omega_S$, the assumption $t^2 |\partial_t^2 h| \le C_h$ implies $\frac{1}{2}[h(st) + h(s^{-1}t)] \le h(t) + \tilde{C}_h(s-1)^2$. Writing $s = 1+\sigma$ and $s^{-1} \le 1 - \sigma + 2\sigma^2$ (for $s$ close to 1), straightforward algebraic expansion combined with the Lipschitz bound $t|\partial_t h| \le \kappa_h$ and the uniform bound $\sup|u| \le M$ yields:
    \begin{align*}
    v^s(t,z) &\le \frac{1}{2}\big[s^{-1}h(st,z) + s h(s^{-1}t,z)\big] - C(s-1)^2 \\
    &\le \frac{h(st,z) + h(s^{-1}t,z)}{2} + (M + 2\kappa_h)(s-1)^2 - C(s-1)^2 \\
    &\le h(t,z) + (\tilde{C}_h + M + 2\kappa_h - C)(s-1)^2.
    \end{align*}
    By selecting $C \ge \tilde{C}_h + M + 2\kappa_h$, we guarantee that $v^s(t,z) \le h(t,z)$ on $\partial_0 \Omega_S$. The comparison principle then forces $v^s \le u$ everywhere in $\Omega_S$. This implies that for all $(t,z) \in \Omega_S$:
    $$
    \frac{s^{-1}u(st,z) + s u(s^{-1}t,z)}{2} - u(t,z) \le C(t+1)(s-1)^2.
    $$
   To extract the bound for the second time derivative, we apply a second-order Taylor expansion to the left-hand side of the above inequality. Let $s = 1 + \sigma$, then $s^{-1} = 1 - \sigma + \sigma^2 + O(\sigma^3)$. Expanding $u(st)$ and $u(s^{-1}t)$ around $t$, we have:
\begin{align*}
s^{-1}u(st) &= (1 - \sigma + \sigma^2)\left[ u + \sigma t \partial_t u + \frac{1}{2}\sigma^2 t^2 \partial_t^2 u + O(\sigma^3) \right] \\
&= u + \sigma(t\partial_t u - u) + \sigma^2\left( u - t\partial_t u + \frac{1}{2}t^2\partial_t^2 u \right) + O(\sigma^3),
\end{align*}
and similarly for the second term:
\begin{align*}
s u(s^{-1}t) &= (1 + \sigma)\left[ u + (-\sigma + \sigma^2)t \partial_t u + \frac{1}{2}\sigma^2 t^2 \partial_t^2 u + O(\sigma^3) \right] \\
&= u + \sigma(u - t\partial_t u) + \sigma^2\left( \frac{1}{2}t^2\partial_t^2 u \right) + O(\sigma^3).
\end{align*}
Summing these two expansions and dividing by $2$, we see that the $O(\sigma)$ terms exactly cancel each other out:
$$
\frac{s^{-1}u(st) + s u(s^{-1}t)}{2} = u + \frac{\sigma^2}{2} \left( t^2\partial_t^2 u - t\partial_t u + u \right) + O(\sigma^3).
$$
Substituting this back into our comparison principle estimate and subtracting $u(t,z)$, we obtain:
$$
\frac{\sigma^2}{2} \left( t^2\partial_t^2 u - t\partial_t u + u \right) + O(\sigma^3) \le C(t+1)\sigma^2.
$$
Dividing by $\sigma^2$ and taking the limit as $\sigma \to 0$ (i.e., $s \to 1$), we arrive at:
$$
\frac{1}{2} t^2 \partial_t^2 u(t,z) - \frac{1}{2} t \partial_t u(t,z) + \frac{1}{2} u(t,z) \le C(t+1).
$$
Multiplying by 2 and rearranging the terms yields:
$$
t^2 \partial_t^2 u(t,z) \le 2C(T+1) + t \partial_t u(t,z) - u(t,z).
$$
Finally, we invoke our assumed a priori bounds $\sup |u| \le M$ and $t|\partial_t u| \le \kappa$. This provides the desired uniform upper bound:
$$
t^2 \partial_t^2 u(t,z) \le 2C(T+1) + \kappa + M := K.
$$
This completes the proof.
\end{proof}

\subsection{continuous solutions}
We are now  ready to establish an existence result analogous to that of \cite[Proposition 6.2]{GLZ21a}:
\begin{proposition}\label{continuous boundary data solution on balls}
    Let $h\in C^0(\partial_0\mathbb{B}_T)$ be a continuous Cauchy-Dirichlet boundary data with $T<+\infty$. Assume moreover that there is a uniform constant $\kappa_h$ such that
    \begin{equation}\label{h Lip condition 1}
|\partial_th|,\,\partial_t^2h\leq\kappa_h,\quad t\in(0,T]
    \end{equation}
    in the sense of distributions. In other words, $h$ is uniformly Lipschitz and locally uniformly semi-concave on $[0,T]$. Furthermore, $\psi(t,z)$ is a bounded continuous function on $\overline{\mathbb{B}_T}$ such that $\psi$ is uniformly Lipschitz and uniformly semi-convex in $t$, and $g\in C^0(\overline{\mathbb{B}})$ is a continuous positive density. Then, there is a continuous parabolic $m$-potential $u\in P_{m,\omega}(\mathbb{B}_T)\cap C^0(\overline{\mathbb{B}}_T)$ such that
    $$
dt\wedge(dd^cu)^m\wedge\omega^{n-m}=e^{\partial_t u+\psi(t,z)}g(z)dt\wedge dV_{\mathbb{B}},
    $$
    on $\mathbb{B}_T$ in the weak sense of Radon measures.
\end{proposition}
\begin{proof}
By standard convolutions, we can find smooth approximations $\psi_j,g_j$
of $\psi$ and $g$ respectively.

\medskip
\noindent
    \textbf{Step 1.} We are going to construct smooth boundary data $h_j$ that approximates $h$ on $\partial_0\mathbb{B}_T$ uniformly. Firstly, we use radial extension $\tilde{h}(z,t):=h\left(\frac{z}{|z|},t\right)$ to extend $h$ to $\{\frac{1}{2}\leq |z|\leq\frac{3}{2}\}\times[0,T]$. By constant extension, we can assume that $\tilde{h}$ is a continuous function defined on $\{\frac{1}{2}\leq |z|\leq\frac{3}{2}\}\times[-1,T+1]$. Then, we use convolutions to obtain smooth functions $h_j$ defined in a neighborhood of $\partial_0\mathbb{B}_T$ such that $h_j\to h$ uniformly on $\partial_0\mathbb{B}_T$. It is then clear that $h_j$ satisfies \eqref{h Lip condition 1} with a (probably larger) uniform constant $\kappa$ in the pointwise sense. At this stage, $h_j|_{\mathbb{B}\times\{0\}}$ may not lie in $\operatorname{SH}_m(\mathbb{B},\omega)$, we modify $h_j$ on $\partial_0\mathbb{B}_T$ as follows. 
    
    Since $h_0$ is a continuous $(\omega,m)$-subharmonic function near $\overline{\mathbb{B}}$, we can apply \cite[Proposition 2.9]{KN26} to find a sequence of smooth $(\omega,m)$-subharmonic functions $\varphi_j$ decreasing (hence converging uniformly) to $h_0$ on $\overline{\mathbb{B}}$, we then consider the following Dirichlet problem of Hessian equations:
    \begin{equation}\label{dirichlet problem}
        \left\{
    \begin{aligned}        
&(dd^cw_j)^m\wedge\omega^{n-m}=\frac{1}{j}\omega^n, \quad\operatorname{on}\mathbb{B},\\        
&w_j=h_j-\varphi_j\quad\partial{\mathbb{B}}.    
\end{aligned}
\right.
    \end{equation}
    By standard stability result (see \cite[Proposition 8.1]{KN26}) we deduce that $w_j\to0$ uniformly on $\mathbb{B}$. Set 
    $$
 \left\{    
\begin{aligned}        
&\hat{h_j}=h_j,\operatorname{on}\partial\mathbb{B}\times(0,T],\\        
&\hat{h}_j=\varphi_j+w_j\quad\partial{\mathbb{B}}\times\{0\}.    
\end{aligned}
\right.
    $$
    It is clear that $\hat{h}_j$ is smooth and $0th$-order compatible at the corner $\partial\mathbb{B}\times\{0\}$, which converges uniformly to $h$ on $\partial_0\mathbb{B}_T$. We next modify $\hat{h}_j$ to ensure that these approximations are at least second-order compatible at $\partial\mathbb{B}\times\{0\}$, in order to use \cref{intro:main_smooth_solution 2} (this modification is just for technical convenience and is not essential). Set
    $$
v_{j,1}(z) := \log \frac{(dd^c \hat{h}_{j,0})^m \wedge \omega^{n-m}}{g_j(z)\omega^n} - \psi_j(0,z)
    $$
    and
$$
v_{j,2}(z) := F^{p\bar{q}}(dd^c \hat{h}_{j,0}) \cdot (v_{j,1})_{p\bar{q}} - \partial_t \psi_j(0,z).
$$
We then define the difference terms to be
$$
D_{j,1}(z) := v_{j,1}(z) - \partial_t \hat{h}_j(z,0),\quad D_{j,2}(z) := v_{j,2}(z) - \partial_{tt} \hat{h}_j(z,0),
$$
and, moreover,
$$
P_j(z,t) := t \cdot D_{j,1}(z) + \frac{t^2}{2} \cdot D_{j,2}(z).
$$
Let $\eta:[0,\infty)\to[0,1]$ be a standard cutoff function such that $\eta|_{[0,1/2]}\equiv1$ and $\eta$ vanishes on $[1,\infty)$. The final modification of $\hat{h}_j$ will be defined as
$$
H_j(z,t) := \hat{h}_j(z,t) + \eta\left(\frac{t}{\delta_j}\right) P_j(z,t),
$$
where $\delta_j\to0$ are small constants to be chosen later. It is easy to check that $H_j$ are smooth boundary data that are second-order compatible at the corner $\partial\mathbb{B}\times\{0\}$. Finally, we must choose $\delta_j > 0$ sufficiently small to ensure that this modification does not destroy the uniform bounds on the scaled time derivatives $t|\partial_t H_j|$ and $t^2|\partial_{tt} H_j|$. Since $P_j(z,t) = t D_{j,1}(z) + \frac{t^2}{2} D_{j,2}(z)$ and the cutoff function $\eta(t/\delta_j)$ is supported on $t \in [0, \delta_j]$, we can estimate the contribution of the modification to the first scaled time derivative:
    \begin{align*}
        t \partial_t \left( \eta\left(\frac{t}{\delta_j}\right) P_j(z,t) \right) &= \frac{t}{\delta_j} \eta'\left(\frac{t}{\delta_j}\right) P_j(z,t) + t \eta\left(\frac{t}{\delta_j}\right) \partial_t P_j(z,t) \\
        &\le C \frac{t}{\delta_j} \left( t |D_{j,1}| + t^2 |D_{j,2}| \right) + C t \left(|D_{j,1}| + t |D_{j,2}|\right) \\
        &\le C \left( \delta_j |D_{j,1}| + \delta_j^2 |D_{j,2}| \right),
    \end{align*}
    where we crucially used the fact that $t \le \delta_j$ on the support of $\eta$. A similar calculation for the second scaled derivative yields:
    \begin{align*}
        t^2 \partial_{tt} \left( \eta\left(\frac{t}{\delta_j}\right) P_j(z,t) \right) &= \left(\frac{t}{\delta_j}\right)^2 \eta'' P_j + 2\frac{t}{\delta_j} t \eta' \partial_t P_j + t^2 \eta \partial_{tt} P_j \\
        &\le C \left( \delta_j |D_{j,1}| + \delta_j^2 |D_{j,2}| \right).
    \end{align*}
    Although the compatibility correction terms $D_{j,1}$ and $D_{j,2}$ may be very large and depend on $j$ (as they involve the linearized operator acting on the approximation $\varphi_j$), they are bounded by constants for each fixed $j$. Therefore, by choosing $\delta_j>0$ sufficiently small such that $C \left( \delta_j |D_{j,1}| + \delta_j^2 |D_{j,2}| \right) \le \frac{1}{j}$, the uniform conditions $t|\partial_t H_j| \le \kappa_h + \frac{1}{j}$ and $t^2|\partial_{tt} H_j| \le C_h + \frac{1}{j}$ are perfectly preserved. Moreover, the modification itself satisfies $|\eta(t/\delta_j) P_j| \le C(\delta_j |D_{j,1}| + \delta_j^2 |D_{j,2}|) \le \frac{1}{j}$, ensuring $H_j \to h$ uniformly on $\partial_0\mathbb{B}_T$. Note also that $H_j(\cdot, 0) = \varphi_j + w_j$ is strictly $(\omega,m)$-subharmonic since $(dd^c w_j)^m \wedge \omega^{n-m} = \frac{1}{j}\omega^n > 0$. Thus, $H_j$ fully satisfies all the assumptions required for the smooth flow in \cref{intro:main_smooth_solution 2}.
    
\medskip
\noindent
\textbf{Step 2.} Using \cref{intro:main_smooth_solution 2}, we can find smooth solutions $u_j\in C^\infty\left(\overline{\mathbb{B}}\times(0,T]\right)\cap C^{4,2}(\overline{\mathbb{B}_T})$ such that
\begin{equation}\label{eq:hessian appro sequence}
\left\{    
\begin{aligned}    
dt\wedge(dd^cu_j)^m\wedge\omega^{n-m}=e^{\partial_t u_j+\psi_j(t,z)}g_j(z)dt\wedge dV_{\mathbb{B}} \quad\operatorname{on}\mathbb{B}_T,\\    
u_j=H_j\quad\operatorname{on}\partial_0\mathbb{B}_T.    
\end{aligned}
\right.
\end{equation}
By \cref{lem:a priori lipschitz} and \cref{lem:a priori semi-concave}, we have that the sequence $\{u_j\}$ is locally uniformly Lipschitz and locally uniformly semi-concave on $(0,T]$. Here, we shall note that the explicit constants calculated in \cref{lem:a priori lipschitz} and \cref{lem:a priori semi-concave} depends only on the uniform bounds of the boundary values $H_j$ of $u_j$, so a priori we do not need a uniform bound of $u_j$. We then claim that $u_j$ is uniformly bounded on $\overline{\mathbb{B}}_T$. Indeed, on each slice $\overline{\mathbb{B}}\times\{t\}$, $u_j(t,\cdot)$ is $(\omega,m)$-subharmonic and hence $u_j(t,\cdot)\leq H_j(t,\cdot)$. Consequently, $u_j\le H_j\leq C$. For the lower bound, set
$$
v(t,z):=A\rho(z)-\sup_{\partial{\mathbb{B}}_T}|h|,
$$
where $\rho \le 0$ is a bounded continuous $(\omega, m)$-subharmonic function on $\Omega$ solving $(dd^c \rho)^m \wedge \omega^{n-m} = g dV$ with zero boundary values. It is clear that we have
$$
dt\wedge(dd^cv)^m\wedge\omega^{n-m}\geq e^{\partial_t v+\psi_j(t,z)}g_j(z)dt\wedge dV_{\mathbb{B}}=e^{\psi_j(t,z)}g_j(z)dt\wedge dV_{\mathbb{B}},
$$
for $A$ large enough. The parabolic comparison principle \cref{lem:comparison_C1} (whose proof is independent from here) below then yields that $v\leq u_j$, whence the uniform lower bound of $u_j$.

Now, invoking \cref{weak cptness} we can find $u\in P_{m,\omega}(\mathbb{B}_T)$ such that $u_j\to u$ in $L^1(\mathbb{B}_T)$. From the discussions above, all the conditions in \cref{key convergence} are fulfilled, this implies that
$$
\lim_{j\to\infty}\dot{u}_j(z,t)=\dot{u}(z,t)
$$
for almost all $(z,t)\in\mathbb{B}_T$, and, moreover,
\begin{equation}
    e^{\partial_t u_j+\psi_j(t,z)}g_j(z)dt\wedge dV_{\mathbb{B}} \to e^{\partial_t u+\psi(t,z)}g(z)dt\wedge dV_{\mathbb{B}} 
\end{equation}
in the weak sense of Radon measures on $\mathbb{B}_T$.

\medskip
\noindent
\textbf{Step 3.}
It remains to show the convergence of parabolic Hessian measures on the left-hand side of \eqref{eq:hessian appro sequence}. By Fubini's theorem, it is easy to see that for almost all $t\in(0,T]$, we have $\lim_{j\to\infty}\dot{u}_j(z,t)=\dot{u}(z,t)$ for almost all $z\in\mathbb{B}$. For all such $t$, we use the stability result \cite[Proposition 8.1]{KN26} again to conclude that $u_j(t,\cdot)\to u_t\in\operatorname{SH}_m(\mathbb{B},\omega)\cap C^0(\mathbb{B})$, which necessarily coincides with $u(t,\cdot)$ by the proof of \cref{weak cptness} (see also \cite[Proposition 1.17]{GLZ21a}). Due to the locally uniformly Lipschitz property of $u_j$ and $u$, we conclude that $u_j(t,\cdot)\to u(t,\cdot)$ uniformly for all $t\in(0,T]$ and hence $u$ is continuous on $(0,T]\times\overline{\mathbb{B}}$ and $u$ has the correct boundary value on $(0,T]\times\partial\mathbb{B}$.

Indeed, fix an arbitrary $t\in(0,T]$, choose a sequence $t_i\to t$ such that $u_j(t_i,\cdot)\to u(t_i,\cdot)$ uniformly as $j\to\infty$. Note that for this fixed $t$, we have $|u_j(t,\cdot)-u_j(t_k,\cdot)|\leq\kappa_t|t-t_k|$ for all $j$ and $k$ large enough. For each $\varepsilon>0$, we first choose $k$ large such that 
$$
|u_j(t,\cdot)-u_j(t_k,\cdot)|\leq\frac{\varepsilon}{3},\quad |u(t,\cdot)-u(t_k,\cdot)|\leq\frac{\varepsilon}{3},\quad\forall j\in \mathbb{Z}^+,
$$
and then choose $j$ large enough for this fixed $t_k$ such that $|u(t_k,\cdot)-u_j(t_k,\cdot)|\leq\frac{\varepsilon}{3}$. It follows from the triangle inequality that $u_j(t,\cdot)\to u(t,\cdot)$ uniformly. As a consequence, each slice $u(t,\cdot)$ is a continuous $(\omega,m)$-subharmonic function, it follows easily from the locally uniformly Lipschitz condition of $u$ that $u$ is continuous globally on $\mathbb{B}_T$. Furthermore, we have the weak convergence
$$
(dd^cu_j(t,\cdot))^m\wedge\omega^{n-m}\to (dd^cu(t,\cdot))^m\wedge\omega^{n-m}
$$
by \cite{KN26}.

Now, fix a test function $\chi\in C_c^0(\mathbb{B}_T)$, we can write
\begin{align*}
   & \int_{\mathbb{B}_T}\chi dt \wedge(dd^cu_j)^m\wedge\omega^{n-m}:=\int_0^Tdt\int_{\mathbb{B}}\chi(t,\cdot)(dd^cu_j(t,\cdot))^m\wedge\omega^{n-m}\\
    \to&\int_0^Tdt\int_{\mathbb{B}}\chi(t,\cdot)(dd^cu(t,\cdot))^m\wedge\omega^{n-m}= \int_{\mathbb{B}_T}\chi dt \wedge(dd^cu)^m\wedge\omega^{n-m},
\end{align*}
where thanks to the CLN inequality established in \cref{CLN}, we could apply the dominated convergence theorem to obtain the convergence above. The proof is finally concluded.

\medskip
\noindent
\textbf{Step 4.} We show that $u(t,\cdot)\to h_0$ in $L^1(\mathbb{B})$ as $t\searrow0$. From below, we construct a continuous subbarrier adapting the idea of \cite[Lemma 3.8]{GLZ21a}. Define
    $$
    \underline{v}(t,z) := h_0(z) + t(\rho(z) - C) + m[t \log(t/T) - t],
    $$
    where $\rho \le 0$ is a continuous $(\omega,m)$-subharmonic function in $\mathbb{B}$ satisfying $(dd^c \rho)^m \wedge \omega^{n-m} = g dV$ with $\rho = 0$ on $\partial \mathbb{B}$. By choosing $C := \max \{ \sup \rho + M_\psi - m \log T, \kappa_h - m \}$, where $M_\psi=\sup_{\mathbb{B}_T}|\psi|$, it is not difficult to see that $\underline{v}$ is a pluripotential subsolution with boundary value $\leq h$. 
 
    Since $H_j \to h$ uniformly on $\partial_0 \mathbb{B}_T$, for any $\varepsilon > 0$, we have $H_j \ge \underline{v} - \varepsilon$ on $\partial_0 \mathbb{B}_T$ for all sufficiently large $j$. The parabolic comparison principle (see \cref{lem:comparison_C1} again) then yields $u_j \ge \underline{v} - \varepsilon$ throughout $\mathbb{B}_T$. Taking $j \to \infty$ and then $\varepsilon \to 0$, we obtain $u \ge \underline{v}$ in $\mathbb{B}_T$. This implies that
    $$
    \liminf_{t \searrow 0} u(t,z) \ge \lim_{t \searrow 0} \underline{v}(t,z) = h_0(z).
    $$

    From above, we use an integral monotonicity argument inspired by \cite[Theorem 3.12]{GLZ21a} (see also \cref{thm:boundary_behavior_envelope} below). Since $u_j$ are smooth solutions, the Hessian Chern-Levine-Nirenberg inequality \cite[Proposition 3.7]{KN26} guarantees that the measure mass $\int_{\mathbb{B}} \chi (dd^c u_j)^m \wedge \omega^{n-m}$ is uniformly bounded for any continuous test function $\chi \ge 0$ compactly supported in $\mathbb{B}$. We claim that there exists a uniform constant $C > 0$ such that
    $$
    \int_{\mathbb{B}} \chi u_j(t,z) g_j dV \le \int_{\mathbb{B}} \chi H_j(0,z) g_j dV + C t, \quad \forall t \in (0,T).
    $$
Indeed, Since $u_j$ is a smooth solution to the complex Hessian flow, it satisfies the equation:
    $$
    (dd^c u_j(t,\cdot))^m \wedge \omega^{n-m} = e^{\partial_t u_j(t,\cdot) + \psi_j(t,\cdot)} g_j dV \ge e^{\partial_t u_j(t,\cdot) - K} g_j dV,
    $$
    where $K > 0$ is a uniform constant such that $\psi_j \ge -K$ on $\overline{\mathbb{B}_T}$. Multiplying both sides by the non-negative test function $\chi$ and integrating over $\mathbb{B}$, we obtain:
    $$
    \int_{\mathbb{B}} \chi e^{\partial_t u_j(t,\cdot) - K} g_j dV \le \int_{\mathbb{B}} \chi (dd^c u_j(t,\cdot))^m \wedge \omega^{n-m}\le C.
    $$
    Let $M_j := \int_{\mathbb{B}} \chi g_j dV$. Since $g_j \to g$ uniformly on $\overline{\mathbb{B}}$ and $g > 0$, the mass $M_j$ is uniformly bounded away from zero and infinity. Applying Jensen's inequality to the convex function $x \mapsto e^x$ with respect to the probability measure $d\mu_j := M_j^{-1} \chi g_j dV$, we get:
    $$
    \exp \left( \int_{\mathbb{B}} \partial_t u_j(t,\cdot) d\mu_j \right) \le \int_{\mathbb{B}} e^{\partial_t u_j(t,\cdot)} d\mu_j \le \frac{C}{M_j}.
    $$
    Taking the logarithm yields a uniform upper bound $C > 0$ for the integral of the time derivative:
    $$
    \int_{\mathbb{B}} \chi \partial_t u_j(t,z) g_j(z) dV \le C, \quad \forall t \in (0,T).
    $$
    Integrating this inequality with respect to time from $0$ to $t$ gives:
    $$
    \int_{\mathbb{B}} \chi u_j(t,z) g_j(z) dV - \int_{\mathbb{B}} \chi u_j(0,z) g_j(z) dV \le C t.
    $$
    Noting that $u_j(0,z) = H_j(0,z)$ by the initial condition, the proof of the claim is therefore finished.
    
    We have established that $u_j(t,\cdot) \to u(t,\cdot)$ uniformly on slices $t>0$, and by construction, $H_j(0,\cdot) \to h_0$ uniformly on $\overline{\mathbb{B}}$. Passing to the limit $j \to \infty$ yields
    $$
    \int_{\mathbb{B}} \chi u(t,z) g dV \le \int_{\mathbb{B}} \chi h_0(z) g dV + C t.
    $$
    Let $w_0 \in \operatorname{SH}_m(\mathbb{B}, \omega)$ be any weak cluster point of $u(t,\cdot)$ as $t \searrow 0$. Taking the limit $t \to 0^+$ gives $\int \chi w_0 g dV \le \int \chi h_0 g dV$. Since the test function $\chi$ was arbitrary, this forces $w_0 \le h_0$ almost everywhere, and hence everywhere. This implies that any cluster point of $u(t,\cdot)$ as $t \searrow 0$ must be $h_0$ and hence $u(t,\cdot)\to h_0$ in $L^1$ as $t\searrow0$. It thus remains to show that the convergence is actually uniform.
    
\medskip
\noindent
\textbf{Step 5.} We finally upgrade the $L^1$ convergence to uniform convergence on $\overline{\mathbb{B}}$, showing that $u(t,\cdot) \to h_0$ uniformly as $t \searrow 0$. This will complete the proof that $u \in C^0(\overline{\mathbb{B}_T})$.

From the lower bound established in Step 4, we have $u(t,z) \ge \underline{v}(t,z)$. Since $\underline{v}$ is continuous on $\overline{\mathbb{B}_T}$ and $\underline{v}(0,z) = h_0(z)$, we immediately obtain the uniform lower bound:
$$
\liminf_{t \searrow 0} \inf_{z \in \overline{\mathbb{B}}} (u(t,z) - h_0(z)) \ge 0.
$$

To establish the uniform upper bound, we adapt the boundary strip argument from \cite[Lemma 3.13]{GLZ21a}. For each $t \in [0,T)$, let $H_t$ be the unique continuous function on $\overline{\mathbb{B}}$ such that $(dd^c H_t)^m \wedge \omega^{n-m} = 0$ in $\mathbb{B}$ and $H_t = h(t,\cdot)$ on $\partial \mathbb{B}$. Since the boundary data $h(t, \zeta)$ is uniformly Lipschitz in time, the stability result \cite[Proposition 8.1]{KN26} ensures that $H_t(z) \le H_0(z) + \kappa_h t$ for all $z \in \overline{\mathbb{B}}$. Note also that for $t>0$, $u(t,\cdot)$ is an $(\omega,m)$-subharmonic function with boundary values $h(t,\cdot)$. The comparison principle yields $u(t,z) \le H_t(z)$ in $\mathbb{B}$. 

Since $u(t,\cdot) \to h_0$ in $L^1(\mathbb{B})$ as $t \searrow 0$, Hartogs' lemma for $\omega$-subharmonic functions (cf. \cite[Lemma 9.14]{GN18}) ensures that this convergence holds uniformly bounded from above on any compact subset. Thus, for any compact subset $K \Subset \mathbb{B}$, we have
\begin{equation}\label{eq:hartogs_interior}
    \limsup_{t \searrow 0} \sup_{z \in K} (u(t,z) - h_0(z)) \le 0.
\end{equation}

Now, fix an arbitrary $\varepsilon > 0$. Both $H_0$ and $h_0$ are continuous on $\overline{\mathbb{B}}$ and they coincide exactly on the boundary $\partial \mathbb{B}$. Therefore, we can find a narrow boundary strip $\mathbb{B} \setminus K$ (where $K$ is a compact subset of $\mathbb{B}$) such that 
$$
\sup_{z \in \mathbb{B} \setminus K} (H_0(z) - h_0(z)) \le \varepsilon.
$$
Evaluating the upper bound in this boundary strip, we obtain:
$$
u(t,z) - h_0(z) \le H_t(z) - h_0(z) \le H_0(z) - h_0(z) + \kappa_h t \le \varepsilon + \kappa_h t.
$$
Taking the limit supremum as $t \searrow 0$ yields $\limsup_{t \searrow 0} \sup_{\mathbb{B} \setminus K} (u(t,z) - h_0(z)) \le \varepsilon$. Combining this with the interior bound \eqref{eq:hartogs_interior} on $K$, we deduce that globally:
$$
\limsup_{t \searrow 0} \sup_{z \in \overline{\mathbb{B}}} (u(t,z) - h_0(z)) \le \varepsilon.
$$
Since $\varepsilon > 0$ is arbitrary, we conclude that $\limsup_{t \searrow 0} \sup_{\overline{\mathbb{B}}} (u(t,z) - h_0(z)) \le 0$. Together with the uniform lower bound, this proves that $u(t,\cdot) \to h_0$ uniformly on $\overline{\mathbb{B}}$ as $t \searrow 0$, hence $u$ is globally continuous on $\overline{\mathbb{B}_T}$. The proof is now complete.
\end{proof}

\section{Parabolic Envelopes}
In this section, unless otherwise stated, we always assume that $h$ is uniformly Lipschitz in $[0,T]$, that is, 
$$
|\partial_th|\leq\kappa_h,\quad \forall t\in(0,T].
$$

\subsection{Definitions and basic estimates}
Having established the solvability of the continuous Cauchy-Dirichlet problem on small Euclidean balls in \cref{continuous boundary data solution on balls}, a basic idea in pluripotential theory is to use the method of Perron envelope and a balayage argument to derive solutions on general domains. Following \cite{GLZ21a}, we introduce the parabolic envelope for the complex Hessian flow. Since most of the properties of envelopes can be derived by copying along the lines of the proofs in \cite{GLZ21a} or the proof in \cref{continuous boundary data solution on balls} directly, we omit their proofs. 

\begin{definition}\label{def:pluripotential_subsolution}
Let $u \in P_m(\Omega_T) \cap L_{loc}^\infty(\Omega_T)$. We say that $u$ is a \emph{pluripotential subsolution} to the complex Hessian flow if it satisfies the inequality
$$
dt \wedge (dd^c u)^m \wedge \omega^{n-m} \ge e^{\partial_t u + \psi(t,z)} g(z) dt \wedge dV
$$
in the sense of measures in $\Omega_T$. It is called a \emph{pluripotential supersolution} if the reverse inequality holds in the sense of measures in $\Omega_T$.

If, moreover, $u^* \le h$ on $\partial_0 \Omega_T$, i.e., for all $(\tau,\zeta) \in \partial_0 \Omega_T$,
$$
\limsup_{\Omega_T \ni (t,z) \to (\tau,\zeta)} u(t,z) \le h(\tau,\zeta),
$$
we say that $u$ is a pluripotential subsolution to the Cauchy-Dirichlet problem for the parabolic complex Hessian equation with boundary data $h$.
\end{definition}

\begin{proposition}\label{prop:slice_equivalence}
Fix $u \in P_m(\Omega_T) \cap L_{loc}^\infty(\Omega_T)$. 
\begin{enumerate}
    \item[(1)] $u$ is a pluripotential subsolution to the complex Hessian flow if and only if for a.e. $t \in (0,T)$,
    $$
    (dd^c u_t)^m \wedge \omega^{n-m} \ge e^{\partial_t u(t,\cdot) + \psi(t,\cdot)} g dV,
    $$
    in the sense of measures in $\Omega$.
    \item[(2)] If $u$ is moreover locally semi-concave in $t$, it is a pluripotential subsolution if and only if for all $t$,
    $$
    (dd^c u_t)^m \wedge \omega^{n-m} \ge e^{\partial_t^+ u(t,\cdot) + \psi(t,\cdot)} g dV,
    $$
    in the sense of measures in $\Omega$.
    \item[(3)] Similar results hold for pluripotential supersolutions. In particular, if $u$ is a pluripotential solution to the complex Hessian flow and $u$ is locally uniformly semi-concave, then for almost all $t\in(0,T)$, we have
     $$
    (dd^c u_t)^m \wedge \omega^{n-m} = e^{\partial_t u(t,\cdot) + \psi(t,\cdot)} g dV.
    $$
\end{enumerate}
\end{proposition}
\begin{proof}
    See \cite[Proposition 3.2, Remark 3.3]{GLZ21a}.
\end{proof}

\begin{lemma}
\label{lem:max-m-parabolic-subsolution}
For any $u, v \in \mathcal{P}_m(\Omega_T) \cap L^\infty_{\mathrm{loc}}(\Omega_T)$, the following hold:
\begin{enumerate}
    \item $\mathbf{1}_{\{u \ge v\}} \partial_t \max(u,v) = \mathbf{1}_{\{u \ge v\}} \partial_t u$ 
    and $\mathbf{1}_{\{u > v\}} \partial_t \max(u,v) = \mathbf{1}_{\{u > v\}} \partial_t u$ 
    almost everywhere in $\Omega_T$;
    \item $dt \wedge dd^c(\max(u,v))^m\wedge\omega^{n-m} \ge \mathbf{1}_{\{u > v\}} dt \wedge (dd^cu)^m\wedge\omega^{n-m} + \mathbf{1}_{\{u \le v\}} dt \wedge (dd^cv)^m\wedge\omega^{n-m}$
    in the sense of positive measures on $\Omega_T$.
\end{enumerate}
In particular, the maximum of two pluripotential subsolutions is again an 
pluripotential subsolution.
\end{lemma}

\begin{proof}
As in \cite[Lemma 3.4]{GLZ21a}, this is an immediate consequence of \cref{lem:m-parabolic-sobolev} and the plurifine locality \cite[Corollary 5.2]{KN26}.
\end{proof}

\begin{definition}\label{def:subsolutions and envelope}
Let $h$ be a Cauchy-Dirichlet boundary data on $\partial_0 \Omega_T$. We let $\mathcal{S}_{h,g,\psi}(\Omega_T)$ denote the set of $m$-parabolic potentials $u \in P_m(\Omega_T)\cap L_{loc}^\infty(\Omega_T)$ such that:
\begin{enumerate}
    \item[(1)] $u$ is a pluripotential subsolution to the complex Hessian flow in $\Omega_T$;
    \item[(2)] $u^* \le h$ on $\partial_0 \Omega_T$.
\end{enumerate}
We let
$$
U = U_{h,g,\psi,\Omega_T} := \sup \left\{ u(t,z) \mid u \in \mathcal{S}_{h,g,\psi}(\Omega_T) \right\}
$$
denote the upper envelope of all subsolutions.
\end{definition}

\begin{lemma}\label{lem:admissible_set_properties}
The admissible class $\mathcal{S}_{h,g,\psi}(\Omega_T)$ is non-empty, uniformly bounded from above, and closed under finite maxima. Consequently, the parabolic envelope $U := U_{h,g,\psi,\Omega_T}$ is well-defined on $\Omega_T$. Furthermore, if we define $M_h := \sup_{\partial_0 \Omega_T} |h|$ and $M_\psi := \sup_{\Omega_T} \psi$, then $U$ and its upper semi-continuous regularization $U^*$ satisfy the explicit uniform bounds:
$$
A\rho(z) - M_h \le U(t,z) \le U^*(t,z) \le M_h, \quad \forall (t,z) \in \Omega_T,
$$
where the scaling constant is given by $A = e^{M_\psi / m}$, and $\rho \le 0$ is a bounded continuous $(\omega, m)$-subharmonic function on $\Omega$ solving $(dd^c \rho)^m \wedge \omega^{n-m} = g dV$ with zero boundary values, whose existence follows from \cite{KN26}.
\end{lemma}

\begin{proof}
We first establish the uniform upper bound. Given any candidate $v \in \mathcal{S}_{h,g,\psi}(\Omega_T)$, the time slice $v_t := v(t, \cdot)$ is an $(\omega, m)$-subharmonic function on $\Omega$ for each fixed $t \in (0, T)$. The boundary condition ensures that $\limsup_{z \to \zeta} v_t(z) \le h(t,\zeta) \le M_h$ for any boundary point $\zeta \in \partial \Omega$. Invoking the standard maximum principle for $(\omega,m)$-subharmonic functions, we immediately obtain $v_t(z) \le M_h$ on $\Omega$. Taking the supremum over all such admissible $v$ yields $U(t,z) \le M_h$, and consequently $U^* \le M_h$.

Next, we verify that $\mathcal{S}_{h,g,\psi}(\Omega_T)$ is non-empty by constructing a concrete global sub-barrier. Consider the following time-independent test function:
$$
\underline{v}(t,z) := A\rho(z) - M_h.
$$
It is straightforward to see that $\underline{v}$ respects the Cauchy-Dirichlet boundary data, as $\underline{v}^* \le -M_h \le h$ on the parabolic boundary $\partial_0 \Omega_T$. Since $\partial_t \underline{v} = 0$, evaluating its parabolic Hessian measure yields:
$$
dt \wedge (dd^c \underline{v})^m \wedge \omega^{n-m} = A^m (dd^c \rho)^m \wedge \omega^{n-m} \wedge dt = A^m g dV \wedge dt.
$$
To guarantee that $\underline{v}$ is a pluripotential subsolution, we require the inequality $A^m g \ge e^{\partial_t \underline{v} + \psi} g = e^\psi g$ to hold. Our choice of $A = e^{M_\psi / m}$ naturally satisfies this condition, confirming that $\underline{v} \in \mathcal{S}_{h,g,\psi}(\Omega_T)$. This proves the non-emptiness of the admissible class and establishes the lower bound for the envelope.

Finally, the fact that $\mathcal{S}_{h,g,\psi}(\Omega_T)$ is closed under finite maxima follows directly from \cref{lem:max-m-parabolic-subsolution}. 
\end{proof}

\subsection{Time regularities}
In this subsection, we assume that $h$ satisfies
\begin{equation}\label{eq:assumptions on h}
     t|\partial_th|,\,t^2\partial_t^2h\leq C,\quad \forall t\in(0,T].
\end{equation}
Invoking \cref{prop:slice_equivalence} and the locally uniformly Lipschitz property of $h$, the proof of the following identity principle is the same as \cite[Proposition 4.1]{GLZ21a}:
\begin{proposition}\label{prop:identity_principle}
    For any intermediate time $S \in (0,T)$, the global parabolic envelope $U_{h,g,\psi,\Omega_T}$ coincides with the envelope $U_{h,g,\psi,\Omega_S}$ defined on the sub-cylinder $\Omega_S$. That is, we have
    $$
    U_{h,g,\psi,\Omega_T} = U_{h,g,\psi,\Omega_S} \quad \text{in } \Omega_S.
    $$
\end{proposition}
Taking into account \cref{prop:identity_principle}, the arguments from \cref{lem:a priori lipschitz} apply verbatim to give the following:
\begin{theorem}\label{thm:envelope_lipschitz}
Assume $h$ satisfies \eqref{eq:assumptions on h}, the parabolic envelope $U := U_{h,g,\psi,\Omega_T}$ satisfies the uniform time-Lipschitz estimate:
    $$
    t|\partial_t U(t,z)| \le \kappa_U, \quad \forall (t,z) \in \Omega_T,
    $$
    where the explicit uniform constant $\kappa_U$ is given by
    \begin{equation}\label{eq:kappa_U_constant}
        \kappa_U := (2M_U + 2\kappa_h + 2m + \kappa_\psi T)(T+1) + M_U.
    \end{equation}
    Here, $M_U := \sup_{\Omega_T} |U|$ denotes the uniform $L^\infty$ bound of the envelope established in \cref{lem:admissible_set_properties}, and $\kappa_\psi$ is the uniform Lipschitz constant of the twist term $\psi$ in time.
\end{theorem}

\begin{theorem}\label{thm:envelope_semiconcave}
    Assume that the Cauchy-Dirichlet boundary data $h$ satisfies \eqref{eq:assumptions on h}. Suppose that the parabolic envelope $U := U_{h,g,\psi,\Omega_T}$ belongs to the admissible class $\mathcal{S}_{h,g,\psi}(\Omega_T)$ (i.e., $U$ is a pluripotential subsolution). Then $U$ is locally uniformly semi-concave in time, satisfying the estimate
    $$
    t^2 \partial_t^2 U(t,z) \le K_U, \quad \forall (t,z) \in \Omega_T.
    $$
    The explicit uniform constant $K_U$ is given by
    \begin{equation}\label{eq:K_U_constant}
        K_U := 2 C_U (T+1) + \kappa_U + M_U, 
    \end{equation}
    where $C_U := \max \left\{ 2C_\psi T^2 + \kappa_\psi, C_h + M_U + 2\kappa_h \right\}$. 
    Here, $M_U$ and $\kappa_U$ are the uniform $L^\infty$ and time-Lipschitz bounds established in \cref{lem:admissible_set_properties} and \cref{thm:envelope_lipschitz}, respectively, and $C_\psi$ is the uniform semi-concavity constant of the twist term $\psi$.
\end{theorem}

\begin{proof}
    Under the assumption that $U$ is a pluripotential subsolution, the proof follows exactly the same as the a priori semi-concavity estimate established for smooth solutions in \cref{lem:a priori semi-concave}. For a scaling parameter $s$ close to $1$, one considers the barrier
    $$
    v^s(t,z) := \frac{s^{-1}U(st,z) + s U(s^{-1}t,z)}{2} - C_U(t+1)(s-1)^2.
    $$
    Relying on the mixed Hessian inequality for bounded $(\omega,m)$-subharmonic functions (cf. \cite{Sun25}), a completely analogous verification confirms that $v^s \in \mathcal{S}_{h,g,\psi}(\Omega_T)$ when the constant $C_U$ is chosen as above. The definition of the envelope then enforces $v^s \le U$ everywhere in $\Omega_T$. We omit the repetitive calculations here.
\end{proof}
\begin{remark}\label{rmk:remove subsolution}
The temporary assumption that $U \in \mathcal{S}_{h,g,\psi}(\Omega_T)$ will be unconditionally verified in the sequel, thereby validating this semi-concavity estimate universally for all compatible data).
\end{remark}

\subsection{Matching boundary values}

Instead of using the harmonic solutions in the Monge-Amp\`ere case, we apply the resolution of general complex Hessian equations here to construct super-barriers: 
\begin{lemma}\label{lem:lateral_boundary_upper_bound}
    Let $\Omega \subset \mathbb{C}^n$ be a bounded, strictly $m$-pseudoconvex domain with smooth boundary. Assume that the Cauchy-Dirichlet boundary data $h$ is uniformly Lipschitz on the lateral boundary $[0,T] \times \partial\Omega$. Then the parabolic envelope $U := U_{h,g,\psi,\Omega_T}$ satisfies
    $$
    \limsup_{\Omega_T \ni (t,z) \to (t_0, \zeta)} U(t,z) \le h(t_0, \zeta), \quad \forall t_0 \in [0,T], \ \forall \zeta \in \partial \Omega.
    $$
\end{lemma}

\begin{proof}
    We construct an explicit global pluripotential supersolution to dominate the envelope $U$. For each fixed $t \in [0,T]$, let $H_t(z) := H(t,z)$ be the unique continuous $(\omega,m)$-subharmonic function on $\overline{\Omega}$ solving the homogeneous Dirichlet problem for the complex Hessian equation:
    \begin{equation}
        \left\{
        \begin{aligned}
            &(dd^c H_t)^m \wedge \omega^{n-m} = 0 \quad \text{in } \Omega, \\
            &H_t(z) = h(t, z) \quad \text{on } \partial \Omega.
        \end{aligned}
        \right.
    \end{equation}
    The existence and continuity of such $H_t$ follow from \cite[Theorem 8.2]{KN26}. Since the boundary data $h(t,z)$ is locally uniformly Lipschitz in $t$, the stability result \cite[Proposition 8.1]{KN26} ensures that the solution $H(t,z)$ is also locally uniformly Lipschitz in $t$ and hence continuous on $\Omega_T$.

    We claim that $H(t,z)$ is a pluripotential supersolution to the parabolic complex Hessian equation in $\Omega_T$. Indeed, it is clear that
    $$
    (dd^c H_t)^m \wedge \omega^{n-m} = 0 \le e^{\partial_t H(t,z) + \psi(t,z)} g(z) dV.
    $$
 On the other hand, the comparison principle easily yields that $H\geq h$ on $\partial_0\Omega_T$.

   For any admissible subsolution $v \in \mathcal{S}_{h,g,\psi}(\Omega_T)$, the comparison principle gives that $v_t \le H_t$ in $\Omega$. Taking the supremum over all such subsolutions yields:
    $$
    U(t,z) \le H(t,z) \quad \text{in } \Omega_T.
    $$
    Finally, fixing $t_0 \in (0,T]$ and $\zeta \in \partial \Omega$, the continuity of $H$ up to the boundary implies:
    $$
    \limsup_{\Omega_T \ni (t,z) \to (t_0, \zeta)} U(t,z) \le \lim_{(t,z) \to (t_0, \zeta)} H(t,z) = H(t_0, \zeta) = h(t_0, \zeta).
    $$
    This establishes the desired upper bound on the lateral boundary.
\end{proof}

\begin{lemma}\label{lem:sub_barriers}
    Let $\Omega \subset \mathbb{C}^n$ be a bounded, strictly $m$-pseudoconvex domain with smooth boundary. Assume that the Cauchy-Dirichlet boundary data $h$ is uniformly Lipschitz on the parabolic boundary $\partial_0 \Omega_T$. Then there exist pluripotential subsolutions $u,v\in \mathcal{S}_{h,g,\psi}(\Omega_T)$ satisfying the following optimal boundary behaviors:
    \begin{enumerate}
        \item[(1)] \textbf{At Dirichlet boundary points:} For any $(s,\zeta) \in [0,T) \times \partial\Omega$,
        $$
        \lim_{\Omega_T \ni (t,z) \to (s,\zeta)} u(t,z) = h(s,\zeta).
        $$
        \item[(2)] \textbf{At Cauchy boundary points:} For any $\zeta \in \overline{\Omega}$,
        $$
        \limsup_{\Omega_T \ni (t,z) \to (0,\zeta)}v(t,z) = h(0,\zeta), \quad \text{and} \quad \lim_{t \to 0^+} v(t,\zeta) = h(0,\zeta).
        $$
    \end{enumerate}
    Furthermore, if the initial data $h_0 := h(0,\cdot)$ is continuous on $\overline{\Omega}$, then $u$ can be chosen to be continuous on $[0,T) \times \overline{\Omega}$.
\end{lemma}

\begin{proof}
    The construction of these sub-barriers relies on the resolution of the homogeneous Dirichlet problem for the elliptic complex Hessian equation to match the boundary values, the arguments are completely analogous to those in the complex Monge-Amp\`ere case. We omit the details and refer the reader to Step 4 of \cref{continuous boundary data solution on balls} and \cite[Lemma 3.7 and Lemma 3.8]{GLZ21a} for the exact barrier formulas.
\end{proof}
Invoking the boundedness of $\psi$, the proof of the following lemma follows from Step 4 of \cref{continuous boundary data solution on balls} or \cite[Lemma 3.10]{GLZ21a} verbatim:
\begin{lemma}\label{lem:average_monotonicity}
    Suppose that $\varphi \in P_m(\Omega_T) \cap L_{loc}^\infty(\Omega_T)$ is a pluripotential subsolution to the complex Hessian flow. Assume there is an open subset $D \Subset \Omega$ and a uniform constant $C > 0$ such that the complex Hessian measure satisfies $\int_D (dd^c \varphi_t)^m \wedge \omega^{n-m} \le C$ for all $t \in (0,T)$. Then, given any non-negative continuous test function $\chi \in C_c^0(D)$, we can find a sufficiently large constant $A> 0$ depending on $C$ such that the integral quantity
    $$
    \mathcal{I}(t) := \int_D \chi \varphi(t,z) g(z) dV - A t
    $$
    is monotonically non-increasing with respect to $t \in (0,T)$.
\end{lemma}

\begin{corollary}\label{cor:boundary_limit_subsolutions}
    Let $\{v_j\}_{j=1}^\infty \subset \mathcal{S}_{h,g,\psi}(\Omega_T)$ be a uniformly bounded sequence of admissible subsolutions. Assume that this sequence is locally uniformly Lipschitz in time, with a uniform Lipschitz constant independent of the index $j$. If $v_j$ converges to a parabolic $m$-potential $v \in P_m(\Omega_T)$ in the $L_{loc}^1(\Omega_T)$ topology, then the limit $v$ satisfies the Cauchy-Dirichlet boundary condition:
    $$
    \limsup_{\Omega_T \ni (t,z) \to (\tau,\zeta)} v(t,z) \le h(\tau,\zeta), \quad \text{for all } (\tau,\zeta) \in \partial_0 \Omega_T.
    $$
\end{corollary}

\begin{proof}
    For boundary points located on the side $(\tau,\zeta) \in (0,T) \times \partial\Omega$, the required upper bound is a direct consequence of the global supersolution barrier constructed in \cref{lem:lateral_boundary_upper_bound}.

    We now turn our attention to the initial bottom boundary where $\tau = 0$. The argument is similar to that in Step 4 of \cref{continuous boundary data solution on balls}. Fix an arbitrary relatively compact domain $D \Subset \Omega$ and choose a non-negative test function $\chi \in C_c^0(D)$. According to a Chern-Levine-Nirenberg type estimate (cf. \cite[Proposition 3.7]{KN26}), the total mass $\int_D (dd^c v_{j,t})^m \wedge \omega^{n-m}$ is uniformly bounded from above for all $j$ and all $t \in (0,T)$, since the sequence is uniformly bounded in $L^\infty(\Omega_T)$.

    By virtue of \cref{lem:average_monotonicity}, we can extract a uniform constant $A > 0$ such that for all $j$ and all $t \in (0,T)$,
    $$
    \int_D \chi v_{j,t} g dV \le \int_D \chi v_{j,0} g dV + A t \le \int_D \chi h_0 g dV + A t.
    $$
Thanks to \cref{lem:L1_slice_estimate} and the convergence $v_j \to v$ in $L_{loc}^1(\Omega_T)$ (hence in $L^1(\Omega_T)$ by the uniform boundedness of functions), we can pass to the limit as $j \to \infty$ on slices to yield
    $$
    \int_D \chi v_t g dV \le \int_D \chi h_0 g dV + A t, \quad \text{for a.e. } t \in (0,T).
    $$
    Let $v^*$ be any weak cluster point of $v_t$ as $t \searrow 0$. Taking the limit supremum as $t \to 0^+$ in the above inequality gives
    $$
    \int_D \chi v^* g dV \le \int_D \chi h_0 g dV.
    $$
    Because the choice of the test function $\chi \ge 0$ and the subset $D$ was arbitrary, it follows that $v^* \le h_0$ pointwise in $\Omega$. Combining this pointwise inequality with the upper semi-continuous extension property of $m$-parabolic potentials at $t=0$ (\cref{lem:usc extension to 0}), we finally conclude that $\limsup_{(t,z) \to (0,\zeta)} v(t,z) \le h_0(\zeta)$.
\end{proof}

Now, we can show that the upper semi-continuous regularization $U^*$ has the correct boundary values:

\begin{theorem}\label{thm:boundary_behavior_envelope}
    Assume that the Cauchy-Dirichlet boundary data $h$ is uniformly Lipschitz with respect to time in $[0,T]$. Then the upper semi-continuous regularization of the envelope $U := U_{h,g,\psi,\Omega_T}$ satisfies:
    \begin{enumerate}
        \item[(i)] For any lateral boundary point $(s,\zeta) \in [0,T) \times \partial\Omega$,
        $$
        \lim_{\Omega_T \ni (t,z) \to (s,\zeta)} U^*(t,z) = h(s,\zeta).
        $$
        \item[(ii)] For any initial Cauchy point $z_0 \in \Omega$,
        $$
        \lim_{t \searrow 0} U^*(t,z_0) = h(0,z_0), \quad \text{and} \quad \limsup_{\Omega_T \ni (t,z) \to (0,z_0)} U^*(t,z) = h(0,z_0).
        $$
    \end{enumerate}
\end{theorem}

\begin{proof}
Taking into account the corresponding properties established in \cite{KN26}, the proof of \cite[Theorem 3.12]{GLZ21a} can be easily adapted here. The first statement follows immediately from \cref{lem:lateral_boundary_upper_bound}) and \cref{lem:sub_barriers}.

The lower bound $\liminf_{t \to 0} U(t,z_0) \ge h_0(z_0)$ is guaranteed by the initial sub-barrier in \cref{lem:sub_barriers}. Thus, it remains to prove the upper bound:
    $$
    \limsup_{\Omega_T \ni (t,z) \to (0,z_0)} U^*(t,z) \le h_0(z_0).
    $$
 By \cref{thm:envelope_lipschitz}, we have that $U$ is locally uniformly Lipschitz in $t \in (0,T)$, hence $U^*(t,\cdot) = (U_t)^*$ in $\Omega$ for all $t \in (0,T)$ thanks to \cref{prop:envelope basic properties}. Consequently, we can use \cref{lem:usc extension to 0} to reduce the problem to showing that
    $$
    \lim_{t \searrow 0} (U_t)^*(z_0) \le h_0(z_0), \quad \forall z_0 \in \Omega.
    $$
    
    Let $\chi \in C_c^0(\Omega)$ be a non-negative continuous test function compactly supported in a subdomain $D \Subset \Omega$. Fix an arbitrary slice $t_0 \in (0,T)$. By Choquet's lemma, there exists an increasing sequence $\{v^j\} \subset \mathcal{S}_{h,g,\psi}(\Omega_T)$ bounded uniformly by $M_U$, such that $(U_{t_0})^* = (\lim_{j \to \infty} v^j_{t_0})^*$ in $\Omega$. 
    
    For the compact subset $D$, the Chern-Levine-Nirenberg inequality for the complex Hessian operator (see \cite[Proposition 3.7]{KN26}) implies that the total mass of the Hessian measures is uniformly bounded:
    $$
    \int_D \chi (dd^c v^j_t)^m \wedge \omega^{n-m} \le C_1, \quad \forall t \in (0,T),
    $$
    where $C_1$ depends on $D$, $\chi$, $\omega$, and $M_U$, but is independent of $t$ and $j$. Applying \cref{lem:average_monotonicity}, we obtain a uniform constant $C_2 > 0$ such that
    $$
    \int_D \chi v^j_{t_0} g dV \le \int_D \chi h_0 g dV + C_2 t_0.
    $$
    By \cite[Corollary 5.5 or Theorem 7.8]{KN26}, the $m$-negligible set $\{z \in \Omega \mid \lim_{j \to \infty} v^j_{t_0}(z) < (U_{t_0})^*(z) \}$ is $m$-polar and hence has Lebesgue measure zero. Thus, applying the monotone convergence theorem \cite[Lemma 5.4]{KN26} to pass the limit $j \to \infty$ inside the integral, we obtain
    $$
    \int_D \chi (U_{t_0})^* g dV \le \int_D \chi h_0 g dV + C_2 t_0.
    $$
    Let $w_0 \in \operatorname{SH}_m(\Omega, \omega)$ be any cluster point of $(U_t)^*$ as $t \searrow 0$. By standard pluripotential compactness, $(U_t)^*$ converges to $w_0$ in $L^q(\Omega)$ for any $1<q<\frac{n}{n-m}$. While $g\in L^p(\Omega)$ with $p>\frac{n}{m}$, we can taking the limit $t_0 \to 0^+$ in the integral inequality above to write
    $$
    \int_D \chi w_0 g dV \le \int_D \chi h_0 g dV.
    $$
    Since the non-negative test function $\chi$ and the subset $D \Subset \Omega$ were arbitrary, and the density $g$ is strictly positive almost everywhere, we deduce that $w_0 \le h_0$ almost everywhere. As both $w_0$ and $h_0$ are $(\omega,m)$-subharmonic functions, this inequality extends everywhere in $\Omega$. This confirms the desired upper bound and concludes the proof.
\end{proof}

\begin{lemma}\label{lem:initial_uniform_convergence}
    Assume that the initial Cauchy data $h_0 := h(0,\cdot)$ is continuous on $\overline{\Omega}$ and that $h$ is locally uniformly Lipschitz on $(0,T)$. Then the upper semi-continuous regularization $U^*(t,\cdot)$ converges uniformly to $h_0$ on $\overline{\Omega}$ as $t \searrow 0$.
\end{lemma}

\begin{proof}
When $h$ is uniformly Lipschitz on $[0,T]$, the proof is the same as that in Step 5 of \cref{continuous boundary data solution on balls}. For the general case, fix a small $\varepsilon > 0$. We define the time-shifted boundary data on the lateral side of the sub-cylinder $\Omega \times [0, T-\varepsilon]$ by:
    $$
    h^\varepsilon(t, \zeta) := h(t+\varepsilon, \zeta) \quad \text{for } (t,\zeta) \in (0, T-\varepsilon] \times \partial\Omega.
    $$
    To ensure continuity at the corner $\{0\} \times \partial\Omega$, we must match the initial data. We define the initial slice at $t=0$ by
    $$
    h^\varepsilon(0,z) := h_0(z) + \phi_\varepsilon(z), \quad \text{for } z \in \overline{\Omega},
    $$
    where $h_0(\cdot) := h(0, \cdot)$, and $\phi_\varepsilon \in \operatorname{SH}_m(\Omega, \omega) \cap C^0(\overline{\Omega})$ is defined as the unique continuous solution to the homogeneous Dirichlet problem for the complex Hessian equation:
    \begin{equation}
        \left\{
        \begin{aligned}
            &(dd^c \phi_\varepsilon)^m \wedge \omega^{n-m} = 0 \quad \text{in } \Omega, \\
            &\phi_\varepsilon(z) = h(\varepsilon, z) - h_0(z) \quad \text{on } \partial\Omega.
        \end{aligned}
        \right.
    \end{equation}
    By construction, it is clear that $h^\varepsilon$ is globally continuous on the parabolic boundary $\partial_0 \Omega_{T-\varepsilon}$. By the stability theorem (cf. \cite[Proposition 8.1]{KN26}), we see that $h^\varepsilon$ converges uniformly to $h$ on the whole parabolic boundary. Consequently, it follows from the definition that $U^\varepsilon:=U_{h^\varepsilon,\psi,g,T-\varepsilon}$ converges to $U$ uniformly on $[0,T-\varepsilon)$ as $\varepsilon\to 0$ (we have used the identity principle \cref{prop:identity_principle} below).
    
It is clear that $h^\varepsilon$ is uniformly Lipschitz on $[0,T-\varepsilon]$, hence we have that $U^\varepsilon(t,\cdot)$ converges uniformly to $h_0$ as $t\to0$ for each fixed $\varepsilon$. The proof is finally concluded. 
\end{proof}

\subsection{Subsolution property of $U$}
The following theorem is parallel to the Monge-Amp\`ere case \cite[Theorem 4.6]{GLZ21a}, so we only sketch the proof.

\begin{theorem}\label{thm:maximal_subsolution}
    Assume that the Cauchy-Dirichlet boundary data $h$ satisfies the locally uniformly Lipschitz condition with respect to time, $t|\partial_t h(t,z)| \le \kappa_h$, for all $(t,z) \in \partial_0 \Omega_T$. Then the parabolic envelope $U := U_{h,g,\psi,\Omega_T}$ is a pluripotential subsolution, i.e., $U \in \mathcal{S}_{h,g,\psi}(\Omega_T)$. Furthermore, it attains the prescribed boundary values in the following sense:
    \begin{enumerate}
        \item[(1)] For any lateral boundary point $(s,\zeta) \in (0,T) \times \partial\Omega$, $\lim_{\Omega_T \ni (t,z) \to (s,\zeta)} U(t,z) = h(s,\zeta)$.
        \item[(2)] For any initial Cauchy point $z_0 \in \Omega$, $\limsup_{\Omega_T \ni (t,z) \to (0,z_0)} U(t,z) = h_0(z_0)$.
        \item[(3)] For any $z \in \Omega$, $\lim_{t \searrow 0} U_t(z) = h_0(z)$.
    \end{enumerate}
    If the initial data $h_0$ is continuous on $\overline{\Omega}$, then $U$ approaches $h$ continuously on the entire parabolic boundary $\partial_0 \Omega_T$.
\end{theorem}

\begin{proof}[Sketch of proof]
We prove the result on finite cylinders; the general case follows by
restriction to \(\Omega_S\), \(S<T\), and the identity principle. The boundary
assertions follow from \cref{thm:boundary_behavior_envelope} and
\cref{lem:initial_uniform_convergence}; it remains to prove
\(U\in\mathcal S_{h,g,\psi}(\Omega_T)\).

We first assume that \(h\) is uniformly Lipschitz on \([0,T]\) and that
\(h_0\in C^0(\overline\Omega)\). Set
\[
H_m(u_t):=(dd^cu_t)^m\wedge\omega^{n-m},\qquad b:=\min(1,T/2).
\]
For \(\kappa>0\), let \(\mathcal S^\kappa_{h,g,\psi}\) be the subclass of
\(\mathcal S_{h,g,\psi}\) consisting of all \(u\) such that
\[
\operatorname*{ess\,sup}_{\Omega}|\partial_tu(t,\cdot)|
\le \frac{\kappa}{\min(t,b)}
\quad\text{for a.e. }t,
\]
and set \(U^\kappa:=\sup\{u:u\in\mathcal S^\kappa_{h,g,\psi}\}\).
The truncated class is stable under finite maxima. By Choquet's lemma, there
is an increasing sequence \(u_j\in\mathcal S^\kappa_{h,g,\psi}\) such that
\[
(U^\kappa)^*=\left(\sup_j u_j\right)^* .
\]
Since the \(u_j\)'s have a common local Lipschitz bound in time, they increase
to \((U^\kappa)^*\) a.e. and \((U^\kappa)^*\in P_m(\Omega_T)\). By the
monotone convergence theorem for Hermitian Hessian measures,
\[
dt\wedge H_m(u_{j,t})\longrightarrow
dt\wedge H_m(((U^\kappa)^*)_t)
\]
weakly. Moreover
\(\partial_tu_j+\psi\to \partial_t(U^\kappa)^*+\psi\) in the sense of
distributions and is locally uniformly bounded. The exponential lower
semicontinuity lemma therefore gives
\[
e^{\partial_t (U^\kappa)^*+\psi}g\,dt\wedge dV
\le
\liminf_{j\to\infty} e^{\partial_tu_j+\psi}g\,dt\wedge dV .
\]
Thus \((U^\kappa)^*\) is a pluripotential subsolution. Since
\((U^\kappa)^*\le U^*\le h\) on \(\partial_0\Omega_T\), and the same
time-Lipschitz bound is inherited by the limit, we get
\[
(U^\kappa)^*=U^\kappa\in\mathcal S^\kappa_{h,g,\psi}(\Omega_T).
\]

We next show that \(U^\kappa\) is independent of \(\kappa\) for large
\(\kappa\). Choose \(\rho\le0\), bounded and \((\omega,m)\)-subharmonic,
such that \(H_m(\rho)\ge e^{-B}g\,dV\). For
\(u\in\mathcal S^\kappa_{h,g,\psi}\), \(s\) close to \(1\), and
\(a:=1-2|s-1|\), define on \(\Omega_S\), \(sS<T\),
\[
w_s(t,z):=a s^{-1}u(st,z)+(1-a)\rho(z)-C(1-a)(t+1).
\]
By the mixed Hessian inequality and homogeneity of \(H_m\),
\(w_s\) is a subsolution once \(C\) absorbs the errors coming from
\(-m\log s\), the Lipschitz variation of \(\psi\), the boundedness of
\(\psi\), and the lower bound for \(H_m(\rho)\). Enlarging \(C\), the
uniform Lipschitz bound of \(h\) also gives \(w_s^*\le h\) on
\(\partial_0\Omega_S\). The factor \(a\) ensures that, for
\(\kappa\ge\kappa_0\),
\[
\operatorname*{ess\,sup}_{\Omega}|\partial_tw_s(t,\cdot)|
\le \frac{\kappa}{\min(t,b)} .
\]
Using the truncated identity principle, \(w_s\le U^\kappa\) in \(\Omega_S\).
Taking the supremum over \(u\in\mathcal S^\kappa\) and letting
\(s\to1^\pm\), we obtain \(t|\partial_tU^\kappa|\le \kappa_0\). Hence
\(U^\kappa\in\mathcal S^{\kappa_0}\), and so
\[
U^\kappa=U^{\kappa_0}\qquad(\kappa\ge\kappa_0).
\]

Now let \(v\in\mathcal S_{h,g,\psi}(\Omega_T)\). For \(S<T\) and
\(\varepsilon>0\), set
\[
v_\varepsilon(t,z):=
v(t+\varepsilon,z)-C\varepsilon(t+1)-\theta(\varepsilon),
\qquad (t,z)\in\Omega_S,
\]
where \(C\) is large and
\(\theta(\varepsilon)=\sup_{\overline\Omega}|U^*(\varepsilon,\cdot)-h_0|\).
Then \(v_\varepsilon\in\mathcal S^\kappa_{h,g,\psi}(\Omega_S)\) for some
\(\kappa\), hence \(v_\varepsilon\le U^{\kappa_0}\). Letting
\(\varepsilon\to0\), then \(S\to T\), gives \(v\le U^{\kappa_0}\). Taking
the supremum over \(v\) yields \(U\le U^{\kappa_0}\), while the reverse
inequality is immediate. Thus
\[
U=U^{\kappa_0}\in\mathcal S_{h,g,\psi}(\Omega_T).
\]

If \(h_0\) is not continuous, take
\(h_0^j\in SH_m(\Omega,\omega)\cap C^0(\overline\Omega)\) with
\(h_0^j=h_0\) on \(\partial\Omega\) and \(h_0^j\searrow h_0\). Let
\(U^j:=U_{h^j,g,\psi,\Omega_T}\). The previous case gives
\(U^j\in\mathcal S_{h^j,g,\psi}\), and \(U^j\searrow V\ge U\). The monotone
convergence of Hessian measures and the same exponential lower semicontinuity
show that \(V\) is a subsolution with \(V^*\le h\); hence \(V\le U\), so
\(V=U\).

Finally, for boundary data satisfying only the weighted Lipschitz condition,
fix \(S<T\) and use shifted data
\[
h^\varepsilon(t,\zeta)=h(t+\varepsilon,\zeta),\qquad
h^\varepsilon(0,\cdot)=h_0+\phi_\varepsilon,
\]
where \(H_m(\phi_\varepsilon)=0\) and
\(\phi_\varepsilon|_{\partial\Omega}=h(\varepsilon,\cdot)-h_0\).
Then \(h^\varepsilon\) is uniformly Lipschitz on \([0,S]\) and
\(h^\varepsilon\to h\) uniformly on \(\partial_0\Omega_S\). The corresponding
envelopes converge uniformly to \(U\) on \(\Omega_S\), and the preceding
limit argument gives that \(U\) is a subsolution on \(\Omega_S\). Since
\(S<T\) is arbitrary, \(U\in\mathcal S_{h,g,\psi}(\Omega_T)\).
\end{proof}

As a result, we can finally improve \cref{thm:envelope_semiconcave} as follows:

\begin{theorem}\label{thm:envelope_semiconcave 2}
    Assume that the Cauchy-Dirichlet boundary data $h$ satisfies \eqref{eq:assumptions on h}. Then $U$ is locally uniformly semi-concave in time, satisfying the estimate
    $$
    t^2 \partial_t^2 U(t,z) \le K_U, \quad \forall (t,z) \in \Omega_T.
    $$
    The explicit uniform constant $K_U$ is given by
    \begin{equation}\label{eq:K_U_constant 2}
        K_U := 2 C_U (T+1) + \kappa_U + M_U, 
    \end{equation}
    where $C_U := \max \left\{ 2C_\psi T^2 + \kappa_\psi, C_h + M_U + 2\kappa_h \right\}$. 
\end{theorem}

\section{The parabolic comparison principle}

The following parabolic comparison principle is crucial in the balayage argument:

\begin{theorem}\label{thm:parabolic comparison}
    Let $\Phi,\Psi\in P_m(\Omega_T)\cap L^\infty(\Omega_T)$ be two locally uniformly semi-concave $m$-parabolic potentials with boundary data $h_\Phi$ and $h_\Psi$ such that $\Phi$ is a pluripotential subsolution and $\Psi$ is a pluripotential supersolution. Assume moreover that
    \begin{equation}\label{h Lip condition Phi}
t|\partial_th_{\Phi}|,\,t^2\partial_t^2h_{\Phi}\leq C,\quad t\in(0,T].
    \end{equation}
    Then, if $h_\Phi\leq h_\Psi$, we have $\Phi\leq\Psi$.
\end{theorem}
As an immediate corollary, we obtain the uniqueness of the pluripotential solution of the complex Hessian flow:

\begin{corollary}\label{cor:uniqueness}
    Assume that $\Phi, \Psi \in P_m(\Omega_T)\cap L^\infty(\Omega_T)$ are two pluripotential solutions to the parabolic complex Hessian equation with the same Cauchy-Dirichlet boundary data $h$ satisfying \eqref{h Lip condition Phi}. If both $\Phi$ and $\Psi$ are locally uniformly semi-concave in $t \in (0, T]$, then $\Phi = \Psi$ in $\Omega_T$.
\end{corollary}

\begin{proof}
Taking $\Phi$ as the subsolution and $\Psi$ as the supersolution, and noting that they share the same boundary data $h_\Phi = h_\Psi = h$ which satisfies condition \eqref{h Lip condition Phi}, the comparison principle directly yields:
    $$
    \Phi \le \Psi \quad \text{in } \Omega_T.
    $$

    Symmetrically, by reversing their roles, taking $\Psi$ as the subsolution and $\Phi$ as the supersolution, the comparison principle again applies and gives:
    $$
    \Psi \le \Phi \quad \text{in } \Omega_T.
    $$

    Combining these two inequalities, we conclude that $\Phi = \Psi$ everywhere in $\Omega_T$.
\end{proof}

As in \cite{GLZ21a}, we first prove \cref{cor:uniqueness} under the assumption that $\Phi$ is $C^1$ in $t$ and then perform regularization to prove the general case:

\begin{lemma}\label{lem:comparison_C1}
    Assume the same setting as in \cref{thm:parabolic comparison}. Suppose moreover that $\Phi$ is of class $C^1$ in $t$ and $\Psi$ is continuous on $[0,T) \times \overline{\Omega}$. Then 
    $$
    h_\Phi \le h_\Psi \implies \Phi \le \Psi \quad \text{in } \Omega_T.
    $$
\end{lemma}

\begin{proof}
    Fix $S \in (0, T)$ and a small constant $\varepsilon > 0$. We aim to show that 
    $$
    \Phi \le \Psi + 2\varepsilon t \quad \text{in } \Omega_S.
    $$
    Consider the upper semi-continuous function on $[0, S] \times \overline{\Omega}$ defined by
    $$
    W(t,z) := \Phi(t,z) - \Psi(t,z) - 2\varepsilon t.
    $$
    Since $h_\Phi \le h_\Psi$, we have $W \le 0$ on the parabolic boundary $\partial_0 \Omega_S$. We argue by contradiction and assume that the maximum of $W$ is strictly positive and is attained at some interior point $(t_0, z_0) \in (0, S] \times \Omega$.

    Define the compact set where the maximum is achieved at time $t_0$:
    $$
    K := \{ z \in \Omega \mid W(t_0, z) = W(t_0, z_0) \}.
    $$
    By the maximum principle in the time variable, we must have
    $$
    \partial_t \Phi(t_0, z) \ge \partial_t^- \Psi(t_0, z) + 2\varepsilon, \quad \forall z \in K.
    $$
    Since $\Psi$ is locally uniformly semi-concave in $t$ and continuous on $[0,T) \times \Omega$, its left time-derivative $\partial_t^- \Psi(t, z)$ is upper semi-continuous in $[0,T) \times \Omega$. Thus, there exists a sufficiently small radius $r > 0$ such that the following inequality holds on the open neighborhood $B := \{ z \in \Omega \mid \operatorname{dist}(z, K) < r \}$:
    \begin{equation}\label{eq:time_deriv_ineq}
        \partial_t \Phi(t_0, z) \ge \partial_t^- \Psi(t_0, z) + \varepsilon, \quad \forall z \in B.
    \end{equation}

    To simplify notations, let $\varphi := \Phi(t_0, \cdot)$ and $v := \Psi(t_0, \cdot)$ be the spatial slices at $t_0$. Since $\Phi$ is a pluripotential subsolution and $\Psi$ is a supersolution, \cref{prop:slice_equivalence} implies that in the sense of measures on $B$:
    \begin{align*}
        (dd^c \varphi)^m \wedge \omega^{n-m} &\ge e^{\partial_t \Phi(t_0, z) + \psi(t_0, z)} g(z) dV, \\
        (dd^c v)^m \wedge \omega^{n-m} &\le e^{\partial_t^- \Psi(t_0, z) + \psi(t_0, z)} g(z) dV.
    \end{align*}
    Combining these with \eqref{eq:time_deriv_ineq}, we obtain the strict measure inequality on $B$:
    \begin{equation}\label{eq:measure_strict_ineq}
        (dd^c \varphi)^m \wedge \omega^{n-m} \ge e^\varepsilon (dd^c v)^m \wedge \omega^{n-m} \quad \text{in } B.
    \end{equation}

    Now, we define the lifted function $\varphi_r := \varphi + m_r$, where $m_r := \min_{\partial B} (v - \varphi)$. By definition, $v \ge \varphi_r$ on $\partial B$. The comparison principle \cite[Corollary 6.2]{KN26} then yields that $\varphi_r \le v$ everywhere in $B$. 

    Evaluating this inequality at the maximum point $z_0 \in K \subset B$, we get:
    $$
    \varphi_r(z_0) - v(z_0) = \varphi(z_0) - v(z_0) + \min_{\partial B} (v - \varphi) \le 0,
    $$
    which rearranges to:
    $$
    \varphi(z_0) - v(z_0) \le \max_{\partial B} (\varphi - v).
    $$
    Since $K \cap \partial B = \emptyset$ by construction, the maximum of $\varphi - v$ on $\overline{B}$ is strictly attained in the interior, so $\varphi(z) - v(z) < \varphi(z_0) - v(z_0)$ for all $z \in \partial B$. This directly contradicts the inequality above.
\end{proof}

The proof of the following lemma follows from \cite[Lemma 6.9]{GLZ21a} verbatim.
\begin{lemma}\label{lem:supersolution_initial_estimate}
    Assume $\Psi \in P_m(\Omega_T)$ has boundary data $h_\Psi$. If $\Psi$ is a pluripotential supersolution to the parabolic complex Hessian equation, then for all $(t,z) \in \Omega_T$,
    $$
    \Psi(t,z) \ge h_\Psi(0,z) - c(t),
    $$
    where $c(t) > 0$ satisfies $\lim_{t \to 0^+} c(t) = 0$.
\end{lemma}

Having the lower bound of a supersolution, we can remove the continuity assumption $\Psi$ in \cref{lem:comparison_C1}, following \cite[Lemma 6.10]{GLZ21a}, which is based on an idea of \cite{DL17}:

\begin{lemma}\label{lem:comparison_C1_interior}
    Assume the same setting as in the parabolic comparison principle. Suppose moreover that $\Phi$ is of class $C^1$ in $t$. Then 
    $$
    h_\Phi \le h_\Psi \implies \Phi \le \Psi \quad \text{in } \Omega_T.
    $$
\end{lemma}

\begin{proof}
It is clear that $\Psi$ is continuous on $(0,T) \times \Omega$. Indeed, the boundary data $h_\Psi$ is continuous on $[0,T) \times \partial\Omega$. Thus, each slice is a supersolution of the elliptic complex Hessian equation and hence is continuous by \cite[Theorem 8.2]{KN26}. However, unlike \cref{lem:comparison_C1}, $\Phi$ is only assumed to be continuous in the interior, not necessarily on $[0,T) \times \overline{\Omega}$.

Fix an arbitrary $S \in (0, T)$. We will prove that $\Phi \le \Psi$ on the sub-cylinder $\Omega_S$. For a sufficiently small $s \in (0, (T-S)/2)$, we define a lifted and time-shifted supersolution $v$ on $[0,S] \times \overline{\Omega}$ as follows:
    $$
    v(t,z) := \Psi(t+s, z) + c(s) + \delta(s) + Ast,
    $$
    where $c(s) > 0$ is the initial error term from \cref{lem:supersolution_initial_estimate} satisfying $\lim_{s \to 0^+} c(s) = 0$, $A > 0$ is a large constant to be chosen, and $\delta(s)$ is the uniform modulus of continuity of the boundary data $h_\Psi$ on the $\partial_s\Omega_S$:
    $$
    \delta(s) := \sup \left\{ |h_\Psi(\tau, z) - h_\Psi(t, z)| \;\middle|\; z \in \partial\Omega, \, t,\tau \in [0,S], \, |t-\tau| \le s \right\}.
    $$
    Clearly, we have $\lim_{s \to 0^+} \delta(s) = 0$.

    Let us verify that $v$ is indeed a supersolution on $\Omega_S$. Since $\Psi$ is a supersolution on $\Omega_T$, evaluating the Hessian measure of $v$ at time $t$ yields:
    \begin{align*}
        (dd^c v(t,\cdot))^m \wedge \omega^{n-m} &= (dd^c \Psi(t+s, \cdot))^m \wedge \omega^{n-m} \\
        &\le e^{\partial_t \Psi(t+s,z) + \psi(t+s,z)} g(z) dV.
    \end{align*}
    To ensure $v$ is a supersolution, we require this measure to be bounded from above by $e^{\partial_t v(t,z) + \psi(t,z)} g(z) dV$. Since $\partial_t v(t,z) = \partial_t \Psi(t+s,z) + As$, this condition reduces to:
    $$
    \psi(t+s, z) \le \psi(t,z) + As.,
    $$
    and hence it suffices to choose $A\geq\kappa_\psi$.

    Next, we check the boundary conditions of $v$ on the parabolic boundary $\partial_0 \Omega_S$. On the lateral boundary $(0,S] \times \partial\Omega$, the definition of $\delta(s)$ and the fact that $\Psi = h_\Psi$ on the boundary gives:
    $$
    v(t,z) \ge \Psi(t+s, z) + \delta(s) = h_\Psi(t+s, z) + \delta(s) \ge h_\Psi(t,z).
    $$
    On the initial bottom slice $\{0\} \times \Omega$, \cref{lem:supersolution_initial_estimate} guarantees that $\Psi(s,z) \ge h_\Psi(0,z) - c(s)$, which implies:
    $$
    v(0,z) = \Psi(s,z) + c(s) + \delta(s) \ge h_\Psi(0,z) + \delta(s) \ge h_\Psi(0,z).
    $$
    Thus, $v$ is a continuous supersolution on $[0,S] \times \overline{\Omega}$ that dominates $h_\Psi$ on $\partial_0 \Omega_S$. 
    
   It then follows from \cref{lem:comparison_C1} that
   $$
    \Phi(t,z) \le v(t,z) = \Psi(t+s, z) + c(s) + \delta(s) + Ast \quad \text{in } \Omega_S.
    $$
    Finally, letting $s \to 0^+$, the continuity of $\Psi$ in the interior of the domain, along with the facts that $c(s) \to 0$ and $\delta(s) \to 0$, yields:
    $$
    \Phi(t,z) \le \Psi(t,z) \quad \text{in } \Omega_S.
    $$
    Since $S < T$ was chosen arbitrarily, the proof is complete.
\end{proof}

We are now ready to prove \cref{thm:parabolic comparison}:

\begin{proof}[Proof of \cref{thm:parabolic comparison}]
    Without loss of generality, we may focus on a slightly smaller parabolic cylinder $\Omega_S$ for any fixed $S \in (0, T)$, and aim to prove $\Phi \le \Psi$ on $\Omega_S$. Letting $S \to T$ will then yield the global result. 
    
\medskip
\noindent
    \textbf{Step 1: Time-shift.}
    For $s > 0$ sufficiently close to $1$ (say, $s \in [1/2, 3/2]$), we introduce the time-shifted function:
    $$
    W^s(t,z) := s^{-1} \Phi(st, z) - C|s-1|(t+1),
    $$
    where $C > 0$ is chosen large enough such that $W^s$ is a pluripotential subsolution, as in \cref{lem:a priori lipschitz}. Moreover, on the parabolic boundary $\partial_0 \Omega_S$, the Lipschitz property of $h_\Phi$ ensures that $W^s \le h_\Phi + O(|s-1|)$.
    
\medskip
\noindent
\textbf{Step 2: Time Convolution via Riemann Sums.}
Let $\eta_\varepsilon \ge 0$ be a standard smooth mollifier supported in $(1-\varepsilon, 1+\varepsilon)$ such that $\int_{\mathbb{R}} \eta_\varepsilon(s) ds = 1$. We define the time-mollified function:
    $$
    \Phi^\varepsilon(t,z) := \int_{\mathbb{R}} W^s(t,z) \eta_\varepsilon(s) ds.
    $$
    By construction, $\Phi^\varepsilon$ is of class $C^1$ in $t$. We claim that $\Phi^\varepsilon$ is also a pluripotential subsolution. Instead of using the subsolution theorem in \cite{GLZ19}, as was done in \cite{GLZ21a}, we prove this directly by approximating the integral with finite Riemann sums, which relies crucially on the Lipschitz and semi-concavity properties of $\Phi$. 
    
    For any integer $N$, we approximate $\Phi^\varepsilon$ by the convex combination (Riemann sum):
    $$
    \Sigma_N(t,z) := \sum_{i=1}^N c_i W^{s_i}(t,z),
    $$
    where $s_i \in (1-\varepsilon, 1+\varepsilon)$, $c_i = \eta_\varepsilon(s_i) \Delta s_i \ge 0$, and $\sum_{i=1}^N c_i = 1$. 

 Since each $W^{s_i}$ is locally uniformly semi-concave in $t$, using \cref{prop:slice_equivalence}, we have for each \(i\) and every fixed
\(t\),
\[
    (dd^c W^{s_i}(t,\cdot))^m\wedge\omega^{n-m}
    \ge
    f_i(t,z)\,g(z)dV,
    \qquad
    f_i(t,z):=\exp\bigl(\partial_t^+W^{s_i}(t,z)+\psi(t,z)\bigr).
\]
By the mixed Hessian inequality (cf. \cite{Sun25}) for bounded \((\omega,m)\)-subharmonic
functions, applied to the finite family \(W^{s_1},\ldots,W^{s_N}\), we obtain
\begin{align*}
    (dd^c \Sigma_N(t,\cdot))^m\wedge\omega^{n-m}
    &=
    \left(dd^c\sum_{i=1}^N c_i W^{s_i}(t,\cdot)\right)^m
    \wedge\omega^{n-m} \\
    &\ge
    \left(\sum_{i=1}^N c_i f_i(t,z)^{1/m}\right)^m g(z)dV \\
    &\ge
    \prod_{i=1}^N f_i(t,z)^{c_i}\,g(z)dV \\
    &=
    \exp\left(\sum_{i=1}^N c_i
    \bigl(\partial_t^+W^{s_i}(t,z)+\psi(t,z)\bigr)\right)g(z)dV \\
    &=
    \exp\bigl(\partial_t^+\Sigma_N(t,z)+\psi(t,z)\bigr)g(z)dV .
\end{align*}
Here the second inequality is the weighted arithmetic-geometric mean
inequality, and in the last line we used
\[
    \partial_t^+\Sigma_N
    =
    \sum_{i=1}^N c_i\,\partial_t^+W^{s_i},
\]
which follows from the local uniform semi-concavity in time.

We now pass to the limit as \(N\to\infty\). Fix a compact interval
\(J\Subset(0,S)\), and choose \(\varepsilon>0\) so small that
\(sJ\Subset(0,S)\) for all \(s\in\operatorname{Supp}\eta_\varepsilon\).
Set
\[
    J_\varepsilon:=\bigcup_{s\in\operatorname{Supp}\eta_\varepsilon}sJ .
\]
Since \(\Phi\) is bounded and locally uniformly Lipschitz in \(t\), there is
a constant \(C_J>0\) such that, for all
\(s,s'\in\operatorname{Supp}\eta_\varepsilon\), \(t\in J\), and \(z\in\Omega\),
\[
    |W^s(t,z)-W^{s'}(t,z)|
    \le C_J |s-s'|.
\]
Indeed, this follows directly from
\[
    W^s(t,z)=s^{-1}\Phi(st,z)-C|s-1|(t+1),
\]
the uniform boundedness of \(\Phi\) on \(J_\varepsilon\times\Omega\), and the
uniform Lipschitz bound of \(\Phi\) in the time variable on \(J_\varepsilon\).

It follows that the Riemann sums \(\Sigma_N\) converge uniformly on
\(J\times\Omega\) to
\[
    \Phi^\varepsilon(t,z)
    :=
    \int_{\mathbb R}W^s(t,z)\eta_\varepsilon(s)\,ds .
\]
In particular, for every fixed \(t\in J\), $\Sigma_N(t,\cdot)\longrightarrow \Phi^\varepsilon(t,\cdot)$
uniformly on \(\Omega\). Therefore, by the continuity of the Hermitian
Hessian operator under uniform convergence of bounded
\((\omega,m)\)-subharmonic functions, we obtain
\[
    (dd^c\Sigma_N(t,\cdot))^m\wedge\omega^{n-m}
    \longrightarrow
    (dd^c\Phi^\varepsilon(t,\cdot))^m\wedge\omega^{n-m}
\]
weakly as Radon measures on \(\Omega\), for every \(t\in J\).

    Secondly, we analyze the right-hand side. Because $\Phi$ is locally uniformly semi-concave in time, its time derivative $\partial_\tau \Phi(\tau, z)$ exists and is continuous outside a countable set of times for each fixed $z$. Therefore, the map $s \mapsto \partial_t W^s(t,z)$ is continuous almost everywhere with respect to the Lebesgue measure on $(1-\varepsilon, 1+\varepsilon)$ for each fixed $z$, making it Riemann integrable. This implies that the sequence of Riemann sums $\partial_t \Sigma_N(t,z)= \sum_{i=1}^N c_i \partial_tW^{s_i}(t,z)$ converges pointwise to the integral $\partial_t \Phi^\varepsilon(t,z)= \int_{\mathbb{R}} \partial_tW^s(t,z) \eta_\varepsilon(s) ds$ for almost all $(t,z)$. Furthermore, the locally uniform Lipschitz condition ensures that $\partial_t \Sigma_N$ is uniformly bounded independently of $N$. By the Dominated Convergence Theorem, the right-hand side converges in the weak sense of measures:
    $$
    \exp( \partial_t \Sigma_N(t,z) + \psi(t,z) ) g(z) dV \to e^{\partial_t \Phi^\varepsilon(t,z) + \psi(t,z)} g(z) dV.
    $$
    Combining both limits, the inequality is preserved, yielding:
    $$
    (dd^c \Phi^\varepsilon(t,\cdot))^m \wedge \omega^{n-m} \ge e^{\partial_t \Phi^\varepsilon(t,z) + \psi(t,z)} g(z) dV
    $$
    in the sense of measures. Thus, $\Phi^\varepsilon$ is indeed a pluripotential subsolution on $\Omega_S$.
    
\medskip
\noindent
    \textbf{Step 3: Conclusion.}
    Since $\Phi^\varepsilon$ is an average of $W^s$, and $W^s \le h_\Phi + O(|s-1|)$ on the parabolic boundary, we have:
    $$
    \Phi^\varepsilon \le h_\Phi + A \varepsilon \le h_\Psi + A \varepsilon \quad \text{on } \partial_0 \Omega_S
    $$
    for some uniform constant $A > 0$. We can thus define a slightly shifted function:
    $$
    \tilde{\Phi}^\varepsilon(t,z) := \Phi^\varepsilon(t,z) - A \varepsilon (t+1),
    $$
    which remains a $C^1$-in-time subsolution and satisfies $\tilde{\Phi}^\varepsilon \le h_\Psi$ on $\partial_0 \Omega_S$. 

    We are now exactly in the setting to apply the interior comparison principle. Applying \cref{lem:comparison_C1_interior} to the pair $(\tilde{\Phi}^\varepsilon, \Psi)$, we deduce that:
    $$
    \tilde{\Phi}^\varepsilon(t,z) \le \Psi(t,z) \quad \text{in } \Omega_S.
    $$
    Finally, letting $\varepsilon \to 0$, the standard properties of the mollifier ensure that $\tilde{\Phi}^\varepsilon$ converges pointwise to $\Phi(t,z)$. This yields $\Phi \le \Psi$ in $\Omega_S$, concluding the proof.
\end{proof}

In our previous derivation, we employed a Riemann sum approximation to circumvent the reliance on a specific characterization of weak subsolutions, a challenge encountered in the Monge-Amp\`ere setting \cite{GLZ19}. Nonetheless, such a characterization for Hessian subsolutions is also interesting on its own. We formally state it as follows:

\begin{question}\label{question}
Let $(\Omega, \omega)$ be a domain in $\mathbb{C}^n$ equipped with a Hermitian metric $\omega$. Assume $\varphi \in \operatorname{SH}_m(\Omega,\omega) \cap L^\infty(\Omega)$ and $0 \leq f \in L^1(\Omega)$. The following are equivalent:

\begin{enumerate}[(i)]
    \item $(dd^c \varphi)^m \wedge \omega^{n-m} \geq f \omega^n$ in the pluripotential sense;
    
    \item For each $(\omega, m)$-positive $(1,1)$-form $\eta$ satisfying the pointwise normalization condition
    \[
  \eta^m \wedge \omega^{n-m} = \omega^n,
    \]
    the following inequality holds in the sense of distributions:
    \[
    dd^c \varphi \wedge\eta^{m-1} \wedge \omega^{n-m} \geq f^{1/m} \omega^n;
    \]

    \item If $f$ is continuous, $(dd^c \varphi)^m \wedge \omega^{n-m} \geq f \omega^n$ in the viscosity sense.
\end{enumerate}
\end{question}

The implication (i) $\Longrightarrow$ (ii) is a direct consequence of the mixed Hessian inequality on Hermitian manifolds recently established in \cite{Sun25}. The converse, however, is substantially more delicate. While B\l ocki \cite{Bło05} proved the equivalence for smooth $\varphi$ and $f$ using linear algebra, and the flat case ($\omega = dd^c|z|^2$) can be resolved via standard mollification as in \cite{GLZ19}, the problem remains open for arbitrary Hermitian metrics. The primary obstruction lies in the fact that convolution does not commute with the operator when the background metric has non-vanishing torsions.  

\section{Existence of Global Pluripotential Solutions}

Having established the solvability on Euclidean balls and the global comparison principle, we are now in a position to solve the Cauchy-Dirichlet problem on general bounded strictly $m$-pseudoconvex domains with smooth boundaries. We use the balayage process and approximation techniques, adapting the arguments from \cite[Proposition 6.4]{GLZ21a}. 

We shall need the following continuous approximation device. When
\(\omega=dd^c|z|^2\), this is a standard consequence of Walsh's lemma \cite{Wal69}.
For a general Hermitian metric, however, translations change the background
form and hence do not preserve \((\omega,m)\)-subharmonicity. We avoid this by viewing the Hermitian Hessian condition as a
variable-coefficient convex subequation and applying the corresponding
obstacle theorem of Harvey--Lawson--Pli\'s \cite{HLP16}.

\begin{lemma}\label{lem:continuous_obstacle_envelope}
Let $\Omega\Subset \mathbb C^n$ be a bounded smooth strictly
$m$-pseudoconvex domain and let $\omega$ be a smooth Hermitian metric
on a neighborhood of $\overline{\Omega}$. For every $f\in C^0(\overline{\Omega})$,
the envelope
\[
P_\omega(f):=\left(\sup\{v\in \operatorname{SH}_m(\Omega,\omega)\cap L^\infty(\Omega):
v^*\le f \ \text{on }\overline{\Omega}\}\right)^*
\]
belongs to $\operatorname{SH}_m(\Omega,\omega)\cap C^0(\overline{\Omega})$.
Moreover,
\[
\lim_{\Omega\ni z\to \zeta}P_\omega(f)(z)=f(\zeta),
\qquad \zeta\in\partial\Omega .
\]
In particular, if $f=\phi$ on $\partial\Omega$, then the boundary value is
$\phi$.
\end{lemma}

\begin{proof}
We refer to \cite{HL11,HL13,HLP16} for concepts and terminology employed in the proof. Choose an open neighborhood \(X\Subset\mathbb C^n\) of \(\overline{\Omega}\)
on which \(\omega\) is defined. By Tietze's extension theorem, after shrinking
\(X\) if necessary, we may extend \(f\) to a continuous function, still denoted
by \(f\), on \(X\).

Let \(F_\omega\subset J^2(X)\) be the reduced subequation defined by
\[
J^2_z q\in (F_\omega)_z
\quad\Longleftrightarrow\quad
dd^cq(z)\in \overline{\Gamma}_m(\omega(z)).
\]
Equivalently, if \(\lambda_\omega(dd^cq(z))\) denotes the eigenvalues of
\(dd^cq(z)\) with respect to \(\omega(z)\), then
\[
J^2_z q\in (F_\omega)_z
\quad\Longleftrightarrow\quad
\lambda_\omega(dd^cq(z))\in\overline{\Gamma}_m .
\]
The value and gradient components of the jet play no role, so this is of the
form considered in the obstacle theorem for reduced subequations.

We verify the hypotheses of the Harvey--Lawson--Pli\'s obstacle theorem \cite[Theorem 2.1]{HLP16}.
First, since \(\overline{\Gamma}_m\) is a closed convex cone and is stable
under addition of non-negative Hermitian forms, \(F_\omega\) is a closed
convex degenerate elliptic subequation. Moreover,
\[
F_\omega+F_\omega\subset F_\omega,
\]
so \(F_\omega\) itself is a monotonicity cone for \(F_\omega\).

Next, \(F_\omega\) is locally affinely jet-equivalent to the constant
coefficient complex Hessian subequation. Indeed, in a local holomorphic
coordinate chart write
\[
\omega(z)=\sqrt{-1}\,g_{j\bar k}(z)\,dz^j\wedge d\bar z^k
\]
and put \(S(z):=g(z)^{-1/2}\). If \(H\) is the Hermitian matrix representing
the \((1,1)\)-part of a real Hessian, define
\[
T_z(H):=S(z)H S(z)^* .
\]
Then
\[
H\in \overline{\Gamma}_m(\omega(z))
\quad\Longleftrightarrow\quad
T_z(H)\in \overline{\Gamma}_m(\omega_{\mathrm{eucl}}).
\]
Extending this transformation by the identity on the gradient part and on the
real Hessian directions which do not enter \(dd^c\), we obtain a smooth fiber
automorphism of \(J^2_{\mathrm{red}}(X)\). This gives the required local
affine jet-equivalence to the constant coefficient complex Hessian
subequation.

The strict monotonicity hypothesis is also satisfied. Since \(dd^c|z|^2\) is
positive definite, its eigenvalues with respect to the Hermitian metric
\(\omega(z)\) lie in \(\Gamma_n\subset\Gamma_m\). Hence \(|z|^2\) is strictly
\(F_\omega\)-subharmonic on a neighborhood of \(\overline{\Omega}\).

It remains to check the boundary convexity. Since \(\Omega\) is smoothly
strictly \(m\)-pseudoconvex, there exists a smooth defining function
\(\rho\) on a neighborhood of \(\overline{\Omega}\) such that
\[
\Omega=\{\rho<0\},\qquad d\rho\ne0 \text{ on }\partial\Omega,
\qquad
dd^c\rho\in\Gamma_m(\omega)
\]
near \(\partial\Omega\). Thus \(\rho\) is strictly \(F_\omega\)-subharmonic
near \(\partial\Omega\), and \(\partial\Omega\) is strictly
\(F_\omega\)-convex.

Furthermore, since \(F_\omega+F_\omega\subset F_\omega\), we have
\[
\operatorname{Int}F_\omega\subset \operatorname{Int}\widetilde F_\omega .
\]
Indeed, if \(J\in\operatorname{Int}F_\omega\) and \(-J\in F_\omega\), then by
convexity one would get \(0\in\operatorname{Int}F_\omega\), which is
impossible for the Hessian cone, since \(0\in\partial F_\omega\). Hence
\(-J\notin F_\omega\), which is equivalent to
\(J\in\operatorname{Int}\widetilde F_\omega\). Therefore the same defining
function \(\rho\) is also strictly \(\widetilde F_\omega\)-subharmonic near
\(\partial\Omega\), and \(\partial\Omega\) is strictly
\(\widetilde F_\omega\)-convex.

We may now apply the obstacle theorem for locally affinely jet-equivalent
convex subequations \cite[Theorems~2.1 and~2.5]{HLP16}. It gives that the
largest \(F_\omega\)-subharmonic minorant of the continuous obstacle \(f\),
namely
\[
H_f(z):=
\sup\{w(z): w\in F_\omega(\Omega),\ w\le f \text{ in }\Omega\},
\]
is continuous on \(\overline{\Omega}\), belongs to \(F_\omega(\Omega)\),
satisfies
\[
H_f\le f \quad\text{in }\Omega,
\]
and assumes the boundary value
\[
H_f=f \quad\text{on }\partial\Omega .
\]

We next identify \(F_\omega(\Omega)\) with
\(\operatorname{SH}_m(\Omega,\omega)\). For smooth functions this is exactly
the definition of the complex Hessian cone. For upper semicontinuous
functions, the viscosity and distributional notions agree for convex
subequations by the strong Bellman principle
\cite[Theorem~1.1]{HL13}. Thus
\[
F_\omega(\Omega)=\operatorname{SH}_m(\Omega,\omega)
\]
after taking the usual upper semi-continuous representative.

It remains to compare \(H_f\) with the envelope \(P_\omega(f)\). Since
\(H_f\in C^0(\overline{\Omega})\cap\operatorname{SH}_m(\Omega,\omega)\),
\(H_f\le f\) in \(\Omega\), and \(H_f=f\) on \(\partial\Omega\), the function
\(H_f\) is admissible in the definition of \(P_\omega(f)\). Hence
\[
P_\omega(f)\ge H_f .
\]
Conversely, if \(v\in\operatorname{SH}_m(\Omega,\omega)\cap L^\infty(\Omega)\)
satisfies \(v^*\le f\) on \(\overline{\Omega}\), then in particular
\(v\le f\) in \(\Omega\). Hence \(v\) is a competitor for the obstacle
envelope \(H_f\), and therefore \(v\le H_f\). Taking the supremum over all
such \(v\), and then upper semi-continuous regularization, gives
\[
P_\omega(f)\le H_f,
\]
because \(H_f\) is continuous. Thus
\[
P_\omega(f)=H_f .
\]
Consequently
\[
P_\omega(f)\in \operatorname{SH}_m(\Omega,\omega)\cap C^0(\overline{\Omega}).
\]
Finally, since \(H_f=f\) on \(\partial\Omega\), we obtain
\[
\lim_{\Omega\ni z\to\zeta}P_\omega(f)(z)
=
\lim_{\Omega\ni z\to\zeta}H_f(z)
=
f(\zeta),
\qquad \zeta\in\partial\Omega .
\]
This proves the lemma.
\end{proof}

\begin{lemma}\label{lem:continuous_approximation}
    Let $u \in \operatorname{SH}_m(\Omega, \omega)$ be a bounded function such that $\lim_{z \to \zeta} u(z) = \phi(\zeta)$ for all $\zeta \in \partial \Omega$, where $\phi \in C^0(\partial \Omega)$. Then there exists a decreasing sequence $\{u_j\}$ of continuous $(\omega, m)$-subharmonic functions on $\overline{\Omega}$ such that $u_j = \phi$ on $\partial \Omega$ and $u_j \searrow u$ pointwise in $\Omega$.
\end{lemma}

\begin{proof}
    By the resolution of the homogeneous Dirichlet problem for the complex Hessian equation (cf. \cite{KN26}), there exists a unique continuous $(\omega, m)$-subharmonic function $v \in \operatorname{SH}_m(\Omega, \omega) \cap C^0(\overline{\Omega})$ such that 
    $$
    (dd^c v)^m \wedge \omega^{n-m} = 0 \quad \text{in } \Omega,
    $$
    with boundary values $v = \phi$ on $\partial \Omega$. Since $u$ shares the same boundary values $\phi$ and satisfies $(dd^c u)^m \wedge \omega^{n-m} \ge 0$, the standard comparison principle yields $u \le v$ in $\Omega$.

    We now take a sequence of continuous functions $f_j \in C^0(\overline{\Omega})$ such that $f_j \searrow u$ pointwise in $\overline{\Omega}$. By replacing $f_j$ with $\min(f_j, v)$, we may assume without loss of generality that $f_j = \phi$ on $\partial \Omega$. For each $j$, we define the $(\omega, m)$-subharmonic envelope:
    $$
    u_j := \sup \left\{ w \in \operatorname{SH}_m(\Omega, \omega) \mid w^* \le f_j \text{ in } \overline{\Omega} \right\}.
    $$
    Since $u \le f_j$, we clearly have $u \le u_j \le f_j$. In particular, this forces $\lim_{z \to \zeta} u_j(z) = \phi(\zeta)$ for all boundary points $\zeta \in \partial \Omega$. By \cref{lem:continuous_obstacle_envelope}, the envelope $u_j$ is continuous on $\overline{\Omega}$. Since the sequence $\{f_j\}$ is decreasing, $\{u_j\}$ is also a decreasing sequence of continuous $(\omega,m)$-subharmonic functions converging pointwise to $u$.
\end{proof}
Now we improve \cref{continuous boundary data solution on balls} to the case where $h$ is only locally uniformly Lipschitz and locally uniformly semi-concave:
\begin{proposition}\label{continuous boundary data solution on balls 2}
    Let $h\in C^0(\partial_0\mathbb{B}_T)$ be a continuous Cauchy-Dirichlet boundary data with $T<+\infty$. Assume moreover that there is a uniform constant $\kappa_h$ such that
    \begin{equation}\label{h Lip condition}
t|\partial_th|\leq\kappa_h,\,t^2\partial_t^2h\leq C_h,\quad t\in(0,T]
    \end{equation}
    in the sense of distributions. In other words, $h$ is locally uniformly Lipschitz and locally uniformly semi-concave on $[0,T]$. Furthermore, $\psi(t,z)$ is a bounded continuous function on $\overline{\mathbb{B}_T}$ such that $\psi$ is uniformly Lipschitz and uniformly semi-convex in $t$, and $g\in C^0(\overline{\mathbb{B}})$ is a continuous positive density. Then, there is a continuous parabolic $m$-potential $u\in P_{m,\omega}(\mathbb{B}_T)\cap C^0(\overline{\mathbb{B}}_T)$ such that
    $$
dt\wedge(dd^cu)^m\wedge\omega^{n-m}=e^{\partial_t u+\psi(t,z)}g(z)dt\wedge dV_{\mathbb{B}},
    $$
    on $\mathbb{B}_T$ in the weak sense of Radon measures. Moreover, the solution $u$ is nothing but the envelope $U=U_{h,\psi,g,T}$.
\end{proposition}
\begin{proof}
Fix a small $\varepsilon > 0$ and $S<T$. We define the boundary data $h^\varepsilon$ on the lateral side of the sub-cylinder $\mathbb{B} \times [0, S]$ exactly as in \cref{lem:initial_uniform_convergence}.
    
It is clear that $h^\varepsilon$ is uniformly Lipschitz and semi-concave on $[0,S]$, hence we have a solution $u^\varepsilon$ of the Hessian flow with boundary data $h^\varepsilon$. The parabolic comparison principle \cref{thm:parabolic comparison} yields that $u^\varepsilon=U^\varepsilon:=U_{h^\varepsilon,\psi,g,S}$. By the identity principle \cref{prop:identity_principle}, we get $U^\varepsilon\to U:=U_{h,\psi,g,S}$ uniformly on $\mathbb{B}_S$ as $\varepsilon\to0$. Since $h_0$ is continuous, we also have that $U$ is continuous on $\overline{\mathbb{B}_S}$ by \cref{thm:boundary_behavior_envelope} and \cref{lem:initial_uniform_convergence}.
    
Moreover, for any $t \in (0, T-\varepsilon]$, since $\frac{t}{t+\varepsilon} < 1$, we compute:
    $$
    t |\partial_t h^\varepsilon(t, \zeta)| = \frac{t}{t+\varepsilon} \cdot (t+\varepsilon) |\partial_t h(t+\varepsilon, \zeta)| \le 1 \cdot \kappa_h = \kappa_h.
    $$
    Similarly, the semi-concavity bound is uniformly preserved:
    $$
    t^2 \partial_t^2 h^\varepsilon(t, \zeta) = \frac{t^2}{(t+\varepsilon)^2} \cdot (t+\varepsilon)^2 \partial_t^2 h(t+\varepsilon, \zeta) \le 1 \cdot C_h = C_h.
    $$
    Since the constants $\kappa_h$ and $C_h$ are preserved independently of $\varepsilon$, the sequence of continuous solutions $\{u^\varepsilon\}$ satisfies uniform local a priori estimates on any compact time subinterval. By \cref{key convergence}, we can argue as in Step 2 of \cref{continuous boundary data solution on balls} to get the convergence of measures on the right-hand side. The proof is thus finished.
\end{proof}
As a byproduct, we automatically obtain the continuity of the envelope $U$ with continuous data, which was dealt with in the Monge-Amp\`ere case independently in \cite[Theorem 5.1]{GLZ21a}. Now, we can solve the parabolic Hessian flow on an arbitrary $m$-pseudoconvex domain:
\begin{theorem}\label{thm:global_existence}
    Let $\Omega \subset \mathbb{C}^n$ be a bounded, strictly $m$-pseudoconvex domain with smooth boundary. Assume $T<+\infty$, and the boundary data $h$ satisfies the locally uniformly Lipschitz and semi-concave conditions:
    $$
    t|\partial_th|,\,t^2\partial_t^2h\leq C,\quad \forall t\in(0,T].
    $$
    Then the parabolic envelope $U := U_{h,g,\psi,\Omega_T}$ is a pluripotential solution to the Cauchy-Dirichlet problem for the parabolic complex Hessian equation in $\Omega_T$ with boundary data $h$.
\end{theorem}

\begin{proof}
    By our previous a priori estimates and \cref{thm:envelope_semiconcave 2}, $U$ is locally uniformly semi-concave in $t \in (0,T)$, $U \in \mathcal{S}_{h,g,\psi}(\Omega_T)$, and it assumes the correct boundary values on $\partial_0 \Omega_T$. It remains to verify that $U$ actually solves the parabolic complex Hessian equation. We proceed in three steps.
    
\medskip
\noindent
    \textbf{Step 1: Continuous data.} 
    We first assume that $h_0$ and the density $g$ are continuous on $\overline{\Omega}$. In this case, the envelope $U$ is continuous on $[0,T) \times \overline{\Omega}$. 
    
    Let $B \Subset \Omega$ be a small Euclidean ball. The restriction of $U$ to the parabolic boundary $\partial_0 B_T$ provides a continuous boundary data $h_B$. Clearly, the restriction of $U$ on $B_T$ is a continuous subsolution to the Cauchy-Dirichlet problem in $B_T$ with data $h_B$. By \cref{continuous boundary data solution on balls}, there exists a continuous pluripotential solution $U_B$ in $B_T$ such that $U_B = h_B$ on $\partial_0 B_T$. Moreover, we have $U_B \ge U$ in $B_T$ by the comparison principle \cref{thm:parabolic comparison}. 
    
    We define the glued function $V$ by setting $V = U_B$ in $B_T$ and $V = U$ in $\Omega_T \setminus B_T$. By the local nature of pluripotential subsolutions, it is easy to see that $V \in \mathcal{S}_{h,g,\psi}(\Omega_T)$. Therefore, by the definition of the envelope, $V \le U$. Since $V \ge U$ by construction, we conclude $V = U$. This implies that $U = U_B$ in $B_T$, meaning $U$ solves the equation in $B_T$. Since the ball $B \Subset \Omega$ is arbitrary, $U$ solves the equation in the whole $\Omega_T$.
    
\medskip
\noindent
    \textbf{Step 2: $L^p$ density approximation.} 
    We next assume $h_0$ is continuous, but $g \in L^p(\Omega)$ with $p>\frac{n}{m}$. 
    Let $(g_j)$ be a sequence of strictly positive continuous functions on $\overline{\Omega}$ converging to $g$ in $L^p(\Omega)$. Let $U^j := U_{h,g_j,\psi,\Omega_T}$ be the corresponding envelopes, which are uniformly bounded thanks to \cref{lem:admissible_set_properties}. Note that the uniform boundedness of $\rho_j$ corresponding to $g_j$ can be derived easily by the comparison principle and the stability estimate in \cite[Corollary 6.2, Proposition 8.1]{KN26}.  Moreover, \cref{thm:envelope_lipschitz} and \cref{thm:envelope_semiconcave 2} guarantee that the sequence $\{U^j\}$ is locally uniformly semi-concave and Lipschitz in time with uniform constants.
    
    By \cref{weak cptness}, up to extracting a subsequence, $U^j$ converges almost everywhere to a limit $V \in P_m(\Omega_T)\cap L^\infty(\Omega_T)$. By \cref{key convergence}, $\partial_t U^j \to \partial_t V$ pointwise almost everywhere and we have the weak convergence
    $$
e^{\partial_tU^j+\psi}g_jdV\to e^{\partial_tV+\psi}gdV
    $$
Observe also that all the $U^j,V$ shares the same boundary data, the stability of complex Hessian equations (see \cite[Proposition 8.1]{KN26}) ensures that $(dd^c U^j_t)^m \wedge \omega^{n-m}$ converges weakly to $(dd^c V_t)^m \wedge \omega^{n-m}$. Hence, passing to the limit, $V$ solves the equation with density $g$ in $\Omega_T$. Also, \cref{lem:lateral_boundary_upper_bound} and \cref{cor:boundary_limit_subsolutions} yield that $V^*|_{\partial_0 \Omega_T} \le h$, which means $V \in \mathcal{S}_{h,g,\psi}(\Omega_T)$ and therefore $V \le U$.

    To prove the reverse inequality $U \le V$, we use a perturbation argument, following \cite{Ko96, GLZ18, GLZ21a}. For each $j$, let $\theta_j \in \operatorname{SH}_m(\Omega, \omega) \cap C^0(\overline{\Omega})$ be the unique solution to $(dd^c \theta_j)^m \wedge \omega^{n-m} = |g_j - g|dV$ with $\theta_j = 0$ on $\partial \Omega$. The stability theorem again gives $\lim_{j \to \infty} \sup_{\overline{\Omega}} |\theta_j| = 0$.
    Fix $0 < S < T$ and a small $\varepsilon > 0$, and define on $\Omega_S$:
    $$
    W^j(t,z) := U(t+\varepsilon, z) - \kappa_\psi \varepsilon t + C(\varepsilon) \theta_j(z),
    $$
    where $\kappa_\psi$ is the Lipschitz constant of $\psi$ in time, and $C(\varepsilon) > 0$ will be chosen shortly. We verify that $W^j \in \mathcal{S}_{h,g_j,\psi}(\Omega_S)$. Evaluating its Hessian measure, we have:
    \begin{align*}
        (dd^c W^j)^m \wedge \omega^{n-m} &\ge (dd^c U(t+\varepsilon, \cdot))^m \wedge \omega^{n-m} + C(\varepsilon)^m (dd^c \theta_j)^m \wedge \omega^{n-m} \\
        &\ge e^{\partial_t U(t+\varepsilon, z) + \psi(t+\varepsilon, z)} g dV + C(\varepsilon)^m |g - g_j| dV.
    \end{align*}
    By the Lipschitz property of $\psi$, we have $\psi(t+\varepsilon, z) \ge \psi(t, z) - \kappa_\psi \varepsilon$. Since $\partial_t W^j(t,z) = \partial_t U(t+\varepsilon, z) - \kappa_\psi \varepsilon$, it follows that:
    $$
    \partial_t U(t+\varepsilon, z) + \psi(t+\varepsilon, z) \ge \partial_t W^j(t,z) + \psi(t,z).
    $$
    Thus, the measure is bounded below by:
    $$
    e^{\partial_t W^j + \psi} g dV + C(\varepsilon)^m |g - g_j| dV = e^{\partial_t W^j + \psi} g_j dV + \left( C(\varepsilon)^m - e^{\partial_t W^j + \psi} \right) |g - g_j| dV.
    $$
    We set $C(\varepsilon) := (\sup_{\Omega_S} \exp(\partial_t U(t+\varepsilon, z) + \psi(t,z)))^{1/m} < +\infty$, which ensures the second term is non-negative. Consequently, $(dd^c W^j)^m \wedge \omega^{n-m} \ge e^{\partial_t W^j + \psi} g_j dV$. Since $\theta_j \le 0$ by the comparison principle, modifying $W^j$ slightly to account for the boundary error (which goes to zero as $\varepsilon \to 0$) shows $W^j \leq U^j+b(\varepsilon)$ for some $b(\varepsilon)$ such that $b(\varepsilon)\to0$ as $\varepsilon\to0$. Letting $j \to \infty$ and then $\varepsilon \to 0$, we recover $U \le V$. Hence $U = V$ solves the equation.
    
\medskip
\noindent
    \textbf{Step 3: Removing the continuity assumption on $h_0$.}
  We finally treat the general case where the initial data $h_0$ is merely bounded and $(\omega,m)$-subharmonic with continuous boundary values. By \cref{lem:continuous_approximation}, we can find a decreasing sequence of continuous $(\omega,m)$-subharmonic functions $h_0^j \in C^0(\overline{\Omega})$ converging pointwise to $h_0$ on $\Omega$, such that $h_0^j = h(0, \cdot)$ on $\partial \Omega$. We then define a sequence of continuous Cauchy-Dirichlet data $h^j$ by setting $h^j = h$ on the lateral boundary $[0,T) \times \partial\Omega$ and $h^j(0,\cdot) = h_0^j$ on $\Omega$. 
    
    By the result of Step 2, the parabolic envelopes $U^j := U_{h^j, g, \psi, \Omega_T}$ are continuous pluripotential solutions to the complex Hessian flow equation with boundary data $h^j$. Since the boundary data $h^j$ is decreasing, the sequence of solutions $U^j$ is also decreasing. Moreover, since $U \le U^j$ for all $j$, the sequence $U^j$ decreases pointwise to a limit $V \in P_m(\Omega_T)$ with $V\geq U$. 
    
Now, \cref{thm:envelope_semiconcave 2} gives that $U^j$ is locally uniformly semi-concave in time. By \cref{key convergence}, we have $\partial_t U^j \to \partial_t V$ almost everywhere, and the Hessian measures $(dd^c U^j_t)^m \wedge \omega^{n-m}$ converge weakly to $(dd^c V_t)^m \wedge \omega^{n-m}$. This guarantees that the limit $V$ is a pluripotential solution to the complex Hessian flow. It is then easy to see that $V$ belongs to the admissible class $\mathcal{S}_{h,g,\psi}(\Omega_T)$, and therefore $V \le U$ by the definition of the Perron envelope, thus $U = V$, confirming that the envelope $U$ is indeed the unique pluripotential solution to the Cauchy-Dirichlet problem.
\end{proof}

We are now ready to prove the main existence result \cref{thm:main_pluripotential}. Here $T$ may take the value $+\infty$. We assume that, for each $0<S<T$, there exists a constant $C(S)>0$ such that for all $(t,z)\in]0,S]\times\partial\Omega$,
\begin{equation}\label{eq:local_h_bound}
t|\partial_th(t,z)|\le C(S), \quad t^2\partial_t^2h(t,z)\le C(S).
\end{equation}

\begin{theorem}\label{thm:final_existence}
    If $h$ satisfies \eqref{eq:local_h_bound}, then the parabolic envelope $U:=U_{h,g,\psi,\Omega_T}$ is a pluripotential solution to the Cauchy-Dirichlet problem for the parabolic complex Hessian equation in $\Omega_T$ with boundary condition $h$.

    Moreover, $U$ is continuous in $(0,T) \times \overline{\Omega}$ and locally uniformly semi-concave in $t \in (0,T)$. In particular, if $h_0$ is continuous on $\overline{\Omega}$, then $U$ is continuous on $[0,T) \times \overline{\Omega}$.
\end{theorem}

\begin{proof}
    For any $S \in (0,T)$, we define $U^S := U_{h,g,\psi,\Omega_S}$. Since $S < +\infty$, our previous result ensures that $U^S$ solves the complex Hessian flow equation with $U^S = h$ on $\partial_0 \Omega_S$. 

    It follows from the identity principle \cref{prop:identity_principle} that, for $0 < S_1 < S_2 < T$, we have $U^{S_1} = U^{S_2}$ in $\Omega_{S_1}$. Letting $S \to T$, we obtain a function $V \in P_m(\Omega_T)$ which solves the parabolic complex Hessian equation in $\Omega_T$ and satisfies $V = h$ on $\partial_0 \Omega_T$.

    Obviously, $U \le U^S$ for all $S \in (0,T)$, which implies $U \le V$. On the other hand, since $V$ is a pluripotential solution taking the boundary data $h$, $V$ is a valid candidate defining the envelope $U$ (i.e., $V \in \mathcal{S}_{h,g,\psi}(\Omega_T)$), hence $V \le U$. Therefore, $V = U$ solves the equation in $\Omega_T$.

    Moreover, by our a priori estimates \cref{thm:envelope_lipschitz} and \cref{thm:envelope_semiconcave 2}, $U^S$ is locally uniformly Lipschitz and semi-concave in $t \in (0,S)$, hence so is $U$ in $(0,T)$. 

    It then follows from \cref{prop:slice_equivalence} that
    $$
    (dd^c U_t)^m \wedge \omega^{n-m} = e^{\partial_t U_t + \psi(t,\cdot)} g dV
    $$
    for almost every $t \in (0,T)$. Since $\partial_t U$ is locally bounded and $\psi(t,\cdot)$ is continuous, the right-hand side belongs to $L^p(\Omega)$ for $p>\frac{n}{m}$. The regularity theorem \cite[Theorem 8.2]{KN26} then ensures that $U_t$ is continuous on $\overline{\Omega}$ for almost all $t \in (0,T)$. Since $U$ is locally uniformly Lipschitz in $t$, we infer that $U$ is continuous in $(0,T) \times \overline{\Omega}$.

    Finally, if the initial data $h_0$ is continuous on $\overline{\Omega}$, the boundary behavior of the envelope ensures that $U(t,\cdot)$ continuously approaches $h_0$ as $t \to 0^+$. Combined with the interior continuity, this shows that $U$ is globally continuous on $[0,T) \times \overline{\Omega}$.
\end{proof}

\begin{remark}
    The pluripotential part of this paper is developed for bounded strictly
\(m\)-pseudoconvex domains in \(\mathbb{C}^n\). A natural next step is to
extend the theory to Hermitian manifolds with boundary, in parallel with the
smooth Cauchy--Dirichlet theory proved in the first part of the paper and
with the elliptic Hessian theory of \cite{GN18,KN26}. It would also be
interesting to develop a corresponding weak theory for complex Hessian flows
on closed K\"ahler or Hermitian manifolds, in parallel with the Monge--Amp\`ere analog established by Guedj--Lu--Zeriahi \cite{GLZ20} in the K\"ahler setting and by Dang \cite{Dang24} in the Hermitian setting. 
Such an extension would provide a pluripotential framework for studying
geometric Hessian flows, including those related to the Hessian-cscK equations
of Guo--Smith--Tong \cite{GST23}.
\end{remark}

\end{document}